\documentclass[10pt]{article}
\usepackage [dvips]{graphics}
\usepackage[centertags]{amsmath}
\usepackage{amsfonts}
\usepackage{amssymb}
\usepackage{amsthm}
\usepackage{amsmath,amscd}
\usepackage{stmaryrd}
\usepackage[]{titletoc}
\usepackage{lipsum}
\usepackage{soul}
\usepackage{mathrsfs}
\usepackage{xypic}
\usepackage{extarrows}
\usepackage[margin=1in]{geometry}
\usepackage{multirow}
\usepackage{array}
\usepackage{tocloft}

\theoremstyle{plain}
\newtheorem{thm}{Theorem}[section]
\newtheorem{defn}[thm]{Definition}
\newtheorem{cor}[thm]{Corollary}
\newtheorem{lem}[thm]{Lemma}
\newtheorem{prop}[thm]{Proposition}

\newtheorem{rem}[thm]{Remark}
\newtheorem{notation}[thm]{Notation}

\newcommand{\bn}{\mathbb{B}_n}
\newcommand{\cn}{\mathbb{C}^n}
\newcommand{\dd}{\mathbb{D}}
\newcommand{\tori}{\mathbb{T}}
\newcommand{\bert}{L_{a,t}^{2}(\mathbb{B}_n)}
\newcommand{\Hol}{\mathrm{Hol}}
\newcommand{\sn}{\mathbb{S}_n}
\newcommand{\intd}{\mathrm{d}}
\newcommand{\la}{\langle}
\newcommand{\ra}{\rangle}
\newcommand{\bpartial}{\bar{\partial}}
\newcommand{\bz}{\bar{z}}
\newcommand{\bw}{\bar{w}}

\newcommand{\BTt}{T^{(t)}}
\newcommand{\BTtp}{T^{(t+1)}}
\newcommand{\BTtpt}{T^{(t+1,t)}}

\newcommand{\BHt}{H^{(t)}}
\newcommand{\BKt}{K^{(t)}}
\newcommand{\BKtp}{K^{(t+1)}}

\newcommand{\BFt}{\mathcal{F}^{(t)}}
\newcommand{\BGt}{\mathcal{G}^{(t)}}
\newcommand{\BPt}{P^{(t)}}
\newcommand{\Tr}{\mathrm{Tr}}
\newcommand{\sgn}{\mathrm{sgn}}
\newcommand{\ind}{\mathbb{N}_0^n}
\newcommand{\odd}{\mathrm{odd}}
\newcommand{\even}{\mathrm{even}}

\allowdisplaybreaks[4]

\makeatletter\@addtoreset{equation}{section} \makeatother
\title {Helton-Howe Trace, Connes-Chern Character and Quantization}
\author{Xiang Tang\thanks{Department of Mathematics and Statistics, Washington University, St. Louis, MO, U.S.A., 63130, xtang@math.wustl.edu.}, Yi Wang
\thanks{Department of Mathematics, Chongqing University, Chongqing,  China, 400044,  wang\_yi@cqu.edu.cn}, and Dechao Zheng\thanks{Department of Mathematics, Vanderbilt University, Nashville, TN, U.S.A., 37240, dechao.zheng@vanderbilt.edu.} }
\date{}

\begin{document}
\maketitle
	
	\begin{abstract} 
	We study the Helton-Howe trace and the Connes-Chern character for Toeplitz operators on weighted Bergman spaces via the idea of quantization. We prove a local formula for the large $t$-limit of the Connes-Chern character as the weight goes to infinity. And we show that the Helton-Howe trace of Toeplitz operators is independent of the weight $t$ and obtain a local formula for the Helton-Howe trace for all weighted Bergman spaces using harmonic analysis and quantization. \\

	\end{abstract}

	
	\section{Introduction}\label{sec: introduction}
	
	Toeplitz extensions are fundamental objects in noncommutative geometry. They are natural examples of finite summable Fredholm modules and define elements in the corresponding $K$-homology group. Trace on Toeplitz operators has been well studied with many interesting results, c.f.  \cite{Zhu:bookoperator}. Since the 70s, trace has been employed to extract geometric information of Toeplitz extension. In particular, Connes \cite[Sec. 2., Theorem 5]{Co:noncommutative} used trace on products of semi-commutators to define the Connes-Chern character of finite summable extensions.  
	
	For the unit disk $\dd$ in $\mathbb{C}$, let $L^2_a(\dd)$ be the Bergman space of $L^2$ analytic functions on $\dd$. Given $f\in \mathscr{C}^\infty(\overline{\dd})$, let $T^{(0)}_f$ be the Toeplitz operator on $L^2_a(\dd)$ associated to the symbol $f$.  The commutator $[T^{(0)}_f, T^{(0)}_g]:=T^{(0)}_fT^{(0)}_g-T^{(0)}_gT^{(0)}_f$ for $f,g\in \mathscr{C}^\infty(\overline{\dd})$ is a trace class operator. Helton and Howe \cite{HH1} discovered an interesting formula for the commutator
	\[
	\Tr\big([T^{(0)}_f, T^{(0)}_g]\big)=\frac{1}{2\pi \sqrt{-1}}\int_{\mathbb{D}} \intd f\wedge \intd g. 
	\]
	The above result is deeply connected to the Pincus function for a pair of noncommuting selfadjoint operators, c.f. \cite{Ca-Pi:exponential, Ca-Pi:mosaics, Pi:commutators}. 
	
	 For $\bn$, the commutator $[T^{(0)}_f, T^{(0)}_g]$ for two Toeplitz operators with smooth symbols $f,g$ on $L^2_a(\bn)$ is a Schatten-$p$ class operator for $p>n$.  Suppose $f_1, ..., f_{2n}\in \mathscr{C}^{\infty}(\overline{\bn}) $. Then the product of the commutators
	\[
	[T^{(0)}_{f_1}, T^{(0)}_{f_2}]\cdots [T^{(0)}_{f_{2n-1}}, T^{(0)}_{2n}]
	\]
	may not be a trace class operator. Helton and Howe \cite{HH2, Howeaftermath} made a breakthrough by considering the antisymmetric sum of $T_{f_1}, ..., T_{f_{2n}}$ defined by
	\[
	[T^{(0)}_{f_1}, ..., T^{(0)}_{f_{2n}}]:=\sum_{\tau\in S_{2n}} \operatorname{sgn}(\tau) T^{(0)}_{f_{\tau(1)}}T^{(0)}_{f_{\tau(2)}}...T^{(0)}_{f_{\tau(2n)}},
	\]
	where $S_{2n}$ is the permutations group of $2n$ elements and $\operatorname{sgn}$ is the sign of the permutation $\tau$. The following is a remarkable generalization of the Helton-Howe trace formula for the commutator of two Toeplitz operators on $L^2_a(\dd)$. 
	\begin{thm}[Helton-Howe]\label{thm:helton-howe} On the Bergman space $L^2_{a}(\bn)$ (and the Hardy space $H^2(\sn)$), the antisymmetric sum $[T^{(0)}_{f_1}, ..., T^{(0)}_{f_{2n}}]$ is a trace class operator, and 
		\begin{equation}\label{eq:Helton-Howe}
		\operatorname{Tr}\big([T^{(0)}_{f_1}, ..., T^{(0)}_{f_{2n}}]\big)=\frac{n!}{(2\pi \sqrt{-1})^n} \int_{\bn} \intd f_1\wedge \intd f_2\wedge \cdots \wedge \intd f_{2n}.
		\end{equation}
	\end{thm}

We observe that by Stoke's theorem, the above integral only depends on the value of $f_1, ..., f_{2n}$ on the unit sphere $\mathbb{S}^{2n-1}=\partial \overline{\bn}$, i.e. 
	\[
	\frac{n!}{(2\pi \sqrt{-1})^n} \int_{\bn} \intd f_1\wedge \intd f_2\wedge \cdots \wedge \intd f_{2n}=\frac{n!}{(2\pi \sqrt{-1})^n} \int_{\mathbb{S}^{2n-1}} f_1\intd f_2\wedge \cdots \wedge \intd f_{2n}.
	\]	
	
	The above idea of Schatten-$p$ commutators was revolutionized by Connes \cite{Co:noncommutative} into a fundamental concept in noncommutative geometry as $p$-summable Fredholm modules. And the Helton-Howe trace formula in Theorem \ref{thm:helton-howe}  inspired  Connes to his ingenious discovery of cyclic cohomology and the Chern character for $p$-summable Fredholm modules.  The building block of the Connes-Chern character is the semi-commutator
	\[
	\sigma_t(f, g)=T^{(t)}_f T^{(t)}_g-T^{(t)}_{fg}.
	\]
	Modulo constant, the Connes-Chern character for the Toeplitz extension is defined to be
	\begin{equation}\label{eq:chern char}
		\tau_t(f_0, \cdots, f_{2p-1}):=\operatorname{Tr}\big(\sigma_t(f_0, f_1)...\sigma_t(f_{2p-2}, f_{2p-1})\big)-\operatorname{Tr}\big(\sigma_t(f_1, f_2)... \sigma_t(f_{2p-1}, f_0)\big),
	\end{equation}
	for $p>n$.  And the Helton-Howe trace, Equation (\ref{eq:Helton-Howe}), is the top degree component of the above Connes-Chern character.  
	
Connes observed that the above cocycle in general is not local, i.e. the value of $\tau_t(f_0, \cdots, f_{2p-1})$ can not be expressed by the germ of $f_0\otimes \cdots \otimes f_{2p-1}$ on the diagonal in 
\[
\underbrace{\mathbb{B}_n\times \cdots \times \mathbb{B}_n}_{2p}.
\] 
Connes  \cite{Connes:book, Connes-Moscovici:local} improved the Chern character, Equation (\ref{eq:chern char}), by employing the Dixmier trace on the operator ideal $L^{1, \infty}$. In a series of works, Engli\v{s} and his coauthors, e.g. \cite{En-Gu-Zh:toeplitz, En-Ro:dixmier, En-Zh:hankel-domains, En-Zh:hankel}, studied a generalization of the Helton-Howe trace formula by considering the Dixmier trace on  the product
	\[
	[T^{(0)}_{f_1},T^{(0)}_{f_2}]\cdots [T^{(0)}_{f_{2n-1}}, T^{(0)}_{f_{2n}}]. 
	\] 
They expressed the Dixmier trace of the above product as an integral of the product of Poisson brackets between $f_{2k-1}$ and $f_{2k}$, $k=1,\cdots, n$. 

In this article, we take a different approach to study the Connes-Chern character (\ref{eq:chern char}) and the Helton-Howe trace, Theorem \ref{thm:helton-howe}. Our main idea is to put the Bergman space and Hardy space into the family of weighted Bergman spaces  $L^2_a(\bn, \lambda_t):=\bert$ for the measure 
	\[
	\intd\lambda_t(z)=\frac{(n-1)!}{\pi^nB(n, t+1)}(1-|z|^2)^t\intd m(z),
	\]
	where $B(n, t+1)$ is the Beta function. Let $T^{(t)}_f$ be the associated Toeplitz operator on $\bert$ with symbol $f$.  We study the large $t$ behavior of the Connes-Chern character and the Helton-Howe trace. 
	
	Our first result is about the Helton-Howe trace for $\BTt_{f_1}, \ldots, \BTt_{f_{2n}}$ for functions $f_1, \ldots, f_{2n}\in \mathscr{C}^2(\overline{\bn})$, which generalizes Theorem \ref{thm:helton-howe} to all weighted Bergman spaces.
	\begin{thm}\label{thm:HH} (Theorem \ref{thm: main}) 
		Suppose $f_1, f_2, \ldots, f_{2n}\in\mathscr{C}^2(\overline{\bn})$ and $t\geq -1$. 
		\begin{enumerate} 
			\item $[\BTt_{f_1}, \BTt_{f_2}, \ldots, \BTt_{f_{2n}}]$ is in the trace class $\mathcal{S}^1$.
			\item 
			\begin{equation}
				\Tr[\BTt_{f_1}, \BTt_{f_2}, \ldots, \BTt_{f_{2n}}]=\frac{n!}{(2\pi i)^n}\int_{\bn}\intd f_1\wedge\intd f_2\wedge\ldots\wedge\intd f_{2n},
			\end{equation}
			which is independent of $t$.
		\end{enumerate}
	\end{thm}
	
Thanks to the use of pseudodifferential calculus  and its generalization in the proof, Helton-Howe's original theorem needs to assume that that the Toeplitz operators to have smooth symbols. In this paper, we develop a new approach to study the trace formula using harmonic analysis. As a result, we obtain an improvement of the Helton-Howe trace formula for Toeplitz operators with $\mathscr{C}^2$ symbols. 
	
Our second result is about the Connes-Chern character $\tau_t(f_0, \ldots, f_{2p-1})$ for $f_0, \ldots, f_{2p-1}\in \mathscr{C}^2(\overline{\bn})$.  Different from the Helton-Howe trace, $\tau_t$ vanishes as $t$ goes to $\infty$.  In the following theorem, we identify the leading term of $\tau_t$ as $t\to \infty$. 
	
\begin{thm}\label{thm:ConnesChern}(Theorem \ref{thm: connes chern})
Suppose $p\geq n+1$ is an integer and $f_0, f_1,\ldots, f_{2p-1}\in\mathscr{C}^2(\overline{\bn})$. Then 
\[
\begin{split}
&\lim_{t\to\infty}t^{p-n}\tau_t(f_0,f_1,\ldots,f_{2p-1})\\
=&\frac{n^p}{\pi^n}\int_{\bn}\bigg(\prod_{j=0}^{p-1}C_1(f_{2j},f_{2j+1})(z)-\prod_{j=0}^{p-1}C_1(f_{2j+1},f_{2j+2})(z)\bigg)\frac{\intd m(z)}{(1-|z|^2)^{n+1}},
\end{split}
\]
where $C_1(f,g)$ is defined as follows, 
\begin{equation}\label{eq:C1}
C_1(f,g)(z)=-\frac{1}{n}(1-|z|^2)\bigg[\sum_{i=1}^n\partial_if(z)\bpartial_ig(z)-Rf(z)\bar{R}g(z)\bigg],\ R=\sum_{i=1}^n z_i\partial _{z_i},\ \overline{R}=\sum_{i=1}^n \bar{z}_i\overline{\partial}_{z_i}.
\end{equation}
\end{thm}

Our approach to the above two main theorems is heavily influenced by the idea of quantization, \cite{Be:quantization, Bo-Gu:spectral, Bo-Le-Ta-We:asymptotic, Co:deformation, En:asymptotics, En:berezin, En:forelli-rudin, En:weighted}. Geometrically the defining function $\psi=1-|z|^2$ on $\bn$ defines the Bergman metric in the following way. 
	\[
	\begin{split}
		\omega:=&i\frac{-\psi \partial \bar{\partial}\psi+\partial \psi\wedge \bar{\partial}\psi}{\psi^2}\\
		=&i\frac{ (1-|z|^2)  \sum_{j=1}^n \partial z_j\wedge \bar{\partial} \bar{z}_j +(\sum_j \bar{z}_j\partial z_j)\wedge (\sum_{j'}z_{j'}\bar{\partial}\bar{z}_{j'})}{(1-|z|^2)^2}
	\end{split}
	\]
	defines a symplectic form on $\bn$, c.f. \cite[Prop. 2.6]{Le-Ta-We:poisson}. The Toeplitz operator $T^{(t)}_f$ gives a quantization of the symplectic form $i\omega $ , e.g. \cite[Theorem 3]{En:weighted}, satisfying
	\begin{equation}\label{eq:asymptotic}
		||T^{(t)}_fT^{(t)}_g-\sum_{j=0}^kt^{-j}T^{(t)}_{C_j(f,g)}||=O(t^{-k-1}),\ t\to \infty,
	\end{equation}
	where $C_j$ is a bilinear operator on $\mathscr{C}^\infty(\overline{\bn})$ and $C_1$ is defined in Equation (\ref{eq:C1}). The asymptotic expansion formula (\ref{eq:asymptotic}) provides the key tool to study the semi-commutator,
	\[
	\sigma_t(f, g)=T^{(t)}_f T^{(t)}_g-T^{(t)}_{fg}.
	\]
The asymptotic expansion formula (\ref{eq:asymptotic}) in the literature, e.g.  \cite{En:weighted}, was well studied for estimates on the operator norm. Estimates about the Schatten-$p$ norm in the expansion (\ref{eq:asymptotic}) are needed in our applications to the tracial property in Theorem \ref{thm:HH} and \ref{thm:ConnesChern}. We prove these estimates in Theorem \ref{thm: quantization bergman}. As we need to study Toeplitz operators with $\mathscr{C}^2$ symbols in Theorem \ref{thm:HH} and \ref{thm:ConnesChern} and an estimate on Schatten-$p$ norm in Theorem \ref{thm: quantization bergman}, we need a new method to develop the asymptotic estimate in Theorem \ref{thm: quantization bergman} different from the classical method via pseudodifferential/Toeplitz operator calculus \cite{BMS, Bo-Gu:spectral, En:berezin, En:weighted, En-Gu-Zh:toeplitz, En-Ro:dixmier, En-Zh:hankel-domains, En-Zh:hankel, Karabegov}, which requires to work with smooth symbols. Our main tool comes from integration formulas developed in Section 4 of \cite{TWZ:semicommutator}. Theorem \ref{thm:ConnesChern}  follows from the Schatten-$p$ estimate of the semi-commutator $\sigma_t(f,g)$. As a byproduct, our method also provides an explicit algorithm to compute the  bilinear differential operator $C_j$ in the asymptotic expansion (\ref{eq:asymptotic}), which is in general hard to compute. 
		
A crucial fact used in our estimate is the different behavior of the quantization in complex normal and complex tangential directions (see Remark \ref{rem: CN CT} and Corollaries \ref{cor: Rabfgk membership}, \ref{cor: semicom with radial functions}). Roughly speaking, the expansion of the integral kernel in the complex tangential direction creates improvement of $\frac{1}{2}$ in the Schatten-$p$ membership, whereas the expansion of the integral kernel in the complex normal direction creates improvement of $1$ in the Schatten-$p$ membership. Essentially, this allows us to reduce our estimates to the complex tangential direction. In contrast, in pseudodifferential calculous, Helton and Howe \cite{HH2} considered symbol functions for order 0 operators that are constant along the radial direction far away from the zero section of the cotangent bundle, which simplifies the corresponding estimates. And the difference between tangential and radial estimates suggests a deep link of our study with the Heisenberg calculus for contact manifolds, e.g. \cite{BG, HoweQuantumMechanics, Ponge, Taylor}.

Instead of a direct computation as in \cite{HH2}, we prove Theorem \ref{thm:HH} in two steps. Suppose $f_1, f_2, \ldots, f_{2n}\in\mathscr{C}^2(\overline{\bn})$ and $t\geq -1$.  
\begin{enumerate}
\item 		\begin{equation}\label{eq:tt+1}
		\Tr[\BTt_{f_1}, \BTt_{f_2}, \ldots, \BTt_{f_{2n}}]=\Tr[T^{(t+1)}_{f_1}, T^{(t+1)}_{f_2}, \ldots, T^{(t+1)}_{f_{2n}}].
		\end{equation}

\item  \begin{equation}\label{eqn:trace limit}
	\lim_{s\to\infty}\Tr[T^{(s)}_{f_1},\ldots,T^{(s)}_{f_{2n}}]=\frac{n!}{(2\pi i)^n}\int\intd f_1\wedge\ldots\wedge\intd f_{2n}.
\end{equation}
\end{enumerate} 
	
Equation (\ref{eq:tt+1}) is proved using the decomposition formula $\BTt_{f_i}=A_i+B_i$ in Section \ref{sec: traces}, Lemmas \ref{lem: Ttf=X+Y+K} and \ref{lem: T=X+Y+Z+E}. Here the operator $A_i=\BTtpt_{f_i}$ is the restriction of $\BTtp_{f_i}$ on $\bert$. We observe that $A_i$ is the ``principal part'' of $\BTt_{f_i}$ for $f_i\in \mathscr{C}^2(\overline{\bn})$. In Section \ref{sec: traces} we show that
	\begin{equation}\label{eqn: intro 1}
	\Tr\big\{[\BTt_{f_1},\BTt_{f_2},\ldots,\BTt_{f_{2n}}]-[A_1,A_2,\ldots,A_{2n}]\big\}=0.
	\end{equation}
	 We point out that the ``minor part'' $B_i$ does not live in Schatten class $\mathcal{S}^p$ for $p$ small enough. This fact prevents us from proving \eqref{eqn: intro 1} only using operator-theoretic tools. In Section \ref{sec: traces}, we develop Hypotheses A which handles the operator-theoretic part of the proof of \eqref{eqn: intro 1}, and Hypotheses B, where the rest is handled. As $[\BTtpt_{f_1},\ldots,\BTtpt_{f_{2n}}]$ is the restriction of $[\BTtp_{f_1},\BTtp_{f_2},\ldots,\BTtp_{f_{2n}}]$ on $\bert$, Lemma \ref{lem: trace on two spaces} gives
\begin{equation}\label{eqn: intro 3}
\Tr[\BTtpt_{f_1},\ldots,\BTtpt_{f_{2n}}]=\Tr[\BTtp_{f_1},\ldots,\BTtp_{f_{2n}}].
\end{equation}
We obtain Equation (\ref{eq:tt+1}) by combining Equations (\ref{eqn: intro 1}) and (\ref{eqn: intro 3}).

Equation (\ref{eqn:trace limit}) is proved by improving Theorem \ref{thm:ConnesChern}. In Theorem \ref{thm:ConnesChern}, $p$ is assumed to be greater than or equal to $n+1$, while there are only $2n$ functions in Equation (\ref{eqn:trace limit}). We observe that in general the product $\sigma_t(f_1, f_2)\ldots \sigma_t(f_{2n-1}, f_{2n})$ is not a trace class operator. This leads us to introduce even and odd partial antisymmetrizations in Section \ref{sec: even and odd}.  For $f_1, ..., f_n, g_1, ..., g_n\in L^\infty(\bn)$,  define
	\begin{equation*}
		[f_1, g_1,\ldots, f_n, g_n]^{\odd}_t=\sum_{\tau\in S_n}\sgn(\tau)\sigma_t(f_{\tau(1)}, g_1)\ldots\sigma_t(f_{\tau(n)}, g_n),
	\end{equation*}
	and
	\begin{equation*}
		[f_1,g_1,\ldots,f_n,g_n]^{\even}_t=\sum_{\tau\in S_n}\sgn(\tau)\sigma_t(f_1, g_{\tau(1)})\ldots\sigma_t(f_n,g_{\tau(n)}).
	\end{equation*}
Using the asymptotic expansion of semi-commutator in Section \ref{sec: Quantization}, we improve Theorem \ref{thm:ConnesChern} to the even and odd partial antisymmetrizations.
	\begin{thm}\label{thm:partial} (Theorem \ref{thm: AS partial trace dimension n})
		Suppose $t\geq-1$ and $f_1, g_1, \ldots, f_n, g_n\in\mathscr{C}^{2}(\overline{\bn})$.
		Then the partial antisymmetrizations $[f_1,g_1,\ldots,f_n, g_n]_t^{\odd}$ and $[f_1,g_1,\ldots,f_n, g_n]_t^{\even}$ are in the trace class. Moreover,
		\[
		\begin{split}
			\lim_{t\to\infty}\Tr[f_1,g_1,\ldots,f_n,g_n]_t^{\odd}=&\lim_{t\to\infty}\Tr[f_1,g_1,\ldots,f_n,g_n]_t^{\even}=\frac{1}{(2\pi i)^n}\int_{\bn}\partial f_1\wedge\bpartial g_1\wedge\ldots\wedge\partial f_n\wedge\bpartial g_n.
		\end{split}
		\]
	\end{thm}
As a corollary of Theorem \ref{thm:partial}, we obtain Equation (\ref{eqn:trace limit}) by further antisymmetrization of the even (odd) antisymmetric sums. 

	Our solution to the generalized Helton-Howe trace formula through quantization is closely related to the method developed in \cite{Bo-Le-Ta-We:asymptotic} to a solution to the Atiyah-Weinstein conjecture for quantized contact transform. Such a similarity suggests that our developments can be generalized to strongly pseudoconvex domains,  egg domains, Fock spaces, submodules and their quotient modules of $L^2_{a,t}(\bn)$, e.g. \cite{Fa-Wa-Xi:helton}, and the Dury-Arveson spaces.  More generally, we hope that our study will shed a light on constructing new cyclic cocycles beyond the Helton-Howe traces, which could have  applications in noncommutative geometry. 
	
	The paper is organized as follows. In Section \ref{sec: preliminary}, we recall the definitions of weighted Bergman spaces and Hardy spaces together with their basic properties and properties of Schatten $p$-class operators. Some tools developed in \cite{TWZ:semicommutator} are also reviewed. In Section \ref{sec: operator btw weighted spaces}, we develop criteria for integral operators to be bounded between different weighted spaces $L^2_{a, t}(\bn)$. We develop some useful estimate for integral operators to belong to Schatten-$p$ class in Section \ref{sec: schatten criterion}. In Section \ref{sec: traces}, we prove Equation (\ref{eqn: intro 1}). We develop the asymptotic expansion formula, Theorem \ref{thm:partial}, in Section \ref{sec: Quantization}. We introduce the even and odd antisymmetrization and prove Theorem \ref{thm: AS partial trace dimension n} and Equation (\ref{eqn:trace limit}) in Section \ref{sec: even and odd}. The two main theorems, Theorem \ref{thm:HH} and \ref{thm:ConnesChern}, are proved in Section \ref{sec: main}. And a proof of the generalized Bochner-Martinelli formula is included in the appendix. \\
	
	\noindent{\bf Acknowledgment:} We would like to thank Mohammad Jabbari, Richard Rochberg, and Jingbo Xia for inspiring discussions. Tang is partially supported by NSF Grants DMS 1800666, 1952551; Wang is partially supported by NSF Grant 2101370.
	
	\section{Preliminaries}\label{sec: preliminary}
	
	In this section, we recall some basic definitions and properties about weighted Bergman spaces and Schatten-$p$ class operators.
	
	\subsection{Spaces on $\bn$}\label{subsec: spaces on bn}
	Let $\bn$ be the open unit ball of $\cn$ and $\sn=\partial\bn$ the unit sphere. Let $m$ be the Lebesgue measure on $\bn$ and $\sigma$ be the surface measure on $\sn$. Denote $\sigma_{2n-1}=\sigma(\sn)=\frac{2\pi^n}{(n-1)!}$.
	
	~
	
	\noindent{\bf Hardy Space: }The Hardy space $H^2(\sn)$ is the Hilbert space of holomorphic functions on $\bn$ with the norm
	\[
	\|f\|_{H^2(\sn)}^2=\sup_{0<r<1}\int_{\sn}|f(rz)|^2\frac{\intd\sigma(z)}{\sigma_{2n-1}}.
	\]
	Equivalently, $H^2(\sn)$ is the closure of analytic polynomials in  $L^2(\sn):=L^2(\sn,\frac{\intd\sigma}{\sigma_{2n-1}})$. The Hardy space is a reproducing kernel Hilbert space on $\bn$ and the reproducing kernel is
	\[
	K_w(z)=\frac{1}{(1-\la z,w\ra)^n},\quad\forall w\in\bn.
	\]
	For any $f\in L^\infty(\sn)$, the Toeplitz operator on $H^2(\sn)$ with symbol $f$ is defined to be the compression
	\[
	T_f=PM_f|_{H^2(\sn)},
	\]
	where $M_f$ is the pointwise multiplication on $L^2(\sn)$, and $P$ is the orthogonal projection from $L^2(\sn)$ onto $H^2(\sn)$.
	Using the reproducing kernel, we can write $T_f$ as an integral operator. For $h\in H^2(\sn)$,
	\[
	T_fh(z)=\int_{\sn}f(w)h(w)K_w(z)\frac{\intd\sigma(z)}{\sigma_{2n-1}},\quad\forall z\in\bn.
	\]
	Our discussion will also involve Hankel operators. The Hankel operator with symbol $f$ is $H_f=(I-P)M_fP$ from $H^2(\sn)$ to $L^2(\sn)$.
	
	~
	
	\noindent{\bf Weighted Bergman Spaces: }For $t>-1$, define the probability measure on $\bn$:
	\[
	\intd\lambda_t(z)=\frac{(n-1)!}{\pi^nB(n, t+1)}(1-|z|^2)^t\intd m(z).
	\]
	Here $B(n, t+1)$ is the Beta function. The weighted Bergman space $\bert$ is the subspace of $L^2(\bn, \lambda_t)$ consisting of holomorphic functions on $\bn$. The reproducing kernel of $\bert$ is
	\[
	\BKt_w(z)=\frac{1}{(1-\la z,w\ra)^{n+1+t}},\quad\forall w\in\bn.
	\]
	For any $f\in L^\infty(\bn)$, the Toeplitz operator $\BTt_f$ is the compression
	\[
	\BTt_f=\BPt M^{(t)}_f|_{\bert},
	\]
	where $\BPt$ is the orthogonal projection from $L^2(\bn, \lambda_t)$ onto $\bert$, and $M^{(t)}_f$ is the multiplication operator on $L^2(\bn,\lambda_t)$. The Hankel operator with symbol $f$ is
	\[
	H^{(t)}_f=(I-P^{(t)})M_f^{(t)}P^{(t)}.
	\]
	Using the reproducing kernels, we can write $\BTt_f, H^{(t)}_f$ as integral operators. For $h\in\bert$, we have the following expressions,
	\[
	\BTt_fh(z)=\int_{\bn}f(w)h(w)\BKt_w(z)\intd\lambda_t(w),\quad\forall z\in\bn, 
	\]
	\[
	H^{(t)}_fh(z)=\int_{\bn}\big(f(z)-f(w)\big)h(w)\BKt_w(z)\intd\lambda_t(w),\quad\forall z\in\bn.
	\]
	It is sometimes understood that the Hardy space is the weighted Bergman space at $t=-1$. This is partially justified by the form of its reproducing kernel.

	An important tool on $\bn$ is the M\"{o}bius transform.
	\begin{defn}\label{defn: Mobius transform}
		For $z\in\bn$, $z\neq0$, the M\"{o}bius transform $\varphi_z$ is the biholomorphic mapping on $\bn$ defined as follows.
		\[
		\varphi_z(w)=\frac{z-P_z(w)-(1-|z|^2)^{1/2}Q_z(w)}{1-\la w,z\ra},\quad\forall w\in\overline{\bn}.
		\]
		Here $P_z$ and $Q_z$ denote the orthogonal projection from $\cn$ onto $\mathbb{C}z$ and $z^\perp$, respectively.
		Define
		\[
		\varphi_0(w)=-w,\quad\forall w\in\overline{\bn}.
		\]
		It is well-known that $\varphi_z$ is an automorphism of $\bn$ satisfying $\varphi_z\circ\varphi_z=Id$. Also, the two variable function $\rho(z,w):=|\varphi_z(w)|=|\varphi_w(z)|$ defines a metric on $\bn$. Moreover, $\beta(z,w):=\tanh^{-1}\rho(z,w)$ coincides with the Bergman metric on $\bn$.
	\end{defn}

	We list some lemmas that serve as basic tools on $\bn$. Most of the following can be found in \cite{Rudinbookunitball, Zhubookspaces}. Some are proved in our other paper \cite{TWZ:semicommutator}.
	
	For non-negative values $A, B$, by $A\lesssim B$ we mean that there is a constant $C$ such that $A\leq CB$. Sometimes, to emphasize that the constant $C$ depends on some parameter $a$, we write $A\lesssim_a B$. The notations $\gtrsim, \gtrsim_a, \approx, \approx_a$ are defined similarly.
	
	\begin{lem} (\cite[Lemma 2.2]{TWZ:semicommutator} ) \label{lem: Mobius basics}
		Suppose $z, w, \zeta\in\bn$.
		\begin{itemize}
			\item[(1)] $\frac{1}{1-\la\varphi_{\zeta}(z),\varphi_{\zeta}(w)\ra}=\frac{(1-\la z,\zeta\ra)(1-\la\zeta,w\ra)}{(1-|\zeta|^2)(1-\la z,w\ra)}$.
			\item[(2)] $1-|\varphi_z(w)|^2=\frac{(1-|z|^2)(1-|w|^2)}{|1-\la z,w\ra|^2}$.
			\item[(3)] For any $R>0$ there exists $C>1$ such that whenever $\beta(z,w)<R$,
			\[
			\frac{1}{C}\leq\frac{1-|z|^2}{1-|w|^2}\leq C,\quad\frac{1}{C}\leq\frac{|1-\la z,\zeta\ra|}{|1-\la w,\zeta\ra|}\leq C.
			\]
			\item[(4)] The real Jacobian of $\varphi_z$ is $\frac{(1-|z|^2)^{n+1}}{|1-\la z,\cdot\ra|^{2n+2}}$ on $\bn$ and $\frac{(1-|z|^2)^n}{|1-\la z,\cdot\ra|^{2n}}$ on $\sn$.
			\item[(5)] For $z\in\bn$,
			\[
			z-\varphi_z(w)=\frac{(1-|z|^2)P_z(w)+(1-|z|^2)^{1/2}Q_z(w)}{1-\la w,z\ra}:=\frac{A_zw}{1-\la w,z\ra},
			\]
			where $A_z=[a_z^{ij}]$ is an $n\times n$ matrix depending on $z$, and $w$ is viewed as a column vector.
			\item[(6)] There exists $C>0$ such that for any $z\in\bn$, $z\neq0$,
			\begin{equation}\label{eqn: |z-w|}
				|z-P_z(w)|\leq |\varphi_z(w)||1-\la z,w\ra|,\quad |Q_z(w)|\leq C|\varphi_z(w)||1-\la z,w\ra|^{1/2},
			\end{equation}
			and
			\begin{equation}\label{eqn: tangent nontangent estimates}
				|z-w|\leq C|\varphi_z(w)||1-\la z,w\ra|^{1/2}.
			\end{equation}
			In contrast, if $n=1$, then $|z-w|=|\varphi_z(w)||1-z\bw|$.
			\item[(7)] $1-|z|^2\leq2|1-\la z,w\ra|$.
		\end{itemize}
	\end{lem}
	
	\begin{lem}(\cite[Lemma 2.4]{TWZ:semicommutator})\label{lem: Rudin Forelli generalizations}~
		
		\begin{itemize}
			\item[(1)] Suppose $t>-1$, $c\in\mathbb{R}$. Then there exists $C>0$ such that for any $z\in\bn$,
			\begin{equation}\label{eqn: Rudin-Forelli 1-1}
				\int_{\bn}\frac{(1-|w|^2)^t}{|1-\la z,w\ra|^{n+1+t+c}}\intd m(w)\leq\begin{cases}
					C(1-|z|^2)^{-c},&c>0,\\
					C\ln\frac{1}{1-|z|^2},&c=0,\\
					C,&c<0,
				\end{cases}
			\end{equation}
			\begin{equation}\label{eqn: Rudin-Forelli 1-2}
				\int_{\sn}\frac{1}{|1-\la z,w\ra|^{n+c}}\intd\sigma(w)\leq \begin{cases}
					C(1-|z|^2)^{-c},&c>0,\\
					C\ln\frac{1}{1-|z|^2},&c=0,\\
					C,&c<0.
				\end{cases}
			\end{equation}
			\item[(2)] Suppose $t>-1$, $a, b, c>0$, $a\geq c, b\geq c$, and $a+b<n+1+t+c$. Then there exists $C>0$ such that for any $z, \xi\in\bn$,
			\begin{equation}\label{eqn: Rudin-Forelli 2}
				\int_{\bn}\frac{(1-|w|^2)^t}{|1-\la z, w\ra|^a|1-\la w, \xi\ra|^b}\intd m(w)\leq C\frac{1}{|1-\la z,\xi\ra|^c}.
			\end{equation}
			\item[(3)] Suppose $\phi:(0,1)\to[0,\infty)$ is measurable. Suppose $a>-n$, $b\in\mathbb{R}$, and
			\[
			\phi(s)\lesssim s^a(1-s)^b,\quad s\in(0,1).
			\]
			Then for any $t>-1-b$, $c>-b$ there exists $C>0$ such that for any $z\in\bn$,
			\begin{equation}\label{eqn: Rudin-Forelli 3}
				\int_{\bn}\phi(|\varphi_z(w)|^2)\frac{(1-|w|^2)^t}{|1-\la z,w\ra|^{n+1+t+c}}\intd m(w)\leq C(1-|z|^2)^{-c}.
			\end{equation}
		\end{itemize}
	\end{lem}
	
	\begin{lem}\label{lem: trace on two spaces}
		Suppose $s>t\geq-1$ and $T$ is a bounded operator on $L_{a,s}^2(\bn)$. Suppose $L_{a,t}^2(\bn)$ is invariant under $T$. Denote $\hat{T}$ its restriction to $L_{a,t}^2(\bn)$. By the closed graph theorem, $\hat{T}$ is bounded on $L_{a,t}^2(\bn)$. Assume that $T\in\mathcal{S}^1\big(L_{a,s}^2(\bn)\big)$ and $\hat{T}\in\mathcal{S}^1\big(L_{a,t}^2(\bn)\big)$. Then
		\[
		\Tr T=\Tr\hat{T}.
		\]
	\end{lem}
	
	\begin{proof}
		The set $\{z^\alpha\}_{\alpha\in\ind}$ forms an orthogonal basis of both $L_{a,s}^2(\bn)$ and $L_{a,t}^2(\bn)$. For $\alpha\in\ind$, we write
		\[
		Tz^\alpha=\hat{T}z^\alpha=\sum_{\beta\in\ind}a_{\alpha,\beta}z^\beta.
		\]
		Compute the traces of $T$ and $\hat{T}$ as follows.
		\begin{flalign*}
			\Tr T=&\sum_{\alpha\in\ind}\frac{\la Tz^\alpha, z^\alpha\ra_{L_{a,s}^2(\bn)}}{\|z^\alpha\|_{L_{a,s}^2(\bn)}^{2}}\\
			=&\sum_{\alpha\in\ind}\frac{a_{\alpha,\alpha}\la z^\alpha,z^\alpha\ra_{L_{a,s}^2(\bn)}}{\|z^\alpha\|^2_{L_{a,s}^2(\bn)}}\\
			=&\sum_{\alpha\in\ind}a_{\alpha,\alpha}\\
			=&\sum_{\alpha\in\ind}\frac{\la \hat{T}z^\alpha, z^\alpha\ra_{L_{a,t}^2(\bn)}}{\|z^\alpha\|_{L_{a,t}^2(\bn)}^{2}}=\Tr\hat{T}.
		\end{flalign*}
		This completes the proof of Lemma \ref{lem: trace on two spaces}.
	\end{proof}

\subsection{Schatten Class Operators}
	
	For $p>0$, a bounded operator $T$ on a Hilbert space $\mathcal{H}$ is said to be in the Schatten-$p$ class $\mathcal{S}^p$ if $|T|^p$ belongs to the trace class. The Schatten-$p$ class operators $\mathcal{S}^p$ are analogues of $l^p$ spaces in the operator-theoretic setting. Conventionally, $\mathcal{S}^\infty$ denotes the  space of compact operators. The following two lemmas will be used constantly.
	\begin{lem}(\cite[Theorem 2.8]{Simon})\label{lem: holder inequality for operators}
	Suppose $A$, $B$ are bounded operators on a Hilbert space, and $1\leq p, q, r\leq\infty$, $\frac{1}{p}+\frac{1}{q}=\frac{1}{r}$. If $A\in\mathcal{S}^p$ and $B\in\mathcal{S}^q$, then
	\[
	AB\in\mathcal{S}^r,\quad\text{ and }\quad\|AB\|_{\mathcal{S}^r}\leq\|A\|_{\mathcal{S}^p}\|B\|_{\mathcal{S}^q}.
	\]
	\end{lem}
	
	\begin{lem}(\cite[Corollary 3.8]{Simon})\label{lem: trace of commutators}
	Suppose $A, B$ are bounded operators on a Hilbert space. If both $AB$ and $BA$ are in the trace class then
	\[
	\Tr[A, B]=0.
	\]
	\end{lem}
	
	\begin{rem}\label{rem: X1...Xn trace}
	Suppose $n>1$ and $X_1, X_2, \ldots, X_n$ are bounded operators on a Hilbert space such that
	\[
	X_i\in\mathcal{S}^p,\quad\forall p>n,
	\]
	and there exists $j\in\{1,\ldots,n\}$ such that $X_j$ is in $\mathcal{S}^p$ for some $p<n$. Then we can choose $p_1,\ldots, p_{j-1}, p_{j+1},\ldots, p_n>n$ and $1\leq p_j<n$ such that
	\[
	X_i\in\mathcal{S}^{p_i}, i=1,\ldots,n,\quad\text{ and }\quad\frac{1}{p_1}+\frac{1}{p_2}+\ldots+\frac{1}{p_n}=1.
	\]
	By an inductive application of Lemma \ref{lem: holder inequality for operators} we have
	\[
	X_1X_2\ldots X_n\in\mathcal{S}^1.
	\]
	Moreover, by Lemma \ref{lem: trace of commutators}, for any $k=1,\ldots,n$,
	\[
	\Tr X_1X_2\ldots X_n=\Tr X_kX_{k+1}\ldots X_nX_1\ldots X_{k-1}.
	\]
	This will be used repeatedly in this paper.
	\end{rem}
	
	\subsection{Integration by Parts}\label{subsec: intgration by parts}
	Some integral formulas developed in \cite{TWZ:semicommutator} will be used in our proofs. We give a brief review here. Let us start with some auxiliary functions and operations.
	
	Recall that by Lemma \ref{lem: Mobius basics} (5), for $z\in\bn$,
	\[
	(1-|z|^2)P_z(w)+(1-|z|^2)^{1/2}Q_z(w)=(1-\la w,z\ra)(z-\varphi_z(w))=A_zw,
	\]
	where $A_z$ is an $n\times n$ matrix depending on $z$, and $w$ is treated as a column vector.
	
	\begin{defn}\label{defn: d numbers I function}
		For multi-indices $\alpha, \beta\in\ind$ and $z\in\bn$, define
		\[
		d_{\alpha,\beta}(z)=\int_{\sn}\big(A_z\zeta\big)^\alpha\overline{\big(A_z\zeta\big)}^\beta\frac{\intd\sigma(\zeta)}{\sigma_{2n-1}}.
		\]
		In particular, $d_{0,0}=1$, and
		\begin{equation}\label{eqn: d at (1,0...0)}
			d_{\alpha,\beta}(z)=\delta_{\alpha,\beta}(1-|z|^2)^{\alpha_1+|\alpha|}\frac{(n-1)!\alpha!}{(n-1+|\alpha|)!},\quad\text{if }z=(z_1,0,\ldots,0).
		\end{equation}
		For multi-indices $\alpha, \beta\in\ind$ and $\zeta\in\cn$, denote
		\[
		I^{\alpha,\beta}(\zeta)=\zeta^\alpha\bar{\zeta}^\beta.
		\]
	\end{defn}
\begin{defn}\label{defn: BFt BGt}
	For $t\in\mathbb{R}$, denote
	\[
	\phi_t(s)=(1-s)^t.
	\]
	Suppose $\phi:(0,1)\to[0,\infty)$ is a measurable function. For a positive integer $m$ and any $t>-1$, define the operations on $\phi$
	\begin{equation}
		\BFt_m\phi(s)=\int_s^1 r^{m-1}\phi(r)(1-r)^t\intd r\in[0,\infty],
	\end{equation}
	and
	\begin{equation}
		\BGt_m\phi(s)=\frac{1}{s^m\phi_{t+1}(s)}\BFt_m\phi(s)=\frac{\int_s^1 r^{m-1}\phi(r)(1-r)^t\intd r}{s^m(1-s)^{t+1}}\in[0,\infty].
	\end{equation}
	For any $t>-1$, inductively define the functions
	\[
	\Phi_{n,0}^{(t)}\equiv1,\quad\Phi_{n,k+1}^{(t)}=M_{\phi_1}\big(\BGt_{n+k}\big)^2\Phi_{n,k}^{(t)}.
	\]
	Equivalently,
	\begin{equation}
		\Phi_{n,k}^{(t)}=M_{\phi_1}\big(\BGt_{n+k-1}\big)^2\ldots M_{\phi_1}\big(\BGt_n\big)^21.
	\end{equation}
\end{defn}

	\begin{lem}(\cite[Lemma 4.2]{TWZ:semicommutator})\label{lem: formula Bn}
		Suppose $t>-1$, $\alpha, \beta\in\ind$. Suppose $\phi: (0,1)\to[0,\infty)$ is measurable and $v\in\mathscr{C}^1(\bn)$. Then the following hold.
		\begin{enumerate}
			\item If $|\alpha|\geq|\beta|$ and all integrals converge absolutely, then
			\begin{flalign}\label{eqn: formula Bn d''}
				&\int_{\bn}\phi(|\varphi_z(w)|^2)I^{\alpha,\beta}(z-w) v(w)\BKt_w(z)\intd\lambda_t(w)\\
				=&\begin{cases}
					\frac{d_{\alpha,\beta}(z)}{B(n,t+1)}\cdot \BFt_{n+|\beta|}\phi(0) v(z)-\sum_{j=1}^n\int_{\bn}\BGt_{|\beta|+n}\phi(|\varphi_z(w)|^2)I^{\alpha,\beta+e_j}(z-w)S_j(w)\BKt_w(z)\intd\lambda_t(w),&\\
					\hspace{10cm} v(z)\neq0, \BFt_{n+|\beta|}\phi(0)<\infty,&\\
					&\\
					-\sum_{j=1}^n\int_{\bn}\BGt_{|\beta|+n}\phi(|\varphi_z(w)|^2)I^{\alpha,\beta+e_j}(z-w)S_j(w)\BKt_w(z)\intd\lambda_t(w),&\\
					\hspace{10cm}v(z)=0, \BFt_{n+|\beta|}\phi(0)\leq\infty,&
				\end{cases}
				\nonumber
			\end{flalign}
			where
			\[
			S_j(w,z)=\frac{(1-|w|^2)\bpartial_{w_j}\big[(1-\la z,w\ra)^{|\beta|}v(w)\big]}{(1-\la w,z\ra)(1-\la z,w\ra)^{|\beta|}}.
			\]
			\item If $|\alpha|\leq|\beta|$ and all integrals converge absolutely, then
			\begin{flalign}\label{eqn: formula Bn d'}
				&\int_{\bn}\phi(|\varphi_z(w)|^2)I^{\alpha,\beta}(z-w)v(z)\BKt_w(z)\intd\lambda_t(z)\\
				=&\begin{cases}
					\frac{d_{\alpha,\beta}(w)}{B(n,t+1)}\cdot\BFt_{n+|\alpha|}\phi(0)v(w)+\sum_{i=1}^n\int_{\bn}\BGt_{|\alpha|+n}\phi(|\varphi_z(w)|^2)I^{\alpha+e_i,\beta}(z-w)\widetilde{S}_i(z)\BKt_w(z)\intd\lambda_t(z),&\\
					\hspace{10cm}v(w)\neq0, \BFt_{n+|\alpha|}\phi(0)<\infty,&\\
					&\\
					\sum_{i=1}^n\int_{\bn}\BGt_{|\alpha|+n}\phi(|\varphi_z(w)|^2)I^{\alpha+e_i,\beta}(z-w)\widetilde{S}_i(z)\BKt_w(z)\intd\lambda_t(z),&\\
					\hspace{10cm}v(w)=0, \BFt_{n+|\alpha|}\phi(0)\leq\infty,&
				\end{cases}
				\nonumber
			\end{flalign}
			where
			\[
			\widetilde{S}_i(z,w)=\frac{(1-|z|^2)\partial_{z_i}\big[(1-\la z,w\ra)^{|\alpha|}v(z)\big]}{(1-\la w,z\ra)(1-\la z,w\ra)^{|\alpha|}}.
			\]
		\end{enumerate}	
	\end{lem}

A particular case is to take $\alpha=\beta=0$, $z=0$, $\phi=1$ in \eqref{eqn: formula Bn d''}, which gives the following. The estimates (2) and (3) are straight forward to verify.
	\begin{lem}\label{lem: formula R on bn}
	For any $t>-1$ the following hold.
	\begin{itemize}
		\item[(1)] For any $v\in\mathscr{C}^1(\overline{\bn})$,
		\begin{equation}\label{eqn: formula R on bn}
			\int_{\bn}v(z)\intd\lambda_t(z)=v(0)+\frac{t+1}{n+t+1}\int_{\bn}\BGt_n1(|z|^2)\bar{R}v(z)\intd\lambda_{t+1}(z).
		\end{equation}
		\item[(2)] $\BGt_n1(s)\approx s^{-n}$ in a neighborhood of $0$, and $\lim_{s\to 1^-}\BGt_n1(s)=\frac{1}{t+1}$.
		\item[(3)] $|\BGt_n1(s)-\frac{1}{t+1}|\lesssim 1-s$ for $s$ in a neighborhood of $1$.
	\end{itemize}
\end{lem}

\begin{lem}(\cite[Lemma 4.6]{TWZ:semicommutator})\label{lem: BM hardy}
	Suppose $\alpha, \beta\in\ind$, and $v\in\mathscr{C}^1(\overline{\bn})$. Then the following hold.
	\begin{enumerate}
		\item If $|\alpha|\geq|\beta|$, then
		\begin{flalign}\label{eqn: BM hardy d''}
			&\int_{\sn}I^{\alpha,\beta}(z-w)v(w)K_w(z)\frac{\intd\sigma(w)}{\sigma_{2n-1}}\\
			=&d_{\alpha,\beta}(z)v(z)-\frac{1}{n}\sum_{j=1}^n\int_{\bn}|\varphi_z(w)|^{-2|\beta|-2n}I^{\alpha,\beta+e_j}(z-w)\frac{\bpartial_j\big[(1-\la z,w\ra)^{|\beta|}v(w)\big]}{(1-\la z,w\ra)^{|\beta|}(1-\la w,z\ra)}K_w(z)\intd\lambda_0(w).\nonumber
		\end{flalign}
		\item If $|\alpha|\leq|\beta|$, then
		\begin{flalign}\label{eqn: BM hardy d'}
			&\int_{\sn}I^{\alpha,\beta}(z-w)v(z)K_w(z)\frac{\intd\sigma(z)}{\sigma_{2n-1}}\\
			=&d_{\beta,\alpha}(w)v(w)+\frac{1}{n}\sum_{i=1}^n\int_{\bn}|\varphi_z(w)|^{-2|\alpha|-2n}I^{\alpha+e_i,\beta}(z-w)\frac{\partial_i\big[(1-\la z,w\ra)^{|\alpha|}v(z)\big]}{(1-\la z,w\ra)^{|\alpha|}(1-\la w,z\ra)}K_w(z)\intd\lambda_0(z).\nonumber
		\end{flalign}
	\end{enumerate}
\end{lem}

	\begin{lem}(\cite[Lemma 4.3]{TWZ:semicommutator})\label{lem: formula F induction}
		Suppose $k$ is a non-negative integer and $\Gamma\subset\ind\times\ind$ is a finite set of multi-indices with $|\alpha|=|\beta|=k$ for every $(\alpha,\beta)\in\Gamma$. Suppose for some $\epsilon>-1-t$, $\{F_{\alpha,\beta}\}_{(\alpha,\beta)\in\Gamma}\subset\mathscr{C}^{2}(\bn\times\bn)$ and
		\begin{equation}\label{eqn: F assumption 1}
			\bigg|\sum_{(\alpha,\beta)\in\Gamma}I^{\alpha,\beta}(z-w)F_{\alpha,\beta}(z,w)\bigg|\lesssim|\varphi_z(w)|^{2k}|1-\la z,w\ra|^{2k+\epsilon},
		\end{equation}
		\begin{equation}\label{eqn: F assumption 2}
			\bigg|\sum_{j=1}^n\sum_{(\alpha,\beta)\in\Gamma}I^{\alpha,\beta+e_j}(z-w)\bpartial_{w_j} F_{\alpha,\beta}(z,w)\bigg|\lesssim|\varphi_z(w)|^{2k+1}|1-\la z,w\ra|^{2k+\epsilon}.
		\end{equation}
		
		Then
		\begin{flalign}\label{eqn: formula F induction}
			&\int_{\bn^2}\Phi^{(t)}_{n,k}(|\varphi_z(w)|^2)\frac{\sum_{(\alpha,\beta)\in\Gamma}I^{\alpha,\beta}(z-w)F_{\alpha,\beta}(z,w)}{|1-\la z,w\ra|^{2k}}\BKt_w(z)\intd\lambda_t(w)\intd\lambda_t(z)\nonumber\\
			=&\frac{\BFt_{n+k}\Phi^{(t)}_{n,k}(0)}{B(n,t+1)}\int_{\bn}(1-|z|^2)^{-2k}\sum_{(\alpha,\beta)\in\Gamma}d_{\alpha,\beta}(z)F_{\alpha,\beta}(z,z)\intd\lambda_t(z)\\
			&-\int_{\bn^2}\Phi^{(t)}_{n,k+1}(|\varphi_z(w)|^2)\frac{\sum_{i,j=1}^n\sum_{(\alpha,\beta)\in\Gamma}I^{\alpha+e_i,\beta+e_j}(z-w)D_{i,j}F_{\alpha,\beta}(z,w)}{|1-\la z,w\ra|^{2(k+1)}}\BKt_w(z)\intd\lambda_t(z)\intd\lambda_t(w).\nonumber
		\end{flalign}
		Here $D_{i,j}$ denotes the operation
		\[
		D_{i,j}=(1-\la z,w\ra)^2\partial_{z_i}\bpartial_{w_j}.
		\]
	\end{lem}

	\begin{lem}\label{lem: Phi esitmates}
		Suppose $k$ is a nonnegative integer. Then there exist $C>c>0$ such that for $t$ large enough,
		\begin{equation}\label{eqn: FPhi}
			ct^{-n-k}\leq\BFt_{n+k}\Phi_{n,k}^{(t)}(0)\leq Ct^{-n-k},
		\end{equation}
		and
		\begin{equation}\label{eqn: FPhi t/4}
			\int_0^1\Phi_{n,k}^{(t)}(s)s^{n+k-1}(1-s)^{\frac{t}{4}}\intd s\leq Ct^{-n-k}.
		\end{equation}
	\end{lem}
	
	\begin{proof}
		The case when $k=0$ is simply a consequence of \cite[Equation (8.9)]{TWZ:semicommutator}. We assume $k>0$.
		
		From the Sterling's asymptotic formula (cf. \cite[Theorem 1.4.1]{specialfunctions}) and the well-known identity $B(x,y)=\frac{\Gamma(x)\Gamma(y)}{\Gamma(x+y)}$ it follows that for fixed $x>-1$,
		\[
		B(x,y)\approx_x y^{-x}
		\]
		for large $y$.
		By \cite[Lemma 8.4]{TWZ:semicommutator}, we compute $\BFt_{n+k}\Phi_{n,k}^{(t)}(0)$ as follows, 
		\begin{flalign*}
			&\BFt_{n+k}\Phi_{n,k}^{(t)}(0)\\
			=&\BFt_{n+k}M_{\phi_1}\big(\BGt_{n+k-1}\big)^2\ldots M_{\phi_1}\big(\BGt_n\big)^21(0)\\
			=&\sum_{j_1,\ldots,j_k=0}^\infty\frac{B(n+k+j_1+\ldots+j_k,t+1)}{1\cdot(1+j_1)(2+j_1)(2+j_1+j_2)\ldots(k+j_1+\ldots+j_{k-1})(k+j_1+\ldots+j_k)}\\
			=&\sum_{0\leq s_1\leq\ldots\leq s_k<\infty}\frac{B(n+k+s_k,t+1)}{(1+s_1)(2+s_1)\ldots(k-1+s_{k-1})(k+s_{k-1})(k+s_k)}\\
			\leq&\sum_{s_k=0}^\infty\sum_{s_1,\ldots,s_{k-1}=0}^\infty \frac{B(n+k+s_k,t+1)}{(1+s_1)(2+s_1)\ldots(k-1+s_{k-1})(k+s_{k-1})(k+s_k)}\\
			=&\sum_{s_k=0}^\infty \frac{B(n+k+s_k,t+1)}{(k-1)!(k+s_k)}\\
			=&\frac{1}{(k-1)!}\int_0^1\big(\sum_{s_k=0}^\infty\frac{x^{n-1+(k+s_k)}}{k+s_k}\big)(1-x)^t\intd x\\
			\leq&\frac{1}{(k-1)!}\int_0^1\big(\sum_{s_k=0}^\infty x^{n-1}\bigg(\ln\frac{1}{1-x}\bigg)(1-x)^t\intd x\\
			\lesssim&\int_0^1x^{n-1}(1-x)^{t-1}\intd x\\
			=&B(n,t)\\
			\lesssim&t^{-n-k},
		\end{flalign*}
		when $t$ is large. The other inequality $\BFt_{n+k}\Phi_{n,k}^{(t)}(0)\gtrsim t^{-n-k}$ is also obvious from the equation above. This proves \eqref{eqn: FPhi}.
		
		By the expansion
		\[
		\frac{1}{(1-s)^{\frac{3t}{4}}}=\sum_{j=0}^\infty \frac{\Gamma(\frac{3t}{4}+j)}{j!\Gamma(\frac{3t}{4})}s^j,
		\]
		we directly compute the following integral, 
		\begin{flalign*}
			&\int_0^1\Phi_{n,k}^{(t)}(s)s^{n+k-1}(1-s)^{\frac{t}{4}}\intd s\\
			=&\sum_{j=0}^\infty\frac{\Gamma(\frac{3t}{4}+j)}{j!\Gamma(\frac{3t}{4})}\int_0^1\Phi_{n,k}^{(t)}(s)s^{n+k+j-1}(1-s)^t\intd s\\
			=&\sum_{j_0=0}^\infty\frac{\Gamma(\frac{3t}{4}+j_0)}{j_0!\Gamma(\frac{3t}{4})}\BFt_{n+k+j_0}\Phi_{n,k}^{(t)}(0).
		\end{flalign*}
		Similarly as the proof of \eqref{eqn: FPhi}, assuming $t$ is large enough, we estimate the above term as follows, 
		\begin{flalign*}
			\leq&\sum_{0\leq s_0\leq s_k<\infty}\frac{\Gamma(\frac{3t}{4}+s_0)B(n+k+s_k,t+1)}{(s_0+1)!\Gamma(\frac{3t}{4})(k-1)!(k+s_k)}\\
			=&\sum_{a=0}^\infty\sum_{b=0}^\infty \frac{\Gamma(\frac{3t}{4}+a)B(n+k+a+b,t+1)}{(a+1)!\Gamma(\frac{3t}{4})(k-1)!(k+a+b)}\\
			\leq&\sum_{a=0}^\infty \frac{\Gamma(\frac{3t}{4}+a)}{(a+1)!\Gamma(\frac{3t}{4})(k-1)!}\int_0^1\sum_{b=0}^\infty\frac{s^{b+1}}{k+a+b}s^{n+k+a-2}(1-s)^t\intd s\\
			\leq&\sum_{a=0}^\infty \frac{\Gamma(\frac{3t}{4}+a)}{(a+1)!\Gamma(\frac{3t}{4})(k-1)!}\int_0^1\sum_{b=0}^\infty\frac{s^{b+1}}{b+1}s^{n+k+a-2}(1-s)^t\intd s\\
			=&\sum_{a=0}^\infty \frac{\Gamma(\frac{3t}{4}+a)}{(a+1)!\Gamma(\frac{3t}{4})(k-1)!}\int_0^1s^{n+k+a-2}(1-s)^t\ln\frac{1}{1-s}\intd s\\
			\lesssim&\sum_{a=0}^\infty \frac{\Gamma(\frac{3t}{4}+a)}{(a+1)!\Gamma(\frac{3t}{4})(k-1)!}B(n+k+a,t)\\
			\leq&\int_0^1s^{n+k-1}(1-s)^{t-1-\frac{3t}{4}}\intd s\\
			=&B(n+k,\frac{t}{4})\\
			\lesssim&t^{-n-k}.
		\end{flalign*}
		This completes the proof of Lemma \ref{lem: Phi esitmates}.
	\end{proof}

	\section{Operators Between Weighted Spaces}\label{sec: operator btw weighted spaces}
	For our proofs in this paper it is important to get criteria for integral operators to be bounded between different weighted spaces. In this section, we obtain such criteria. Most of the lemmas established in this section follow from standard techniques from, for example, \cite{Rudinbookunitball, Zhubookspaces}. One thing less standard is that our integral kernels may involve functions of $|\varphi_z(w)|^2$. These functions come from applying the integration by parts formulas in Subsection \ref{subsec: intgration by parts} on integral formulas of Toeplitz operators.
	\begin{lem}\label{lem: bounded operator using Schur's test}
		Suppose $t>-1, s>-1$, $a>-n$, $b\geq0$, and $c<t+1+b-\frac{s+1}{2}$. Suppose $F(z,w)$ is measurable on $\bn\times\bn$, and
		\[
		|F(z,w)|\leq\frac{|\varphi_z(w)|^{2a}(1-|\varphi_z(w)|^2)^b}{|1-\la z, w\ra|^{n+1+t-c}},\quad\forall z, w\in\bn.
		\]
		Then the integral operator
		\[
		Th(z)=\int_{\bn}h(w)F(z,w)\intd\lambda_t(w)
		\]
		is bounded from $L^2(\lambda_{s+2c})$ to $L^2(\lambda_s)$.
	\end{lem}
	
	\begin{proof}
		By assumption, we can take $x\in\mathbb{R}$ so that
		\[
		\max\{-1-t-b, -c-s-1-b\}<x<\min\{b-c, t+b-2c-s\}.
		\]
		Then we have the following inequalities,
		\[
		t+x>-1-b,\quad -x-c>-b,\quad x+c+s>-1-b,\quad t-x-2c-s>-b.
		\]
		Take $p(w)=(1-|w|^2)^x$ and $q(z)=(1-|z|^2)^{x+c}$. The integral kernel of $T$ as an operator from $L^2(\lambda_{s+2c})$ to $L^2(\lambda_s)$ is
		\[
		\frac{B(n,s+2c+1)}{B(n,t+1)}F(z,w)(1-|w|^2)^{t-s-2c}.
		\]
		Then by the inequalities above and Lemma \ref{lem: Rudin Forelli generalizations} (3), we have the following estimates of integrals,
		\begin{flalign*}
			&\int_{\bn}|F(z,w)|(1-|w|^2)^{t-s-2c}p(w)\intd\lambda_{s+2c}(w)\\
			\lesssim&\int_{\bn}|\varphi_z(w)|^{2a}(1-|\varphi_z(w)|^2)^b\frac{(1-|w|^2)^{t+x}}{|1-\la z,w\ra|^{n+1+t-c}}\intd m(w)\\
			\lesssim&(1-|z|^2)^{x+c}=q(z),
		\end{flalign*}
		and
		\begin{flalign*}
			&\int_{\bn}|F(z,w)|(1-|w|^2)^{t-s-2c}q(z)\intd\lambda_s(z)\\
			\lesssim& \int_{\bn}|\varphi_z(w)|^{2a}(1-|\varphi_z(w)|^2)^b\frac{(1-|w|^2)^{t-s-2c}(1-|z|^2)^{x+c+s}}{|1-\la z,w\ra|^{n+1+t-c}}\intd m(z)\\
			\lesssim&(1-|w|^2)^x=p(w).
		\end{flalign*}
		
		The conclusion follows from Schur's test (cf. \cite[Theorem 3.6]{Zhu:bookoperator}). This completes the proof of Lemma \ref{lem: bounded operator using Schur's test}.
	\end{proof}

	\begin{lem}\label{lem: bounded operator to Sn using Schur's test}
		Suppose $t>-1$ and $0<d<c<t+1$. Suppose $F(z,w)$ is piecewise continuous on $\bn\times\bn$, and
		\[
		|F(z,w)|\leq\frac{1}{|1-\la z,w\ra|^{n+1+t-c}},\quad\forall z, w\in\bn.
		\]
		For any $0<r<1$, define the integral operator
		\[
		T_rh(z)=\int_{\bn}h(w)F(rz,w)\intd\lambda_t(w),\quad z\in\sn.
		\]
		Then each $T_r$ defines a bounded operator from $L^2(\lambda_{-1+2d})$ to $L^2(\sn)$. Moreover,
		\[
		\sup_{0<r<1}\|T_r\|_{L^2(\lambda_{-1+2d})\to L^2(\sn)}<\infty.
		\]
	\end{lem}
	
	\begin{proof}
		Take $p(w)=(1-|w|^2)^{-d}$, $q(z)=1$. The integral kernel of $T_r$ is
		\[
		T_r(z,w):=\frac{B(n,2d)}{B(n,t+1)}F(rz,w)(1-|w|^2)^{t+1-2d}.
		\]
		As $t-d>-1$ and $d<c$, by Lemma \ref{lem: Rudin Forelli generalizations}, we have the following estimates of integrals,
		\[
		\int_{\bn}|T_r(z,w)|p(w)\intd\lambda_{-1+2d}(w)\lesssim\int_{\bn}\frac{(1-|w|^2)^{t-d}}{|1-\la rz,w\ra|^{n+1+t-c}}\intd m(w)\lesssim 1=q(z),
		\]
		and
		\begin{eqnarray*}
		\int_{\sn}|T_r(z,w)|q(z)\frac{\intd\sigma(z)}{\sigma_{2n-1}}&\lesssim& \int_{\sn}\frac{(1-|w|^2)^{t+1-2d}}{|1-\la z,rw\ra|^{n+1+t-c}}\intd\sigma(z)\\
		& \lesssim&(1-|w|^2)^{c-2d}\lesssim(1-|w|^2)^{-d}=p(w).
		\end{eqnarray*}
		The estimates are independent of $r$. Thus the conclusion follows from Schur's test. This completes the proof of Lemma \ref{lem: bounded operator to Sn using Schur's test}.
	\end{proof}
	
	\begin{lem}\label{lem: bounded operator to H2 using Schur's test}
		Suppose $t>-1$ and $0<d<c<t+1$. Suppose $F(z,w)$ is piecewise continuous on $\bn\times\bn$, and
		\[
		|F(z,w)|\leq\frac{1}{|1-\la z,w\ra|^{n+1+t-c}},\quad\forall z, w\in\bn.
		\]
		Assume further that $F(z,w)$ is holomorphic in $z$ for each fixed $w$. Define the operator
		\[
		Th(z)=\int_{\bn}h(w)F(z,w)\intd\lambda_t(w),\quad z\in\bn.
		\]
		Then $T$ defines a bounded operator from $L^2(\lambda_{-1+2d})$ to $H^2(\sn)$.
	\end{lem}
	
	\begin{proof}
		Since $F(z,w)$ is holomorphic in $z$, for every $h$, the function $Th$ is holomorphic in $\bn$. Denote
		\[
		M=\sup_{0<r<1}\|T_r\|_{L^2(\lambda_{-1+2d})\to L^2(\sn)}<\infty.
		\]
		Then for any $h\in L^2(\lambda_{-1+2d})$, we verify the following identity
		\[
		\sup_{0<r<1}\int_{\sn}|Th(rz)|^2\intd\sigma(z)=\sup_{0<r<1}\|T_rh\|_{L^2(\sn)}^2\leq M^2\|h\|_{L^2(\lambda_{-1+2d})}^2.
		\]
		Therefore $T$ is bounded. This completes the proof of Lemma \ref{lem: bounded operator to H2 using Schur's test}.
	\end{proof}
	
	The following fact about embedding operators between different weighted spaces is well-known to experts. In our proof, estimates of the operator norms and Schatten norms of these embeddings are needed. We give the calculation for completeness.
	\begin{lem}\label{lem: Embedding Schatten membership}
		Suppose $t\geq-1$ and $c>0$. Then for any $p>\frac{2n}{c}$ the following hold.
		\begin{itemize}
			\item[(1)] The embedding map $E_{t,c}$ from $L_{a,t}^2(\bn)$ to $L_{a,t+c}^2(\bn)$ (identifying $L_{a,-1}(\bn)$ with $H^2(\sn)$) is in the Schatten $p$ class $\mathcal{S}^p$.
			\item[(2)]There exists $C>0$, independent of $t$, such that
			\[
			\|E_{t,c}\|=1,\quad\|E_{t,c}\|_p\leq C(t+n+1)^{\frac{n}{p}}.
			\]
		\end{itemize}
	\end{lem}
	
	\begin{proof}
		For any multi-index $\alpha\in\ind$, we have the norm
		\[
		\|z^\alpha\|_{\bert}^2=\frac{\Gamma(n+t+1)\alpha!}{\Gamma(n+|\alpha|+t+1)}.
		\]
		Thus $E_{t,c}$ is unitarily equivalent to a diagonal operator with entry
		\[
		\frac{\|z^\alpha\|_{L_{a,t+c}^2(\bn)}}{\|z^\alpha\|_{L_{a,t}^2(\bn)}}=\sqrt{\frac{\Gamma(n+t+c+1)\Gamma(n+|\alpha|+t+1)}{\Gamma(n+|\alpha|+t+c+1)\Gamma(n+t+1)}}=\sqrt{\frac{B(n+|\alpha|+t+1, c)}{B(n+t+1,c)}}
		\]
		at $\alpha\in\ind$. It follows immediately that $\|E_{t,c}\|=1$. Thus it suffices to show
		\begin{eqnarray*}
		\sum_{\alpha\in\ind}\bigg(\frac{B(n+|\alpha|+t+1, c)}{B(n+t+1,c)}\bigg)^\frac{p}{2}&=&\sum_{d=0}^\infty\frac{(d+n-1)!}{d!(n-1)!}\cdot\bigg(\frac{B(n+d+t+1, c)}{B(n+t+1,c)}\bigg)^\frac{p}{2}\lesssim (t+n+1)^n.
		\end{eqnarray*}
		Since $B(n+d+t+1, c)$ is decreasing in $d$, and $\frac{(d+n-1)!}{d!}\approx d^{n-1}$ for $d\geq1$, we have the following inequalities,
		\begin{flalign*}
			&\sum_{d=0}^\infty\frac{(d+n-1)!}{d!(n-1)!}\cdot\bigg(\frac{B(n+d+t+1, c)}{B(n+t+1,c)}\bigg)^\frac{p}{2}\\
			\lesssim&1+\int_0^\infty x^{n-1}\bigg(\frac{B(n+x+t+1, c)}{B(n+t+1,c)}\bigg)^{p/2}\intd x\\
			\lesssim&1+\int_0^\infty x^{n-1}\bigg(\frac{n+t+1}{n+x+t+1}\bigg)^{\frac{pc}{2}}\intd x.
		\end{flalign*}
		If we take the change of variable $y=\frac{x}{n+t+1}$ then the integral above has the following bounds,
		\[
		\lesssim1+(n+t+1)^n\int_0^\infty y^{n-1}\bigg(\frac{1}{1+y}\bigg)^{\frac{pc}{2}}\intd y\lesssim(n+t+1)^n
		\]
		when $n<\frac{pc}{2}$, i.e., $p>\frac{2n}{c}$. This completes the proof of Lemma \ref{lem: Embedding Schatten membership}.
	\end{proof}

	\section{Schatten Class Criteria}\label{sec: schatten criterion}
	In this section, we obtain criteria for integral operators to belong to Schatten class. As explained in the beginning of Section \ref{sec: operator btw weighted spaces}, the integral kernels may involve functions of $|\varphi_z(w)|^2$. Thus we need to modify some of the well-known results to fit our case.
	
	A standard way of proving Schatten class membership of integral operators is to use the following lemma. See \cite[Lemma 8.26]{Zhu:bookoperator} for a proof when $n=1$, using interpolation. The exact same proof works for general $n$.
	\begin{lem}\label{lem: Lp implies Schatten p }
		Suppose $t>-1$ and $G(z,w)$ is measurable on $\bn\times\bn$. Suppose $p\geq 2$ and
		\[
		\int_{\bn^2}|G(z,w)|^p|\BKt_w(z)|^2\intd\lambda_t(w)\intd\lambda_t(z)<\infty.
		\]
		Define the operator
		\[
		Th(z)=\int_{\bn}h(w)G(z,w)\BKt_w(z)\intd\lambda_t(w).
		\]
		Then $T$ defines a bounded operator on $L^2(\lambda_t)$ that belongs to $\mathcal{S}^p$.
	\end{lem}
	
	\begin{cor}\label{cor: integral schatten membership no phi}
		Suppose $t>-1$ and $F(z,w)$ is measurable on $\bn\times\bn$. Suppose $c>0$ and
		\[
		|F(z,w)|\leq\frac{1}{|1-\la z,w\ra|^{n+1+t-c}},\quad z, w\in\bn.
		\]
		Define the integral operator
		\[
		Th(z)=\int_{\bn}h(w)F(z,w)\intd\lambda_t(w),\quad h\in L^2(\lambda_t).
		\]
		Then $T$ is bounded on $L^2(\lambda_t)$, and for any $p>\frac{n}{c}$ and $p\geq2$, $T\in\mathcal{S}^p$.
	\end{cor}
	
	\begin{proof}
		Let
		\[
		G(z,w)=\frac{F(z,w)}{\BKt_w(z)},\quad z, w\in\bn.
		\]
		Then by assumption, we have the following bound,
		\[
		|G(z,w)|=|F(z,w)|\cdot|1-\la z,w\ra|^{n+1+t}\leq|1-\la z,w\ra|^c.
		\]
		Choose $p$ so that $\frac{n}{c}<p<\frac{n+1+t}{c}$. Then by Lemma \ref{lem: Rudin Forelli generalizations}, we compute the following integral. 
		\begin{flalign*}
			&\int_{\bn}\int_{\bn}|G(z,w)|^p|\BKt_w(z)|^2\intd\lambda_t(w)\intd\lambda_t(z)\\
			\lesssim&\int_{\bn}\int_{\bn}\frac{(1-|z|^2)^t(1-|w|^2)^t}{|1-\la z,w\ra|^{2(n+1+t)-cp}}\intd m(w)\intd m(z)\\
			\lesssim&\int_{\bn}(1-|z|^2)^{t-(n+1+t-cp)}\intd m(z)\\
			=&\int_{\bn}(1-|z|^2)^{-1+(cp-n)}\intd m(z)<\infty.
		\end{flalign*}
		Since $G(z,w)$ is bounded, for $p\geq\frac{n+1+t}{c}$ the integral is also finite. Therefore we have the following bound, 
		\[
		\int_{\bn}\int_{\bn}|G(z,w)|^p|\BKt_w(z)|^2\intd\lambda_t(w)\intd\lambda_t(z)<\infty,\quad\forall p>\frac{n}{c}.
		\]
		Finally, by Lemma \ref{lem: Lp implies Schatten p }, $T\in\mathcal{S}^p$ for any $p>\frac{n}{c}$ and $p\geq2$. This completes the proof of Corollary \ref{cor: integral schatten membership no phi}.
	\end{proof}

	\begin{cor}\label{cor: commutator schatten membership no phi}
		Let $t, c, F, T$ be as in Corollary \ref{cor: integral schatten membership no phi}. Then for any Lipschitz function $u$ on $\bn$ and any $p>\frac{n}{c+\frac{1}{2}}$ and $p\geq2$, the commutator
		\[
		[T, M_u]\in\mathcal{S}^p.
		\]
	\end{cor}
	\begin{proof}
		By definition, for $h\in L^2(\lambda_t)$, we have the following expression,
		\[
		[T,M_u]h(z)=\int_{\bn}\big(u(w)-u(z)\big)F(z,w)h(w)\intd\lambda_t(w).
		\]
		Since $u$ is Lipschitz, by Lemma \ref{lem: Mobius basics} (6), we get the following bounds, 
		\[
		\big|\big(u(w)-u(z)\big)F(z,w)\big|\lesssim|z-w||F(z,w)|\lesssim\frac{1}{|1-\la z,w\ra|^{n+1+t-c-1/2}}.
		\]
		The corollary then follows from Corollary \ref{cor: integral schatten membership no phi}. This completes the proof of Corollary \ref{cor: commutator schatten membership no phi}.
	\end{proof}
	
	It is well-known that Hankel operators of Lipschitz symbols belong to $\mathcal{S}^p$ for any $p>2n$. In fact, the Schatten-class membership for Hankel operators is completely characterized (see \cite{Luecking92}, \cite[Theorem 8.36]{Zhu:bookoperator}). 
	\begin{cor}\label{cor: hankel schatten membership}
		Suppose $t>-1$ and $u$ is a Lipschitz function on $\bn$. Then the Hankel operator $H^{(t)}_u=(I-P^{(t)})M_uP^{(t)}$ belongs to $\mathcal{S}^p$ for any $p>2n$.
	\end{cor}
	\begin{proof}
		The corollary follows from Corollary \ref{cor: commutator schatten membership no phi} and the equation
		\[
		H^{(t)}_u=(I-P^{(t)})M_uP^{(t)}=[M_u, P^{(t)}]P^{(t)}.
		\]
		This completes the proof of Corollary \ref{cor: hankel schatten membership}.
	\end{proof}
	
	Corollaries \ref{cor: integral schatten membership no phi} and \ref{cor: commutator schatten membership no phi} provide convenient criteria for an integral operator to be in $\mathcal{S}^p$, when $p\geq2$. However, in this paper we will also need to deal with the case when $1\leq p<2$. Moreover, we will need to consider integral operators of the form
	\[
	Th(z)=\int_{\bn}\phi(|\varphi_z(w)|^2)F(z,w)h(w)\intd\lambda_t(w),
	\]
	where $\phi$ is an unbounded function. For such $T$, if we take the double integral as in Lemma \ref{lem: Lp implies Schatten p }, its integral is very likely to be infinite. In application, it is enough for us to obtain Schatten-class membership of the operators $TP^{(t)}$ or $P^{(t)}T$. An alternative way is to take advantage of Lemma \ref{lem: Embedding Schatten membership}. In particular, the following lemma holds.
	
	\begin{lem}\label{lem: integral operator Schatten membership}
		Suppose $t>-1$, $a>-n$, $b\geq0$ and $c>0$. Suppose $F(z,w)$ is measurable on $\bn\times\bn$, $\phi:(0,1)\to[0,\infty)$ is measurable, and
		\[
		\phi(s)\leq s^a(1-s)^b,\quad s\in(0,1),
		\]
		\[
		|F(z,w)|\leq\frac{1}{|1-\la z, w\ra|^{n+1+t-c}},\quad\forall z, w\in\bn.
		\]
		Define the integral operator on $L^2(\lambda_t)$.
		\[
		Th(z)=\int_{\bn}\phi(|\varphi_z(w)|^2)F(z,w)h(w)\intd\lambda_t(w).
		\]
		Then both $\BPt T$ and $T\BPt$ belong to $\mathcal{S}^p$ for any $p>\max\{\frac{n}{c}, \frac{n}{b+\frac{1+t}{2}}\}$.
	\end{lem}
	
	\begin{proof}
		Notice that $\BPt T\in\mathcal{S}^p$ if and only if $T^*\BPt\in\mathcal{S}^p$, and that $T^*$ is an integral operator with integral kernel satisfying the same estimate as $T$. Thus it suffices to prove the statement for $T\BPt$. For any $q>\max\{\frac{n}{c},\frac{n}{b+\frac{1+t}{2}}\}$, let $c'=\frac{n}{q}$. Then $c'<b+\frac{1+t}{2}$. Split the map $T\BPt$ as follows.
		\[
		T\BPt: L^2(\lambda_t)\xrightarrow{\BPt}L_{a,t}^2(\bn)\xrightarrow{E_{t,2c'}}L_{a,t+2c'}^2(\bn)\xrightarrow{\hat{T}}L^2(\lambda_t).
		\]
		Here $\hat{T}:L_{a,t+2c'}^2(\bn)\to L^2(\lambda_t)$ is defined by the same integral formula as $T$.
		By Lemma \ref{lem: Embedding Schatten membership}, $E_{t,2c'}\in\mathcal{S}^p$ for any $p>\frac{n}{c'}=q$. Also by Lemma \ref{lem: bounded operator using Schur's test}, $\hat{T}$ is bounded. Since $q$ is any number with $q>\max\{\frac{n}{c},\frac{n}{b+\frac{1+t}{2}}\}$, we have $T\BPt\in\mathcal{S}^p, \forall p>\max\{\frac{n}{c},\frac{n}{b+\frac{1+t}{2}}\}$. This completes the proof of Lemma \ref{lem: integral operator Schatten membership}.
	\end{proof}
	
	In the case when $\phi=s^k\Phi^{(t)}_{n,k}$, the following Schatten-norm estimate holds.
	\begin{thm}\label{thm: integral Schatten membership}
		Suppose $t>-1$, $c>0$ and $k$ is a non-negative integer. Suppose $F(z,w)$ is measurable on $\bn\times\bn$, and
		\[
		|F(z,w)|\leq\frac{|\varphi_z(w)|^{2k}}{|1-\la z, w\ra|^{n+1+t-c}},\quad\forall z, w\in\bn.
		\]
		Define the integral operator on $L^2(\bn,\lambda_t)$.
		\[
		Th(z)=\int_{\bn}\Phi_{n,k}^{(t)}(|\varphi_z(w)|^2)F(z,w)h(w)\intd\lambda_t(w).
		\]
	Then both $\BPt T$ and $T\BPt$ belong to $\mathcal{S}^p$ for any $p>\max\{\frac{n}{c}, \frac{n}{k+\frac{1+t}{2}}\}$, $p\geq1$. Moreover, for $p>\frac{n}{c}$, and $p\geq1$, and $t$ large enough, we have
		\[
		\|\BPt T\|_p\lesssim t^{-k+\frac{n}{p}},\quad \| T\BPt\|_p\lesssim t^{-k+\frac{n}{p}}.
		\]
	\end{thm}
	We postpone the proof of Theorem \ref{thm: integral Schatten membership} to the end of this section. As in the proof of Corollary \ref{cor: commutator schatten membership no phi}, Lemma \ref{lem: integral operator Schatten membership} and Theorem \ref{thm: integral Schatten membership} imply the following.
	\begin{cor}\label{cor: commutator schatten membership with phi}
		Suppose $t>-1$, $c>0$, $F(z,w)$ is measurable on $\bn\times\bn$, and
		\[
		|F(z,w)|\leq\frac{1}{|1-\la z,w\ra|^{n+1+t-c}},\quad z, w\in\bn.
		\]
		Suppose $\phi:(0,1)\to[0,\infty)$ is measurable. Define the integral operator on $L^2(\lambda_t)$,
		\[
		Th(z)=\int_{\bn}\phi(|\varphi_z(w)|^2)F(z,w)h(w)\intd\lambda_t(w).
		\]
		Assume that $u$ is a Lipschitz function on $\bn$.
		\begin{enumerate}
			\item Suppose $a>-n$, $b\geq0$ and $\phi(s)\leq s^a(1-s)^b$. Then
			\[
			[T,M_u]P^{(t)},\quad P^{(t)}[T, M_u]\in\mathcal{S}^p,\quad\forall p>\max\{\frac{n}{c+\frac{1}{2}}, \frac{n}{b+\frac{1+t}{2}}\}.
			\]
			\item If $\phi(s)=s^k\Phi^{(t)}_{n,k}$, then for $t$ large enough and $p>\frac{n}{c+\frac{1}{2}}$,
			\[
			\big\|[T,M_u]P^{(t)}\big\|_{\mathcal{S}^p}\lesssim t^{-k+\frac{n}{p}},\quad\big\|P^{(t)}[T,M_u]\big\|_{\mathcal{S}^p}\lesssim t^{-k+\frac{n}{p}}.
			\]
		\end{enumerate}
	\end{cor}
	
	A trivial application of Theorem \ref{thm: integral Schatten membership} gives the following.
	\begin{lem}\label{lem: Toeplitz Schatten Class}
		Suppose $c>0$, $t>-1$ and $f\in L^{\infty}(\bn)$ satisfies
		\[
		|f(z)|\leq (1-|z|^2)^c,\quad\forall z\in\bn.
		\]
		Then for any $p>\max\{\frac{n}{c},\frac{2n}{1+t}\}, p\geq1$,
		\[
		\BTt_f\in\mathcal{S}^p.
		\]
		For any $p>\frac{n}{c}$, $p\geq1$ and $t$ large enough,
		\[
		\|\BTt_f\|_{\mathcal{S}^p}\lesssim_p t^{\frac{n}{p}}.
		\]
	\end{lem}
	\begin{proof}
		By definition, we have the following expression of $\BTt_f$,
		\[
		\BTt_fh(z)=\int_{\bn}f(w)h(w)\BKt_w(z)\intd\lambda_t(w).
		\]
		By the assumption, we have the following inequalities
		\[
		|f(w)\BKt_w(z)|\leq\frac{(1-|z|^2)^c}{|1-\la z,w\ra|^{n+1+t}}\lesssim\frac{1}{|1-\la z, w\ra|^{n+1+t-c}}.
		\]
		Since $\Phi^{(t)}_{n,0}=1$, the conclusion follows directly from Theorem \ref{thm: integral Schatten membership}. This completes the proof of Lemma \ref{lem: Toeplitz Schatten Class}.
	\end{proof}

	\begin{lem}\label{lem:kernel est}
		Suppose $\phi:(0,1)\to[0,\infty)$ is measurable. Then for any $c, d\in\mathbb{R}$ there exist $C>0$ and $t_0>0$ such that whenever $t>t_0$,
		\begin{flalign}\label{eqn: Rudin-Forelli 4}
			\int_{\bn}\phi(|\varphi_z(w)|^2)\frac{(1-|w|^2)^{\frac{t}{2}+c}}{|1-\la z,w\ra|^{n+1+t+d}}\intd m(w)\leq C\int_0^1\phi(s)s^{n-1}(1-s)^{\frac{t}{4}}\intd s\cdot(1-|z|^2)^{-\frac{t}{2}-d+c},\quad\forall z\in\bn.
		\end{flalign}
	\end{lem}
	
	\begin{proof}
		Make the change of variable $w=\varphi_z(\xi)$. Using Lemma \ref{lem: Mobius basics} we have the following equation,
		\[
		\frac{(1-|w|^2)^t}{|1-\la z,w\ra|^{n+1+t}}\intd m(w)\xlongequal[\xi=\varphi_z(w)]{w=\varphi_z(\xi)}\frac{(1-|\xi|^2)^t}{|1-\la z,\xi\ra|^{n+1+t}}\intd m(\xi).
		\]
		Then we have the following estimate of the left side of Equation (\ref{eqn: Rudin-Forelli 4}),
		\begin{flalign*}
			&\int_{\bn}\phi(|\varphi_z(w)|^2)\frac{(1-|w|^2)^{\frac{t}{2}+c}}{|1-\la z,w\ra|^{n+1+t+d}}\intd m(w)\\
			=&\int_{\bn}\phi(|\xi|^2)\bigg(\frac{(1-|z|^2)(1-|\xi|^2)}{|1-\la z,\xi\ra|^2}\bigg)^{c-\frac{t}{2}}\bigg(\frac{|1-\la z,\xi\ra|}{1-|z|^2}\bigg)^d\frac{(1-|\xi|^2)^t}{|1-\la z,\xi\ra|^{n+1+t}}\intd m(\xi)\\
			=&(1-|z|^2)^{c-\frac{t}{2}-d}\int_{\bn}\phi(|\xi|^2)\frac{(1-|\xi|^2)^{c+\frac{t}{2}}}{|1-\la z,\xi\ra|^{n+1+2c-d}}\intd m(\xi)\\
			=&(1-|z|^2)^{c-\frac{t}{2}-d}\int_0^1\phi(r^2)(1-r^2)^{c+\frac{t}{2}}r^{2n-1}\bigg[\int_{\sn}\frac{1}{|1-\la rz,\zeta\ra|^{n+1+2c-d}}\intd\sigma(\zeta)\bigg]\intd r\\
			\lesssim&(1-|z|^2)^{c-\frac{t}{2}-d}\int_0^1\phi(r^2)(1-r^2)^{c+\frac{t}{2}}r^{2n-1}(1-|rz|^2)^{-a}\intd r\\
			\leq&(1-|z|^2)^{c-\frac{t}{2}-d}\int_0^1\phi(r^2)(1-r^2)^{c-a+\frac{t}{2}}r^{2n-1}\intd r\\
			\xlongequal{s=r^2}&\frac{1}{2}(1-|z|^2)^{c-\frac{t}{2}-d}\int_0^1\phi(s)(1-s)^{c-a+\frac{t}{2}}s^{n-1}\intd s.
		\end{flalign*}
		Here $a=2c-d$ if $2c-d>0$; $a=\frac{1}{2}$ if $2c-d=0$; $a=0$ if $2c-d<0$. For $t$ large enough, $c-a+\frac{t}{2}>\frac{t}{4}$. So we have the following inequality,
		\[
		\int_0^1\phi(s)(1-s)^{c-a+\frac{t}{2}}s^{n-1}\intd s\leq \int_0^1\phi(s)(1-s)^{\frac{t}{4}}s^{n-1}\intd s.
		\]
		This completes the proof of Lemma \ref{lem:kernel est}.
	\end{proof}

	\begin{proof}[{\bf Proof of Theorem \ref{thm: integral Schatten membership}}]
		The Schatten class memberships of $T\BPt$ and $\BPt T$ are implied by Lemma \ref{lem: integral operator Schatten membership} and \cite[Lemma 8.3]{TWZ:semicommutator}. It remains to prove the Schatten norm estimates. We may assume that $t$ is large enough so that $0<c<k+\frac{1+t}{2}$.
		Also, as in the proof of Lemma \ref{lem: integral operator Schatten membership}, it suffices to prove the statement for $T\BPt$.
		
		Split the operator $T\BPt$ as
		\[
		L^2(\bn,\lambda_t)\xrightarrow{\BPt}\bert\xrightarrow{E_{t,2c}} L_{a,t+2c}^2(\bn)\xrightarrow{\hat{T}}L^2(\bn,\lambda_t),
		\]
		where $\hat{T}$ is the same integral operator as $T$.
		By Lemma \ref{lem: Embedding Schatten membership}, it suffices to show that $\hat{T}$ defines a bounded operator from $L_{a,t+2c}^2(\bn)$ to $L^2(\bn,\lambda_t)$ with $\|\hat{T}\|_{L_{a,t+2c}^2(\bn)\to\bert}\lesssim t^{-k}$ for large $t$.
		
		This is done by Schur's test. Since $c<k+\frac{t+1}{2}$, if we take $x=c+\frac{t+1}{2}$, then $t+1+k>x>2c-k$. Take $p(w)=(1-|w|^2)^{-x}$ and $q(z)=(1-|z|^2)^{c-x}$. The integral kernel of $\hat{T}:L^2(\bn,\lambda_{t+2c})\to L^2(\bn,\lambda_t)$ is
		\[
		T(z,w):=\frac{B(n,t+2c+1)}{B(n,t+1)}\Phi_{n,k}^{(t)}(|\varphi_z(w)|^2)F(z,w)(1-|w|^2)^{-2c}.
		\]
		Then by \eqref{eqn: Rudin-Forelli 3}, \eqref{eqn: Rudin-Forelli 4}, \eqref{eqn: FPhi} and \eqref{eqn: FPhi t/4}, we have the following estimates, 
		\begin{flalign*}
			&\int_{\bn}|T(z,w)|p(w)\intd\lambda_{t+2c}(w)\\
			\approx&\frac{1}{B(n,t+1)}\int_{\bn}\Phi^{(t)}_{n,k}(|\varphi_z(w)|^2)|\varphi_z(w)|^{2k}\frac{(1-|w|^2)^{t-x}}{|1-\la z,w\ra|^{n+1+t-c}}\intd m(w)\\
			\lesssim&\frac{\int_0^1\Phi_{n,k}^{(t)}(s)s^{n+k-1}(1-s)^{\frac{t}{4}}\intd s}{B(n,t+1)}q(z)\lesssim t^{-k}q(z).
		\end{flalign*}
		Similarly, we have the following inequality, 
		\begin{equation*}
			\int_{\bn}|T(z,w)|q(z)\intd\lambda_t(z)
			\lesssim t^{-k}p(w).
		\end{equation*}
		From this we have the following bound for $p>n/c$ and large $t$,
		\[
		\|T\|_{\mathcal{S}^p}\leq\|E_{t,2c}\|_{\mathcal{S}^p}\|\hat{T}\|\lesssim t^{-k+n/p}.
		\]
		This completes the proof of Theorem \ref{thm: integral Schatten membership}.
	\end{proof}

\section{Traces on Different Weighted Bergman Spaces}\label{sec: traces}
As explained in the introduction, the goal of this section is to prove Equation \eqref{eqn: intro 1}. More precisely, we will prove Lemma \ref{lem: AS sum trace t-tp} and Lemma \ref{lem: AS sum trace t-tp hardy} stated below.

Suppose $t>s\geq-1$. It is well-known that $L_{a,s}^2(\bn)\subset L_{a,t}^2(\bn)$, and $L_{a,s}^2(\bn)$ is dense in $L_{a,t}^2(\bn)$. For a Toeplitz operator $\BTt_f$ on $\bert$, if the restriction of $\BTt_f$ on $L_{a,s}^2(\bn)$ defines a bounded operator on $L_{a,s}^2(\bn)$, then we denote this restriction to be $T^{(t,s)}_f$. On the other hand, if a Toeplitz operator $T^{(s)}_f$ on $L_{a,s}^2(\bn)$ extends (uniquely) into a bounded operator on $\bert$, then we denote this operator to be $T^{(s,t)}_f$. It follows from Lemma \ref{lem: bounded operator using Schur's test} that if $f\in L^\infty(\bn)$ and $t>s>-1$, then the restricted operator $T^{(t,s)}_f$ is well-defined.

	\begin{lem}\label{lem: AS sum trace t-tp}
		Suppose $f_1, f_2, \ldots, f_{2n}\in\mathscr{C}^2(\overline{\bn})$ and $t>-1$. Then
		\[
		[\BTt_{f_1},\BTt_{f_2},\ldots,\BTt_{f_{2n}}]-[\BTtpt_{f_1},\BTtpt_{f_2},\ldots,\BTtpt_{f_{2n}}]
		\]
		is a trace class operator on $\bert$ with zero trace.
	\end{lem}

In the case of the Hardy space we show the following.
\begin{lem}\label{lem: AS sum trace t-tp hardy}
	Suppose $f_1, f_2, \ldots, f_{2n}\in\mathscr{C}^2(\overline{\bn})$. Then
	\[
	[T_{f_1},T_{f_2},\ldots,T_{f_{2n}}]-[T^{(1,-1)}_{f_1},T^{(1,-1)}_{f_2},\ldots,T^{(1,-1)}_{f_{2n}}]
	\]
	is a trace class operator on $H^2(\sn)$ with zero trace.
\end{lem}

The fact that each $T^{(1,-1)}_{f_i}$ is well-defined will be explained in Remark \ref{rem: T 1 -1}. The proof of Lemma \ref{lem: AS sum trace t-tp} (and Lemma \ref{lem: AS sum trace t-tp hardy}) involves writing $\BTt_f$ (resp. $T_f$) as the sum of $\BTtpt_f$ (resp. $T^{(1,-1)}_f$) and some perturbation operators.
	
	\begin{defn}\label{defn: X Y}
		For $t>-1$, define
		\begin{equation}
			X^{(t)}_w(z)=\frac{t+1}{n+t+1}\BGt_n1(|w|^2)\bar{R}_w\BKt_w(z)-\BKtp_w(z),
		\end{equation}
		and
		\begin{equation}
			Y^{(t)}_w(z)=\frac{t+1}{n+t+1}\BGt_n1(|w|^2)\BKt_w(z).
		\end{equation}
		For a symbol function $g$, formally define the integral operators
		\[
		X^{(t)}_{g}h(z)=\int_{\bn}g(w)h(w)X^{(t)}_w(z)\intd\lambda_{t+1}(w);\quad Y^{(t)}_gh(z)=\int_{\bn}g(w)h(w)Y^{(t)}_w(z)\intd\lambda_{t+1}(w).
		\]
		It will be clear in subsequent proof that for $f\in\mathscr{C}^2(\overline{\bn})$, $X^{(t)}_f$ and $Y^{(t)}_{\bar{R}f}$ define bounded operators on $\bert$ (see Lemma \ref{lem: B for bergman} and Lemma \ref{lem: B X Y}).
	\end{defn}

	\begin{lem}\label{lem: Ttf=X+Y+K}
		Suppose $f\in\mathscr{C}^2(\overline{\bn})$. Then for any $t>-1$,
		\begin{equation}\label{eqn: Ttf=X+Y+K}
			\BTt_f=\BTtpt_f+X^{(t)}_{f}+Y^{(t)}_{\bar{R}f}+f(0)E_0,
		\end{equation}
		where $E_0h=h(0)$ is a rank one operator.
	\end{lem}
	
	\begin{proof}
		For any $h\in\Hol(\overline{\bn})$, by Lemma \ref{lem: formula R on bn}, we compute $\BTt_f$ as follows, 
		\begin{flalign*}
			\BTt_fh(z)=&\int_{\bn}f(w)h(w)\BKt_w(z)\intd\lambda_t(w)\\
			=&f(0)h(0)+\frac{t+1}{n+t+1}\int_{\bn}\BGt_n1(|w|^2)\bar{R}f(w)h(w)\BKt_w(z)\intd\lambda_{t+1}(w)\\
			&+\frac{t+1}{n+t+1}\int_{\bn}\BGt_n1(|w|^2)f(w)h(w)\bar{R}\BKt_w(z)\intd\lambda_{t+1}(w)\\
			=&f(0)\big(E_0f\big)(z)+Y^{(t)}_{\bar{R}f}h(z)+X^{(t)}_fh(z)+\BTtpt_fh(z).
		\end{flalign*}
		This completes the proof of Lemma \ref{lem: Ttf=X+Y+K}.
	\end{proof}

In the case of the Hardy space, we need to apply integration by parts twice to make our arguments work.
\begin{defn}\label{defn: X Y hardy}
	Let $\phi(s)=s^{-n}$.
	Define
	\[
	X^{(-1)}_w(z)=\frac{1}{n(n+1)}\mathcal{G}_n^{(0)}\phi(|w|^2)\bar{R}^2_wK_w(z)-K_w^{(1)}(z),
	\]
	\[
	Y^{(-1)}_w(z)=\frac{2}{n(n+1)}\mathcal{G}_n^{(0)}\phi(|w|^2)\bar{R}_wK_w(z),
	\]
	and
	\[
	Z^{(-1)}_w(z)=\frac{1}{n(n+1)}\mathcal{G}_n^{(0)}\phi(|w|^2)K_w(z).
	\]
	Formally define the symboled integral operators
	\[
	X^{(-1)}_fh(z)=\int_{\bn}h(w)f(w)X^{(-1)}_w(z)\intd\lambda_1(w),
	\]
	\[
	Y^{(-1)}_fh(z)=\int_{\bn}h(w)f(w)Y^{(-1)}_w(z)\intd\lambda_1(w),
	\]
	and
	\[
	Z^{(-1)}_fh(z)=\int_{\bn}h(w)f(w)Z^{(-1)}_w(z)\intd\lambda_1(w).
	\]
	Again, it will be clear later that for $f\in\mathscr{C}^2(\overline{\bn})$, $X^{(-1)}_f$, $Y^{(-1)}_{\bar{R}f}$ and $Z^{(-1)}_{\bar{R}^2f}$ define bounded operators on $H^2(\sn)$ (see Lemma \ref{lem: X Y Z -1}).
\end{defn}
\begin{lem}\label{lem: T=X+Y+Z+E}
	Suppose $f\in\mathscr{C}^2(\overline{\bn})$. Then
	\[
	T_f=T^{(1,-1)}_f+X^{(-1)}_f+Y^{(-1)}_{\bar{R}f}+Z^{(-1)}_{\bar{R}^2f}+f(0)E_0.
	\]
\end{lem}
	
Taking $\alpha=\beta=0$ and $z=0$ in Lemma \ref{lem: BM hardy}, we get the following.
\begin{lem}\label{lem: formula Sn R 0}
	Suppose $v\in\mathscr{C}^1(\overline{\bn})$. Then
	\begin{equation}\label{eqn: formula Sn R 0}
		\int_{\sn}v(w)\frac{\intd\sigma(w)}{\sigma_{2n-1}}=v(0)+\frac{1}{n}\int_{\bn}|w|^{-2n}\bar{R}v(w)\intd\lambda_0(w).
	\end{equation}
\end{lem}

Also, taking $\alpha=\beta=0, t=0, z=0$ in Lemma \ref{lem: formula Bn} gives the following.
\begin{lem}\label{lem: formula Bn R 0}
	Suppose $\phi:(0,1)\to[0,\infty)$ is measurable and $v\in\mathscr{C}^1(\bn)$. Then whenever all integrals converge absolutely, we have
	\begin{flalign}\label{eqn: formula Bn R 0}
		&\int_{\bn}\phi(|w|^2)v(w)\intd\lambda_0(w)\\
		=&\begin{cases}
		\frac{\mathcal{F}^{(0)}_n\phi(0)}{B(n,t+1)}v(0)+\frac{1}{n+1}\int_{\bn}\mathcal{G}^{(0)}_n\phi(|w|^2)\bar{R}v(w)\intd\lambda_1(w),&\mathcal{F}^{(0)}_n\phi(0)<\infty\\
		\frac{1}{n+1}\int_{\bn}\mathcal{G}^{(0)}_n\phi(|w|^2)\bar{R}v(w)\intd\lambda_1(w),&\mathcal{F}^{(0)}_n\phi(0)\leq\infty,~v(0)=0
		\end{cases}.\nonumber
	\end{flalign}
\end{lem}

\begin{proof}[{\bf Proof of Lemma \ref{lem: T=X+Y+Z+E}}]
	By Lemma \ref{lem: formula Sn R 0} and Lemma \ref{lem: formula Bn R 0}, for $h\in\Hol(\overline{\bn})$, we compute $T_f$ as follows,
	\begin{flalign*}
		T_fh(z)=&\int_{\sn}h(w)f(w)K_w(z)\frac{\intd\sigma(w)}{\sigma_{2n-1}}\\
		\xlongequal{\eqref{eqn: formula Sn R 0}}&f(0)h(0)+\frac{1}{n}\int_{\bn}\phi(|w|^2)\bar{R}\bigg[h(w)f(w)K_w(z)\bigg]\intd\lambda_0(w)\\
		\xlongequal{\eqref{eqn: formula Bn R 0}}&f(0)h(0)+\frac{1}{n(n+1)}\int_{\bn}\mathcal{G}_n^{(0)}\phi(|w|^2)\bar{R}^2\bigg[h(w)f(w)K_w(z)\bigg]\intd\lambda_1(w)\\
		=&f(0)h(0)+\frac{1}{n(n+1)}\int_{\bn}\mathcal{G}_n^{(0)}\phi(|w|^2)h(w)\intd\lambda_1(w)\\
		&\qquad \qquad \qquad \bigg[\bar{R}^2f(w)K_w(z)+2\bar{R}f(w)\bar{R}_wK_w(z)+f(w)\bar{R}_w^2K_w(z)\bigg]\\
		=&f(0)\big(E_0f\big)(z)+T^{(1,-1)}_fh(z)+X^{(-1)}_fh(z)+Y^{(-1)}_{\bar{R}f}h(z)+Z^{(-1)}_{\bar{R}^2f}h(z).
	\end{flalign*}
	This completes the proof of Lemma \ref{lem: T=X+Y+Z+E}.
\end{proof}

In view of \eqref{eqn: Ttf=X+Y+K}, Lemma \ref{lem: AS sum trace t-tp} essentially says that the trace
\[
\Tr[\BTtpt_{f_1},\BTtpt_{f_2},\ldots, \BTtpt_{f_{2n}}]
\]
if exist, is invariant under the perturbations of $X^{(t)}_{f_i}$, $Y^{(t)}_{\bar{R}f_i}$ and $f_i(0)E_0$. Similarly, Lemma \ref{lem: AS sum trace t-tp hardy} can also be interpreted as the stability of trace under certain perturbations.

	\subsection{Hypotheses A}\label{subsec: hypotheses A}
	In this subsection, let us assume the following.
	
	~
		
	\noindent{\bf Hypotheses A:} Suppose $A_1, A_2, \ldots, A_{2n}, B_1, B_2, \ldots, B_{2n}$ are bounded linear operators on a Hilbert space $\mathcal{H}$. Denote $C_i=A_i+B_i, i=1, \ldots, 2n$. The operators satisfy the following properties.
	\begin{enumerate}
		\item[(1)] For any $i=1, \ldots, 2n$,
		\[
		B_i\in\mathcal{S}^p,\quad\forall p>n.
		\]
		\item[(2)] For any $i, j=1,\ldots, 2n$,
		\[
		[A_i, A_j]\in\mathcal{S}^p,\quad\forall p>n.
		\]
		\item[(3)] For any $i, j=1, \ldots, 2n,$
		\[
		[A_i, B_j]\in\mathcal{S}^p,\quad\text{ for some } p<n.
		\]
	\end{enumerate}
	
	~
	The goal of this subsection is to prove the following.
	\begin{prop}\label{prop: A}
		Assume Hypotheses A. Then the operator
		\[
		[C_1, C_2, \ldots, C_{2n}]-[A_1, A_2, \ldots, A_{2n}]
		\]
		is in the trace class. Moreover, the operator
		\[
		[C_1, C_2, \ldots, C_{2n}]-[A_1, A_2, \ldots, A_{2n}]-\sum_{k=1}^{2n}[A_1,\ldots,A_{k-1},B_k,A_{k+1},\ldots,A_{2n}]
		\]
		is a trace class operator with zero trace.
	\end{prop}
	
	\begin{lem}\label{lem: 2B trace zero}
		Suppose $\{X_1, X_2, \ldots, X_{2n}\}$ is a subset of $\{A_1,\ldots,A_{2n}, B_1,\ldots, B_{2n}\}$, and at least two of $X_1, X_2,\ldots, X_{2n}$ are in $\{B_1, B_2, \ldots, B_{2n}\}$. Then the operator
		\[
		[X_1, X_2,\ldots, X_{2n}]
		\]
		is in the trace class of zero trace.
	\end{lem}
	
	\begin{proof}
		The lemma can be restated as follows: 
		
		``if $B\in\{B_1, B_2,\ldots, B_{2n}\}$, $\{X_1,\ldots,X_{2n-1}\}\subset\{A_1,\ldots,A_{2n}, B_1,\ldots,B_{2n}\}$, and at least one of $X_1,\ldots,X_{2n-1}$ is in $\{B_1,\ldots,B_{2n}\}$, then the operator
		\[
		[B, X_1,\ldots, X_{2n-1}]
		\]
		is in the trace class with zero trace.''
		
		Under the above assumption on $B$ and $X_1, \cdots, X_{2n-1}$, we compute the antisymmetrization $[B,X_1,\ldots,X_{2n-1}]$ as follows.
		\begin{flalign*}
			&[B,X_1,\ldots,X_{2n-1}]\\
			=&\sum_{k=1}^n\bigg(\sum_{\tau\in S_{2n-1}}\sgn(\tau)X_{\tau_1}\ldots X_{\tau_{2k-2}}BX_{\tau_{2k-1}}\ldots X_{\tau_{2n-1}}-\sum_{\tau\in S_{2n-1}}\sgn(\tau)X_{\tau_1}\ldots X_{\tau_{2k-1}}BX_{\tau_{2k}}\ldots X_{\tau_{2n-1}}\bigg)\\
			=&\frac{1}{2}\sum_{k=1}^n\bigg(\sum_{\tau\in S_{2n-1}}\sgn(\tau)[X_{\tau_1},X_{\tau_2}]\ldots X_{\tau_{2k-2}}BX_{\tau_{2k-1}}\ldots X_{\tau_{2n-1}}\\
			&\qquad\qquad \qquad -\sum_{\tau\in S_{2n-1}}\sgn(\tau)[X_{\tau_1},X_{\tau_2}]\ldots X_{\tau_{2k-1}}BX_{\tau_{2k}}\ldots X_{\tau_{2n-1}}\bigg)\\
			&\ldots\\
			=&2^{-n+1}\sum_{k=1}^n\bigg(\sum_{\tau\in S_{2n-1}}\sgn(\tau)[X_{\tau_1},X_{\tau_2}]\ldots[X_{\tau_{2k-3}},X_{\tau_{2k-2}}]BX_{\tau_{2k-1}}[X_{\tau_{2k}},X_{\tau_{2k+1}}]\ldots[X_{\tau_{2n-2}},X_{\tau_{2n-1}}]\\
			&~~~~-\sum_{\tau\in S_{2n-1}}\sgn(\tau)[X_{\tau_1},X_{\tau_2}]\ldots[X_{\tau_{2k-3}},X_{\tau_{2k-2}}]X_{\tau_{2k-1}}B[X_{\tau_{2k}},X_{\tau_{2k+1}}]\ldots[X_{\tau_{2n-2}},X_{\tau_{2n-1}}]\bigg)\\
			=&2^{-n+1}\sum_{k=1}^n\sum_{\tau\in S_{2n-1}}\sgn(\tau)\bigg([X_{\tau_1},X_{\tau_2}]\ldots[X_{\tau_{2k-3}},X_{\tau_{2k-2}}]BX_{\tau_{2k-1}}[X_{\tau_{2k}},X_{\tau_{2k+1}}]\ldots[X_{\tau_{2n-2}},X_{\tau_{2n-1}}]\\
			&~~~~-[X_{\tau_1},X_{\tau_2}]\ldots[X_{\tau_{2k-3}},X_{\tau_{2k-2}}]X_{\tau_{2k-1}}B[X_{\tau_{2k}},X_{\tau_{2k+1}}]\ldots[X_{\tau_{2n-2}},X_{\tau_{2n-1}}]\bigg).
		\end{flalign*}
		In this proof, we use the following notation. 
\begin{notation}		
		For operators $S, T$, write $S\sim T$ when $S-T$ is a trace class operator with zero trace.
\end{notation}		
		For each $k=1,\ldots, n$ and $\tau\in S_{2n-1}$, we {\bf claim} that
		\begin{flalign*}
			&[X_{\tau_1},X_{\tau_2}]\ldots[X_{\tau_{2k-3}},X_{\tau_{2k-2}}]BX_{\tau_{2k-1}}[X_{\tau_{2k}},X_{\tau_{2k+1}}]\ldots[X_{\tau_{2n-2}},X_{\tau_{2n-1}}]\\
			\sim&BX_{\tau_{2k-1}}[X_{\tau_{2k}},X_{\tau_{2k+1}}]\ldots[X_{\tau_{2n-2}},X_{\tau_{2n-1}}][X_{\tau_1},X_{\tau_2}]\ldots[X_{\tau_{2k-3}},X_{\tau_{2k-2}}],
		\end{flalign*}
		and
		\begin{flalign*}
			&[X_{\tau_1},X_{\tau_2}]\ldots[X_{\tau_{2k-3}},X_{\tau_{2k-2}}]X_{\tau_{2k-1}}B[X_{\tau_{2k}},X_{\tau_{2k+1}}]\ldots[X_{\tau_{2n-2}},X_{\tau_{2n-1}}]\\
			\sim&B[X_{\tau_{2k}},X_{\tau_{2k+1}}]\ldots[X_{\tau_{2n-2}},X_{\tau_{2n-1}}][X_{\tau_1},X_{\tau_2}]\ldots[X_{\tau_{2k-3}},X_{\tau_{2k-2}}]X_{\tau_{2k-1}}.
		\end{flalign*}
		If $X_{\tau_{2k-1}}\in\{B_1,B_2,\ldots,B_{2n}\}$, then we have
		\[
		B, X_{\tau_{2k-1}}\in\mathcal{S}^p, \forall p>n,\quad[X_{\tau_i},X_{\tau_j}]\in\mathcal{S}^p, \forall p>n, \forall i, j=1,\ldots,2n-1.
		\]
		By Remark \ref{rem: X1...Xn trace} it is easy to see that the claim holds. If $X_{\tau_{2k-1}}\in\{A_1,\ldots, A_{2n}\}$, then at least one of $X_{\tau_1},\ldots,X_{\tau_{2k-2}}, X_{\tau_{2k}},\ldots, X_{\tau_{2n-1}}$ is in $\{B_1, B_2,\ldots, B_{2n}\}$. Thus at least one of the commutators
		\[
		[X_{\tau_1},X_{\tau_2}],\ldots,[X_{\tau_{2k-3}},X_{\tau_{2k-2}}], [X_{\tau_{2k}},X_{\tau_{2k+1}}], \ldots, [X_{\tau_{2n-2}},X_{\tau_{2n-1}}]
		\]
		is in $\mathcal{S}^p$ for some $p<n$. Again, by Remark \ref{rem: X1...Xn trace}, the claim also follows. Thus in both cases the claim holds. By the claim, we compute $[B,X_1,\ldots,X_{2n-1}]$.
		\begin{flalign*}
			&[B,X_1,\ldots,X_{2n-1}]\\
			=&2^{-n+1}\sum_{k=1}^n\sum_{\tau\in S_{2n-1}}\sgn(\tau)\bigg([X_{\tau_1},X_{\tau_2}]\ldots[X_{\tau_{2k-3}},X_{\tau_{2k-2}}]BX_{\tau_{2k-1}}[X_{\tau_{2k}},X_{\tau_{2k+1}}]\ldots[X_{\tau_{2n-2}},X_{\tau_{2n-1}}]\\
			&~~~~-[X_{\tau_1},X_{\tau_2}]\ldots[X_{\tau_{2k-3}},X_{\tau_{2k-2}}]X_{\tau_{2k-1}}B[X_{\tau_{2k}},X_{\tau_{2k+1}}]\ldots[X_{\tau_{2n-2}},X_{\tau_{2n-1}}]\bigg)\\
			\sim&2^{-n+1}\sum_{k=1}^n\sum_{\tau\in S_{2n-1}}\sgn(\tau)BX_{\tau_{2k-1}}[X_{\tau_{2k}},X_{\tau_{2k+1}}]\ldots[X_{\tau_{2n-2}},X_{\tau_{2n-1}}][X_{\tau_1},X_{\tau_2}]\ldots[X_{\tau_{2k-3}},X_{\tau_{2k-2}}]\\
			&-2^{-n+1}\sum_{k=1}^n\sum_{\tau\in S_{2n-1}}\sgn(\tau)B[X_{\tau_{2k}},X_{\tau_{2k+1}}]\ldots[X_{\tau_{2n-2}},X_{\tau_{2n-1}}][X_{\tau_1},X_{\tau_2}]\ldots[X_{\tau_{2k-3}},X_{\tau_{2k-2}}]X_{\tau_{2k-1}}\\
			=&2^{-n+1}\sum_{k=1}^n\sum_{\tau\in S_{2n-1}}\sgn(\tau)BX_{\tau_1}[X_{\tau_2},X_{\tau_3}]\ldots[X_{\tau_{2n-2}},X_{\tau_{2n-1}}]\\
			&-2^{-n+1}\sum_{k=1}^n\sum_{\tau\in S_{2n-1}}\sgn(\tau)B[X_{\tau_1},X_{\tau_2}]\ldots[X_{\tau_{2n-3}},X_{\tau_{2n-2}}]X_{\tau_{2n-1}}\\
			=&\sum_{k=1}^n\sum_{\tau\in S_{2n-1}}\sgn(\tau)BX_{\tau_1}\ldots X_{\tau_{2n-1}}-\sum_{k=1}^n\sum_{\tau\in S_{2n-1}}\sgn(\tau)BX_{\tau_1}\ldots X_{\tau_{2n-1}}\\
			=&0.
		\end{flalign*}
	Here the third-to-last equality is because the antisymmetric sums $\sum_{\tau\in S_{2n-1}}\sgn(\tau)\cdots$ are invariant under any even permutation.
		This completes the proof of Lemma \ref{lem: 2B trace zero}.
	\end{proof}

A verbatim repetition of the proof of Lemma \ref{lem: 2B trace zero} also proves the following.
\begin{cor}\label{cor: 1 B improved}
Assume in addition to Hypotheses A that
\[
B_i\in\mathcal{S}^p,\quad\text{ for some}~ p<n.
\]
Suppose $\{X_1, X_2, \ldots, X_{2n}\}\subset\{A_1,\ldots, A_{2n}, B_1,\ldots, B_{2n}\}$, and at least one of $X_1, X_2,\ldots, X_{2n}\in\{B_1,\ldots, B_{2n}\}$. Then the operator
\[
[X_1, X_2,\ldots, X_{2n}]
\]
is in the trace class with zero trace. Consequently, the operator
\[
[C_1, C_2,\ldots, C_{2n}]-[A_1, A_2,\ldots, A_{2n}]
\]
is in the trace class with zero trace.
\end{cor}

\begin{rem}\label{rem: difference btw Toeplitz and pseudodifferential}
Notice that the assumption of Corollary \ref{cor: 1 B improved} is equivalent to the following.
\begin{enumerate}
	\item For $i, j=1,\ldots, 2n$,
	\[
	[A_i, A_j]\in\mathcal{S}^p,\quad\forall p>n.
	\]
	\item For $i=1,\ldots,2n$,
	\[
	B_i\in\mathcal{S}^p,\quad\text{ for some }p<n.
	\]
\end{enumerate}
Thus Corollary \ref{cor: 1 B improved} essentially says that the trace
\[
\Tr[A_1,A_2,\ldots, A_{2n}]
\]
is stable under any perturbation that belongs to $\mathcal{S}^p$ for some $p<n$. 
\end{rem}
	\begin{proof}[{\bf Proof of Proposition \ref{prop: A}}]
		By definition, we have the following equation
		\[
		[C_1, C_2,\ldots,C_{2n}]-[A_1, A_2,\ldots,A_{2n}]=\sum[X_1,X_2,\ldots,X_{2n}],
		\]
		where the sum is taken over all tuples $(X_1,\ldots,X_{2n})$ such that each $X_i$ belongs to $\{A_i, B_i\}$, and at least one of $X_1,\ldots, X_{2n}$ belongs to $\{B_1,B_2,\ldots,B_{2n}\}$. For this tuple $(X_1,X_2,\ldots,X_{2n})$, we have the following expansion
		\begin{flalign*}
			&[X_1,X_2,\ldots,X_{2n}]\\
			=&\sum_{\tau\in S_{2n}}\sgn(\tau)X_{\tau_1}X_{\tau_2}\ldots X_{\tau_{2n}}\\
			=&2^{-n}\sum_{\tau\in S_{2n}}\sgn(\tau)[X_{\tau_1},X_{\tau_2}][X_{\tau_3},X_{\tau_4}]\ldots[X_{\tau_{2n-1}},X_{\tau_{2n}}].
		\end{flalign*}
		By Hypotheses A, each commutator $[X_i,X_j]$ is at least in $\mathcal{S}^p$ for any $p>n$, and when $X_i\in\{B_1,B_2,\ldots,B_{2n}\}$, $[X_i,X_j]\in\mathcal{S}^p$ for some $p<n$, $\forall j=1,\ldots,2n$. Thus by Lemma \ref{lem: holder inequality for operators}, each product
		\[
		[X_{\tau_1},X_{\tau_2}][X_{\tau_3},X_{\tau_4}]\ldots[X_{\tau_{2n-1}},X_{\tau_{2n}}]
		\]
		is in the trace class. So we obtain the following estimate
		\[
		[C_1, C_2,\ldots,C_{2n}]-[A_1, A_2,\ldots,A_{2n}]\in\mathcal{S}^1.
		\]
		On the other hand, we have the following equation
		\begin{flalign*}
			&[C_1, C_2, \ldots, C_{2n}]-[A_1, A_2, \ldots, A_{2n}]-\sum_{k=1}^{2n}[A_1,\ldots,A_{k-1},B_k,A_{k+1},\ldots,A_{2n}]\\
			=&\sum[X_1,X_2,\ldots,X_{2n}],
		\end{flalign*}
		where each $X_i\in\{A_i, B_i\}$, and at least two of $X_1,X_2,\ldots,X_{2n}$ is in $\{B_1,B_2,\ldots,B_{2n}\}$. By Lemma \ref{lem: 2B trace zero}, each $[X_1,X_2,\ldots,X_{2n}]$ is a trace class operator of zero trace. Thus we conclude that
		\[
		[C_1, C_2, \ldots, C_{2n}]-[A_1, A_2, \ldots, A_{2n}]-\sum_{k=1}^{2n}[A_1,\ldots,A_{k-1},B_k,A_{k+1},\ldots,A_{2n}]
		\]
		is a trace class operator of zero trace. This completes the proof of Proposition \ref{prop: A}.
	\end{proof}

\subsection{Proof of Lemmas \ref{lem: AS sum trace t-tp} and \ref{lem: AS sum trace t-tp hardy}}
In this subsection, temporarily fix the notations
	\[
	A_i=\begin{cases}
		\BTtpt_{f_i},&\text{ if }t>-1\\
		T^{(1,-1)}_{f_i},&\text{ if }t=-1
	\end{cases},\quad
B_i=\begin{cases}
	X^{(t)}_{f_i}+Y^{(t)}_{\bar{R}f_i}+f_i(0)E_0,&\text{ if }t>-1\\
	X^{(-1)}_{f_i}+Y^{(-1)}_{\bar{R}f_i}+Z^{(-1)}_{\bar{R}^2f_i}+f_i(0)E_0,&\text{ if }t=-1
\end{cases},
	\]
	and $C_i=A_i+B_i$. Then by Lemma \ref{lem: Ttf=X+Y+K} and Lemma \ref{lem: T=X+Y+Z+E},
	\[
	C_i=\begin{cases}
	\BTt_{f_i},&\text{ if }t>-1,\\
	T_{f_i},&\text{ if }t=-1.
	\end{cases}
	\]
	Thus Lemmas \ref{lem: AS sum trace t-tp} and \ref{lem: AS sum trace t-tp hardy} are equivalent to the property that
	\[
	[C_1, C_2,\ldots, C_{2n}]-[A_1, A_2,\ldots, A_{2n}]
	\]
	is a trace class operator with zero trace.
	
	~
	
	Suppose $X_w(z), Y_w(z)$ are measurable functions on $\bn\times\bn$. We introduce the following hypotheses. 
	
	For $t>-1$, we say that $X_w(z), Y_w(z)$ satisfy {\bf Hypotheses B at $t$}, if there are $\epsilon>0$ and $C>0$ such that the following hold.
	\begin{enumerate}
		\item For each $w\in\bn$, $X_w(z), Y_w(z)$ is holomorphic in $z$, 
		\item $|X_w(z)|\leq C|w|^{-2n+\epsilon}|\BKt_w(z)|$, 
		\item $|Y_w(z)|\leq C|w|^{-2n-1+\epsilon}|\BKt_w(z)|$.
	\end{enumerate}
	For a function $f$ on $\bn$, formally define the symboled integral operators
	\begin{equation}\label{eqn: XY operator defn}
		X_fh(z)=\int_{\bn}h(w)f(w)X_w(z)\intd\lambda_{t+1}(w),\quad Y_fh(z)=\int_{\bn}h(w)f(w)Y_w(z)\intd\lambda_{t+1}(w).
	\end{equation}

	\begin{lem}\label{lem: B X Y}
	Assume Hypotheses B. Then for any $f\in\mathscr{C}^2(\overline{\bn})$, the integral operators $X_f$ and $Y_{\bar{R}f}$ define bounded operators on $\bert$ that belong to $\mathcal{S}^p$ for any $p>n$.
\end{lem}

\begin{proof}
	For $f\in\mathscr{C}^2(\overline{\bn})$, we have
	\[
	|f(w)|\lesssim 1,\quad |\bar{R}f(w)|\lesssim|w|.
	\]
	Thus by Hypotheses B, we obtain the following estimates,
	\[
	|f(w)X_w(z)|\lesssim|w|^{-2n+\epsilon}|\BKt_w(z)|,\quad|\bar{R}f(w)Y_w(z)|\lesssim|w|^{-2n+\epsilon}|\BKt_w(z)|.
	\]
	Thus for any $h\in\Hol(\overline{\bn})$, $X_fh(z), Y_{\bar{R}f}h(z)$ are defined pointwise. If $f$ has compact support contained in $\bn$, then it is easy to see that $X_f, Y_{\bar{R}f}$ belong to $\mathcal{S}^p$ for any $p$. In general, we can always write $f=f_1+f_2$, where $f_1$ has compact support in $\bn$, and the support of $f_2$ is away from the origin. We might as well assume that $f$ itself has support away from zero. In this case, we have the following estimates
	\[
	|f(w)X_w(z)|\lesssim|\BKt_w(z)|,\quad|\bar{R}f(w)Y_w(z)|\lesssim|\BKt_w(z)|.
	\]
	Denote $G_w(z)$ to be either $f(w)X_w(z)$ or $\bar{R}f(w)Y_w(z)$. Then $G_w(z)$ is holomorphic in $z$, and
	\[
	|G_w(z)|\lesssim|\BKt_w(z)|.
	\]
	Define
	\[
	T_Gh(z)=\int_{\bn}h(w)G_w(z)\intd\lambda_{t+1}(w).
	\]
	Then $T_G$ equals either $X_f$ or $Y_{\bar{R}f}$. Split the map as follows.
	\[
	T_G: \bert\xrightarrow{E_{t,2}}L_{a,t+2}^2(\bn)\xrightarrow{\hat{T}_G}\bert,
	\]
	where $\hat{T}_G$ is defined by the same integral formula as $T_G$. By Lemma \ref{lem: bounded operator using Schur's test}, $\hat{T}_G$ is bounded. By Lemma \ref{lem: Embedding Schatten membership}, $E_{t,2}\in\mathcal{S}^p$ for all $p>n$. Therefore $T_G\in\mathcal{S}^p$ for all $p>n$. This completes the proof of Lemma \ref{lem: B X Y}.
\end{proof}

	We postpone the proof of the following proposition to the end of this section.
	\begin{prop}\label{prop: B}
		Suppose $t>-1$ and $X_w(z), Y_w(z)$ satisfy Hypotheses B at $t$. Suppose  that $f_1, f_2, \ldots, f_{2n}\in\mathscr{C}^2(\overline{\bn})$, and $K_1, K_2, \ldots, K_{2n}$ are bounded operators on $\bert$ that belong to $\mathcal{S}^p$ for some $p<n$. Denote
		\[
		\hat{A}_i=\BTtpt_{f_i},\quad \hat{B}_i=X_{f_i}+Y_{\bar{R}f_i}+K_i,\quad\hat{C}_i=\hat{A}_i+\hat{B}_i,\quad i=1,\ldots, 2n.
		\]
		Then the operator
		\[
		[\hat{C}_1, \hat{C}_2,\ldots, \hat{C}_{2n}]-[\hat{A}_1, \hat{A}_2, \ldots, \hat{A}_{2n}]
		\]
		is a trace class operator with zero trace on $\bert$.
	\end{prop}

\begin{lem}\label{lem: B for bergman}
	For $t>-1$, the integral kernels $X^{(t)}_w(z)$ and $Y^{(t)}_w(z)$ satisfy Hypotheses B at $t$.
\end{lem}

\begin{proof}
	The fact that $X^{(t)}_w(z)$ and $Y^{(t)}_w(z)$ are holomorphic in $z$ is obvious from Definition \ref{defn: X Y}.
	By direct computation, we obtain the following equation
	\[
	\bar{R}_w\BKt_w(z)=(n+1+t)\frac{\la z,w\ra}{(1-\la z,w\ra)^{n+2+t}}=(n+1+t)\BKtp_w(z)-(n+1+t)\BKt_w(z).
	\]
	So we obtain the following estimates,
	\begin{equation}\label{eqn: temp 5}
		|\bar{R}_w\BKt_w(z)|\lesssim|w||\BKtp_w(z)|,\quad\big|\bar{R}_w\BKt_w(z)-(n+1+t)\BKtp_w(z)\big|\lesssim|\BKt_w(z)|.
	\end{equation}
	Also by Lemma \ref{lem: formula Bn R 0}, we get the following inequalities
	\begin{equation}\label{eqn: temp 6}
		|\BGt_n1(s)|\lesssim s^{-n},\quad\big|\BGt_n1(s)-\frac{1}{t+1}\big|\lesssim s^{-n}(1-s).
	\end{equation}
	By \eqref{eqn: temp 5} and \eqref{eqn: temp 6}, we read
	\begin{flalign*}
		\big|Y^{(t)}_w(z)\big|\lesssim&|w|^{-2n}|\BKt_w(z)|,
	\end{flalign*}
	and
	\begin{flalign*}
		\big|X^{(t)}_w(z)\big|=&\bigg|\bigg((t+1)\BGt_n1(|w|^2)-1\bigg)\cdot\frac{\bar{R}_w\BKt_w(z)}{n+t+1}+\bigg(\frac{\bar{R}_w\BKt_w(z)}{n+t+1}-\BKtp_w(z)\bigg)\bigg|\\
		\leq&\bigg|(t+1)\BGt_n1(|w|^2)-1\bigg|\cdot\bigg|\frac{\bar{R}_w\BKt_w(z)}{n+t+1}\bigg|+\bigg|\frac{\bar{R}_w\BKt_w(z)}{n+t+1}-\BKtp_w(z)\bigg|\\
		\lesssim&|w|^{-2n}(1-|w|^2)\cdot|w|\big|\BKtp_w(z)\big|+\big|\BKt_w(z)\big|\\
		\lesssim&|w|^{-2n+1}\big|\BKt_w(z)\big|.
	\end{flalign*}
	This completes the proof of Lemma \ref{lem: B for bergman}.
\end{proof}

\begin{proof}[{\bf Proof of Lemma \ref{lem: AS sum trace t-tp}}]
	Take $K_i=f_i(0)E_0$. Then each $K_i$ is a rank one operator. Take $X_w(z)=X^{(t)}_w(z)$ and $Y_w(z)=Y^{(t)}_w(z)$ and define $\hat{A}_i$, $\hat{B}_i$, $\hat{C}_i$ as in Proposition \ref{prop: B}. Then clearly they satisfy
	\[
A_i=\hat{A}_i,\quad B_i=\hat{B}_i,\quad C_i=\hat{C}_i,\quad i=1,\ldots,2n.
	\]
	By Lemma \ref{lem: B for bergman}, $X_w(z)$ and $Y_w(z)$ satisfy Hypotheses B at $t$. Thus by Proposition \ref{prop: B}, we verify
	\begin{flalign*}
&[\BTt_{f_1},\BTt_{f_2},\ldots,\BTt_{f_{2n}}]-[\BTtpt_{f_1},\BTtpt_{f_2},\ldots,\BTtpt_{f_{2n}}]\\
=&[C_1, C_2,\ldots, C_{2n}]-[A_1, A_2,\ldots, A_{2n}]\\
=&[\hat{C}_1, \hat{C}_2,\ldots, \hat{C}_{2n}]-[\hat{A}_1, \hat{A}_2, \ldots, \hat{A}_{2n}]
	\end{flalign*}
	is a trace class operator with zero trace. This completes the proof of Lemma \ref{lem: AS sum trace t-tp}.
\end{proof}

Next, we prove Lemma \ref{lem: AS sum trace t-tp hardy}.

	\begin{lem}\label{lem: XYZ kernel estimates}
	The integral kernels $X^{(-1)}_w(z), Y^{(-1)}_w(z)$ and $Z^{(-1)}_w(z)$ satisfy the following estimates.
	\begin{enumerate}
		\item $|X^{(-1)}_w(z)|\lesssim |w|^{-2n+1/2}|K^{(0)}_w(z)|$;
		\item $|Y^{(-1)}_w(z)|\lesssim |w|^{-2n+1/2}|K^{(0)}_w(z)|$;
		\item $|Z^{(-1)}_w(z)|\lesssim|w|^{-2n-1/2}|K_w(z)|$.
	\end{enumerate}
	In particular, $X^{(-1)}_w(z), Y^{(-1)}_w(z)$ satisfy Hypotheses B at $t=0$.
\end{lem}
\begin{proof}
	We compute the following limit by L'Hospital's rule,
	\[
	\lim_{s\to1^-}\frac{\mathcal{G}_n^{(0)}\phi(s)-1}{1-s}=\lim_{s\to1^-}\frac{\int_s^1r^{-1}\intd r-(1-s)}{(1-s)^2}=\lim_{s\to1^-}\frac{-s^{-1}+1}{-2(1-s)}=\frac{1}{2}.
	\]
	Also by \cite[Lemma 8.2]{TWZ:semicommutator} with $a=n+\frac{1}{4}$, $\mathcal{G}_n^{(0)}\phi(s)\lesssim s^{-n-\frac{1}{4}}$. Therefore we obtain the estimates
	\begin{equation}\label{eqn: temp Gn0phi estimate}
		\bigg|\mathcal{G}_n^{(0)}\phi(s)-1\bigg|\lesssim s^{-n-\frac{1}{4}}(1-s),\quad\big|\mathcal{G}_n^{(0)}\phi(s)\big|\lesssim s^{-n-\frac{1}{4}}.
	\end{equation}
	Also, direct computation gives the following identities
	\[
	\bar{R}_wK_w(z)=\frac{n\la z,w\ra}{(1-\la z,w\ra)^{n+1}}=nK^{(0)}_w(z)-nK_w(z),
	\]
	\[
	\bar{R}^2_wK_w(z)=\frac{n(n+1)\la z,w\ra}{(1-\la z,w\ra)^{n+2}}-\frac{n^2\la z,w\ra}{(1-\la z,w\ra)^{}}=n(n+1)K_w^{(1)}(z)-n(2n+1)K_w^{(0)}(z)+n^2K_w(z).
	\]
	So we have the following estimates, 
	\begin{equation}\label{eqn: temp R^2 Kwz-Kwz estimate}
		\bigg|\bar{R}_w^2K_w(z)-n(n+1)K_w^{(1)}(z)\bigg|\lesssim|K_w^{(0)}(z)|, 
	\end{equation}
	and
	\begin{equation}\label{eqn: temp R R^2 Kwz estimates}
		|\bar{R}_w^2K_w(z)|\lesssim|w||K^{(1)}_w(z)|,\quad |\bar{R}_wK_w(z)|\lesssim|w||K_w^{(0)}(z)|.
	\end{equation}
	By \eqref{eqn: temp Gn0phi estimate}, \eqref{eqn: temp R^2 Kwz-Kwz estimate} and \eqref{eqn: temp R R^2 Kwz estimates}, we estimate $|X^{(-1)}_w(z)|, |Y^{(-1)}_w(z)|, |Z^{(-1)}_w(z)|$ as follows,
	\begin{flalign*}
		|X^{(-1)}_w(z)|\leq&\frac{1}{n(n+1)}\bigg|\big(\mathcal{G}_n^{(0)}\phi(|w|^2)-1\big)\bar{R}^2_wK_w(z)\bigg|+\bigg|\frac{1}{n(n+1)}\bar{R}^2_wK_w(z)-K_w^{(1)}(z)\bigg|\\
		\lesssim&|w|^{-2n-\frac{1}{2}}(1-|w|^2)|w||K_w^{(1)}(z)|+|K_w^{(0)}(z)|\\
		\lesssim&|w|^{-2n+\frac{1}{2}}|K_w^{(0)}(z)|,
	\end{flalign*}
	\begin{flalign*}
		|Y^{(-1)}_w(z)|\lesssim&|w|^{-2n-\frac{1}{2}}|w||K_w^{(0)}(z)|=|w|^{-2n+\frac{1}{2}}|K_w^{(0)}(z)|,
	\end{flalign*}
	and
	\begin{flalign*}
		|Z^{(-1)}_w(z)|\lesssim&|w|^{-2n-\frac{1}{2}}|K_w(z)|.
	\end{flalign*}
	This completes the proof of Lemma \ref{lem: XYZ kernel estimates}.
\end{proof}

By Lemma \ref{lem: XYZ kernel estimates} and Lemma \ref{lem: B X Y}, for $f\in\mathscr{C}^2(\overline{\bn})$, the integral formulas of $X^{(-1)}_f$ and $Y^{(-1)}_{\bar{R}f}$ define bounded operators on $L_{a,0}^2(\bn)$. Denote these operators $X^{(-1,0)}_f$ and $Y^{(-1,0)}_{\bar{R}f}$.

\begin{lem}\label{lem: Z -1 0}
Suppose $f, g\in\mathscr{C}^2(\overline{\bn})$. Then the integral formula of $Z^{(-1)}_{\bar{R}^2f}$ defines a bounded operator on $L_{a,0}^2(\bn)$ that belongs to $\mathcal{S}^p$ for any $p>\frac{2n}{3}$. Denote this operator to be $Z^{(-1,0)}_{\bar{R}^2f}$.
\end{lem}
\begin{proof}
Split the map as follows.
\[
Z^{(-1,0)}_{\bar{R}^2f}: L_{a,0}^2(\bn)\xrightarrow{E_{0,3-2\epsilon}}L_{a,3-2\epsilon}^2(\bn)\xrightarrow{\hat{Z}^{(-1,0)}_{\bar{R}^2f}}L_a^2(\bn).
\]
Here $\hat{Z}^{(-1,0)}_{\bar{R}^2f}$ is defined by the same integral formula as $Z^{(-1,0)}_{\bar{R}^2f}$, and $\epsilon>0$ is any sufficiently small number. The boundedness of $\hat{Z}^{(-1,0)}_{\bar{R}^2f}$ follows from Lemma \ref{lem: XYZ kernel estimates} and Lemma \ref{lem: bounded operator using Schur's test}. Finally, by Lemma \ref{lem: Embedding Schatten membership}, $E_{0,3-2\epsilon}\in\mathcal{S}^p$ for any $p>\frac{n}{\frac{3}{2}-\epsilon}$. Since $\epsilon>0$ is arbitrary, we have $Z^{(-1,0)}_{\bar{R}^2f}\in\mathcal{S}^p$ for any $p>\frac{2n}{3}$. This completes the proof of Lemma \ref{lem: Z -1 0}.
\end{proof}

By Lemma \ref{lem: Ttf=X+Y+K}, for $f\in\mathscr{C}^2(\overline{\bn})$, the operator
\[
T^{-1,0}_f=T^{(1,0)}_f+X^{(-1,0)}_f+Y^{(-1,0)}_{\bar{R}f}+Z^{(-1,0)}_{\bar{R}^2f}+f(0)E_0
\]
is a well-defined bounded operator on $L_{a,0}^2(\bn)$. It follows from Lemma \ref{lem: XYZ kernel estimates}, Lemma \ref{lem: Z -1 0} and Proposition \ref{prop: B} that the following holds.
\begin{lem}\label{lem: t-tp hardy on bergman}
	Suppose $f_1, f_2,\ldots, f_{2n}\in\mathscr{C}^2(\overline{\bn})$. Then the operator on $L_{a,0}^2(\bn)$,
	\[
	[T_{f_1}^{(-1,0)},T^{(-1,0)}_{f_2},\ldots,T^{(-1,0)}_{f_{2n}}]-[T^{(1,0)}_{f_1},T^{(1,0)}_{f_2},\ldots,T^{(1,0)}_{f_{2n}}]
	\]
	is a trace class operator of zero trace.
\end{lem}

\begin{lem}\label{lem: X Y Z -1}
	Suppose $f, g\in\mathscr{C}^2(\overline{\bn})$.
	\begin{enumerate}
		\item The operators $X^{(-1)}_f, Y^{(-1)}_{\bar{R}f}$ on $H^2(\sn)$ belong to $\mathcal{S}^p$ for any $p>n$.
		\item The operator $Z^{(-1)}_{\bar{R}^2f}$ on $H^2(\sn)$ belongs to $\mathcal{S}^p$ for any $p>\frac{n}{2}$.
	\end{enumerate}
\end{lem}
\begin{proof}
	If $f$ has compact support in $\bn$, then $X^{(-1)}_f, Y^{(-1)}_{\bar{R}f}, Z^{(-1)}_{\bar{R}^2f}$ belong to the trace class. Since have the decomposition $f=f_1+f_2$, where $f_1$ has compact support in $\bn$, and $f_2$ has support away from the origin, we might as well assume that the support of $f$ does not contain the origin. Then by Lemma \ref{lem: XYZ kernel estimates}, we obtain the following bounds
	\[
	|f(w)X^{(-1)}_w(z)|\lesssim|K_w^{(0)}(z)|,\quad|\bar{R}f(w)Y^{(-1)}_w(z)|\lesssim|K_w^{(0)}(z)|,\quad|\bar{R}^2Z^{(-1)}_w(z)|\lesssim|K_w(z)|.
	\]
	Split the maps as follows.
	\[
	X^{(-1)}_f: H^2(\sn)\xrightarrow{E_{-1,2}}L_{a,1}^2(\bn)\xrightarrow{\hat{X}^{(-1)}_f}H^2(\sn),
	\]
	\[
	Y^{(-1)}_{\bar{R}f}: H^2(\sn)\xrightarrow{E_{-1,2}}L_{a,1}^2(\bn)\xrightarrow{\hat{Y}^{(-1)}_{\bar{R}f}}H^2(\sn),
	\]
	\[
	Z^{(-1)}_{\bar{R}^2f}: H^2(\sn)\xrightarrow{E_{-1,4-2\epsilon}}L_{a,3-2\epsilon}\xrightarrow{\hat{Z}^{(-1)}_{\bar{R}^2f}}H^2(\sn).
	\]
	Here the operators with hats are defined by the same integral formulas as the corresponding operators without hats, and $\epsilon>0$ is any sufficiently small number. The boundedness of the operators with hats follow from the estimates in Lemma \ref{lem: XYZ kernel estimates} and Lemma \ref{lem: bounded operator to H2 using Schur's test}. Finally, by Lemma \ref{lem: Embedding Schatten membership}, the Schatten-class memberships of $X^{(-1)}_f, Y^{(-1)}_{\bar{R}f}$ and $Z^{(-1)}_{\bar{R}^2f}$ follow from those of the embedding operators. This completes the proof of Lemma \ref{lem: X Y Z -1}.
\end{proof}

\begin{rem}\label{rem: T 1 -1}
It follows from Lemma \ref{lem: T=X+Y+Z+E} and Lemma \ref{lem: X Y Z -1} that for any $f\in\mathscr{C}^2(\overline{\bn})$, the operator $T^{(1,-1)}_f=T_f-X^{(-1)}_f-Y^{(-1)}_{\bar{R}f}-Z^{(-1)}_{\bar{R}^2f}-f(0)E_0$ is a well-defined bounded operator on $H^2(\sn)$.
\end{rem}

\begin{lem}\label{lem: X Y commutators for hardy}
	Suppose $f, g\in\mathscr{C}^2(\overline{\bn})$. Then the commutators
	\[
	[X^{(-1)}_f, T^{(1,-1)}_g],\quad[Y^{(-1)}_{\bar{R}f}, T^{(1,-1)}_f]
	\]
	belong to $\mathcal{S}^p$ for any $p>\frac{2n}{3}$.
\end{lem}

\begin{proof}
	As in the proof of Lemma \ref{lem: X Y Z -1}, we may assume the support of $f$ is away from the origin. For any $h\in H^2(\sn),$ since $X^{(-1)}_w(z), Y^{(-1)}_w(z)$ are holomorphic in $z$, it is easy to verify the following formulas,
	\[
	\bigg(T^{(1,-1)}_gX^{(-1)}_f-X^{(-1)}_{gf}\bigg)h(\xi)=\int_{\bn^2}\big[g(z)-g(w)\big]f(w)h(w)K^{(1)}_z(\xi)X^{(-1)}_w(z)\intd\lambda_{1}(w)\intd\lambda_{1}(z),
	\]
	\[
	\bigg(X^{(-1)}_{f}T^{(1,-1)}_g-X^{(-1)}_{gf}\bigg)h(\xi)=\int_{\bn^2}f(z)\big[g(w)-g(z)\big]h(w)X^{(-1)}_z(\xi)K^{(1)}_w(z)\intd\lambda_{1}(w)\intd\lambda_{1}(z),
	\]
	and similarly,
	\[
	\bigg(T^{(1,-1)}_gY^{(-1)}_{\bar{R}f}-Y^{(-1)}_{g\bar{R}f}\bigg)h(\xi)=\int_{\bn^2}\big[g(z)-g(w)\big]\bar{R}f(w)h(w)K^{(1)}_z(\xi)Y^{(-1)}_w(z)\intd\lambda_{1}(w)\intd\lambda_{1}(z),
	\]
	\[
	\bigg(Y^{(-1)}_{\bar{R}f}T^{(1,-1)}_g-Y^{(-1)}_{g\bar{R}f}\bigg)h(\xi)=\int_{\bn^2}\bar{R}f(z)\big[g(w)-g(z)\big]h(w)Y^{(-1)}_z(\xi)K^{(1)}_w(z)\intd\lambda_{1}(w)\intd\lambda_{1}(z).
	\]	
	 If we denote $G_w(z)$ to be either $f(w)X^{(-1)}_w(z)$ or $\bar{R}f(w)Y^{(-1)}_w(z)$, then set
	\[
	\bigg(T^{(1,-1)}_gX^{(-1)}_f-X^{(-1)}_{gf}\bigg), \text{ or }\bigg(T^{(1,-1)}_gY^{(-1)}_{\bar{R}f}-Y^{(-1)}_{g\bar{R}f}\bigg)=H,
	\]
	\[
	\bigg(X^{(-1)}_{f}T^{(1,-1)}_g-X^{(-1)}_{gf}\bigg), \text{ or }\bigg(Y^{(-1)}_{\bar{R}f}T^{(1,-1)}_g-Y^{(-1)}_{g\bar{R}f}\bigg)=WS.
	\]
	Here
	\[
	Sh(z)=\int_{\bn}\big[g(w)-g(z)\big]h(w)K^{(1)}_w(z)\intd\lambda_{1}(w), \ 
	Wh(\xi)=\int_{\bn}h(z)G_z(\xi)\intd\lambda_{1}(z),
	\]
	and
	\[
	Hh(\xi)=\int_{\bn^2}[g(z)-g(w)]h(w)G_w(z)K^{(1)}_z(\xi)\intd\lambda_{1}(w)\intd\lambda_1(z)=\int_{\bn}h(w)H_w(z)\intd\lambda_1(w),
	\]
	where
	\[
	H_w(\xi)=\int_{\bn}[g(z)-g(w)]G_w(z)K^{(1)}_z(\xi)\intd\lambda_{1}(z).
	\]
	By Lemma \ref{lem: XYZ kernel estimates} and since we assume the support of $f$ does not contain the origin, we obtain the following bound
	\[
	|G_w(z)|\lesssim|K^{(0)}_w(z)|.
	\]
	By Lemma \ref{lem: Rudin Forelli generalizations}, we have the estimate
	\[
	|H_w(\xi)|\lesssim\int_{\bn}\frac{1}{|1-\la z,w\ra|^{n+1/2}|1-\la z,\xi\ra|^{n+2}}\intd\lambda_1(\xi)\lesssim\frac{1}{|1-\la w,\xi\ra|^{n+1/2}}.
	\]
	For any $\epsilon>0$ sufficiently small, split the map as follows.
	\[
	H: H^2(\sn)\xrightarrow{E_{-1,3-2\epsilon}}L_{a,2-2\epsilon}\xrightarrow{\hat{H}}H^2(\sn),
	\]
	\[
	WS: H^2(\sn)\xrightarrow{E_{-1, 3-2\epsilon}}L_{a,2-2\epsilon}^2(\bn)\xrightarrow{\hat{S}}L^2(\lambda_{1-2\epsilon})\xrightarrow{\hat{W}}H^2(\sn).
	\]
	By the estimates above, Lemma \ref{lem: bounded operator using Schur's test} and Lemma \ref{lem: bounded operator to H2 using Schur's test}, the operators above with hats are bounded. Thus by Lemma \ref{lem: Embedding Schatten membership}, $H$ and $WS$ belong to $\mathcal{S}^p$ for any $p>\frac{2n}{3}$. Therefore the commutators have the same Schatten-class membership. This completes the proof of Lemma \ref{lem: X Y commutators for hardy}.
\end{proof}

\begin{proof}[{\bf Proof of Lemma \ref{lem: AS sum trace t-tp hardy}}]
	By Lemma \ref{lem: t-tp hardy on bergman}, the operator
	\[
	[T^{(-1,0)}_{f_1},\ldots,T^{(-1,0)}_{f_{2n}}]-[T^{(1,0)}_{f_1},\ldots, T^{(1,0)}_{f_{2n}}]
	\]
	on $L_{a,0}^2(\bn)$ is a trace class operator with zero trace. On the other hand, by Lemma \ref{lem: X Y Z -1}, Lemma \ref{lem: X Y commutators for hardy} and Proposition \ref{prop: A}, the operator
	\[
	[T_{f_1},\ldots, T_{f_{2n}}]-[T^{(1,-1)}_{f_1},\ldots, T^{(1,-1)}_{f_{2n}}]
	\]
	is a trace class operator on $H^2(\sn)$. Clearly, we check the following equation
	\[
	\bigg([T^{(-1,0)}_{f_1},\ldots,T^{(-1,0)}_{f_{2n}}]-[T^{(1,0)}_{f_1},\ldots, T^{(1,0)}_{f_{2n}}]\bigg)\bigg|_{H^2(\sn)}=[T_{f_1},\ldots, T_{f_{2n}}]-[T^{(1,-1)}_{f_1},\ldots, T^{(1,-1)}_{f_{2n}}].
	\]
	Thus by Lemma \ref{lem: trace on two spaces}, we obtain the following equation
	\begin{flalign*}
		&\Tr\bigg([T_{f_1}, T_{f_2},\ldots, T_{f_{2n}}]-[T^{(1,-1)}_{f_1}, T^{(1,-1)}_{f_2},\ldots, T^{(1,-1)}_{f_{2n}}]\bigg)\\
		=&\Tr\bigg([T^{(-1,0)}_{f_1}, T^{(-1,0)}_{f_2},\ldots,T^{(-1,0)}_{f_{2n}}]-[T^{(1,0)}_{f_1}, T^{(1,0)}_{f_2},\ldots, T^{(1,0)}_{f_{2n}}]\bigg)\\
		=&0.
	\end{flalign*}
	This completes the proof of Lemma \ref{lem: AS sum trace t-tp hardy}.
\end{proof}

\subsection{Proof of Proposition \ref{prop: B}}	

\begin{lem}\label{lem: Ttpt commutator}
Suppose $t>-1$ and $f$, $g$ are Lipschitz functions on $\bn$. Then
\[
[\BTtpt_f, \BTtpt_g]\in\mathcal{S}^p,\quad\forall p>n.
\]
\end{lem}
\begin{proof}
Denote $P^{(t+1,t)}$ to be the restriction of $P^{(t+1)}$ to $L^2(\lambda_t)$. By Lemma \ref{lem: bounded operator using Schur's test} it is easy to see that $P^{(t+1,t)}$ is bounded. Obviously, $P^{(t+1,t)}$ satisfies the following identities.
\[
\BPt P^{(t+1,t)}=P^{(t+1,t)},\quad P^{(t+1,t)}P^{(t)}=P^{(t)},\quad \big(P^{t+1,t}\big)^2=P^{(t+1,t)}.
\]
So we compute the commutator as follows.
\begin{flalign*}
&[\BTtpt_f,\BTtpt_g]\\
=&[P^{(t+1,t)}M_f\BPt, P^{(t+1,t)}M_g\BPt]\\
=&P^{(t+1,t)}M_f P^{(t+1,t)}M_g\BPt-P^{(t+1,t)}M_g P^{(t+1,t)}M_f\BPt\\
=&P^{(t+1,t)}M_g\big(I-P^{(t+1,t)}\big)M_f\BPt-P^{(t+1,t)}M_f\big(I-P^{(t+1,t)}\big)M_g\BPt\\
=&P^{(t+1,t)}[M_g,P^{(t+1,t)}][M_f,P^{t+1,t}]\BPt-P^{(t+1,t)}[M_f,P^{(t+1,t)}][M_g,P^{t+1,t}]\BPt.
\end{flalign*}
By Corollary \ref{cor: commutator schatten membership no phi}, we arrive at the following property, 
\[
[M_g,P^{(t+1,t)}],\quad[M_f,P^{(t+1,t)}]\in\mathcal{S}^p,~\forall p>2n.
\]
The Schatten-$p$ membership of the commutator $[T^{(t+1, t)}_f, T^{(t+1, t)}_g]$ follows from the above property easily. This completes the proof of Lemma \ref{lem: Ttpt commutator}.
\end{proof}

	\begin{lem}\label{lem: B commutator}
		Assume Hypotheses B. Then for any $f, g\in\mathscr{C}^2(\overline{\bn})$, the commutators
		\[
		[X_f, \BTtpt_g] \text{ and }[Y_{\bar{R}f}, \BTtpt_g]
		\]
		belong to $\mathcal{S}^p$ for some $p<n$.
	\end{lem}
	
	\begin{proof}
		As in the proof of Lemma \ref{lem: B X Y}, we may assume the support of $f$ does not contain the origin. For any $h\in\bert,$ since $X_w(z), Y_w(z)$ are holomorphic in $z$, it is easy to verify the following integral expression,
		\[
		\bigg(\BTtpt_gX_f-X_{gf}\bigg)h(\xi)=\int_{\bn^2}\big[g(z)-g(w)\big]f(w)h(w)\BKtp_z(\xi)X_w(z)\intd\lambda_{t+1}(w)\intd\lambda_{t+1}(z),
		\]
		\[
		\bigg(X_{f}\BTtpt_g-X_{gf}\bigg)h(\xi)=\int_{\bn^2}f(z)\big[g(w)-g(z)\big]h(w)X_z(\xi)\BKtp_w(z)\intd\lambda_{t+1}(w)\intd\lambda_{t+1}(z),
		\]
		and similarly,
		\[
		\bigg(\BTtpt_gY_{\bar{R}f}-Y_{g\bar{R}f}\bigg)h(\xi)=\int_{\bn^2}\big[g(z)-g(w)\big]\bar{R}f(w)h(w)\BKtp_z(\xi)Y_w(z)\intd\lambda_{t+1}(w)\intd\lambda_{t+1}(z),
		\]
		\[
		\bigg(Y_{\bar{R}f}\BTtpt_g-Y_{g\bar{R}f}\bigg)h(\xi)=\int_{\bn^2}\bar{R}f(z)\big[g(w)-g(z)\big]h(w)Y_z(\xi)\BKtp_w(z)\intd\lambda_{t+1}(w)\intd\lambda_{t+1}(z).
		\]
		As in the proof of Lemma \ref{lem: B X Y}, let $G_w(z)$ be either $f(w)X_w(z)$ or  $\bar{R}f(w)Y_w(z)$. Write
		\[
		Th(z)=\int_{\bn}[g(z)-g(w)]h(w)G_w(z)\intd\lambda_{t+1}(w),
		\]
		\[
		Sh(z)=\int_{\bn}\big[g(w)-g(z)\big]h(w)\BKtp_w(z)\intd\lambda_{t+1}(w),
		\]
		and
		\[
		Wh(\xi)=\int_{\bn}h(z)G_z(\xi)\intd\lambda_{t+1}(z).
		\]
		As mappings, set
		\[
		\BTtpt_gX_f-X_{gf}, \text{ or } \BTtpt_gY_{\bar{R}f}-Y_{g\bar{R}f} =P^{(t+1)}T,
		\]
		\[
		X_{f}\BTtpt_g-X_{gf}, \text{ or } Y_{\bar{R}f}\BTtpt_g-Y_{g\bar{R}f} =WS.
		\]
		Take $\epsilon>0$ small enough. Split $P^{(t+1)}T$ and $WS$ as the composition of the following operators.
		\[
		P^{(t+1)}T: \bert\xrightarrow{E_{t,2+\epsilon}}L_{a,t+2+\epsilon}^2\xrightarrow{T}L^2(\lambda_t)\xrightarrow{P^{(t+1)}}\bert,
		\]
		\[
		WS: \bert\xrightarrow{E_{t,2+\epsilon}}L_{a,t+2+\epsilon}^2\xrightarrow{S}L^2(\lambda_{t+2})\xrightarrow{W}L^2(\lambda_t).
		\]
		By Lemma \ref{lem: Mobius basics}, we have the following inequalities, 
		\[
		|g(z)-g(w)|\lesssim|z-w|\lesssim|1-\la z,w\ra|^{1/2}.
		\]
		By Lemma \ref{lem: bounded operator using Schur's test}, $T, P^{(t+1)}, S, W$ define bounded operators between the spaces indicated above. Again, by Lemma \ref{lem: Embedding Schatten membership}, $E_{t,2+\epsilon}\in\mathcal{S}^p$ for some $p<n$. Thus altogether, the operators
		\[
		\bigg(\BTtpt_gX_f-X_{gf}\bigg), \bigg(\BTtpt_gY_{\bar{R}f}-Y_{g\bar{R}f}\bigg), \bigg(X_{f}\BTtpt_g-X_{gf}\bigg), \bigg(Y_{\bar{R}f}\BTtpt_g-Y_{g\bar{R}f}\bigg)
		\]
		are in $\mathcal{S}^p$ for some $p<n$. Thus so do the commutators
		\[
		[\BTtpt_g, X_f]=\bigg(\BTtpt_gX_f-X_{gf}\bigg)-\bigg(X_{f}\BTtpt_g-X_{gf}\bigg),
		\]
		and
		\[
		[\BTtpt_g, Y_{\bar{R}f}]=\bigg(\BTtpt_gY_{\bar{R}f}-Y_{g\bar{R}f}\bigg)-\bigg(Y_{\bar{R}f}\BTtpt_g-Y_{g\bar{R}f}\bigg).
		\]
		This completes the proof of Lemma \ref{lem: B commutator}.
	\end{proof}
	
	From Lemma \ref{lem: B X Y}, Lemma \ref{lem: Ttpt commutator} and Lemma \ref{lem: B commutator} it follows that if we set
	\[
	A_i=\BTtpt_{f_i},\quad B_i=X_{f_i}+Y_{\bar{R}f_i}+K_i,\quad i=1, 2,\ldots, 2n,
	\]
	then the operators $\{A_i, B_i\}$ satisfy Hypotheses A defined Subsection \ref{subsec: hypotheses A}. Thus by Proposition \ref{prop: A}, the proof of Proposition \ref{prop: B} reduces to the proof of
	\[
	\Tr\bigg(\sum_{k=1}^{2n}[A_1,\ldots,A_{k-1},B_k,A_{k+1},\ldots,A_{2n}]\bigg)=0.
	\]
	Each $B_k$ splits into the sum of $X_{f_k}$, $Y_{\bar{R}f_k}$ and $K_k$. The part with $K_k$ can be handled by Corollary \ref{cor: 1 B improved} as $K_k$ is assumed to belong to $\mathcal{S}^p$ for some $p<n$. Thus it remains to prove
	\begin{equation}\label{eqn: 6}
	\Tr\bigg(\sum_{k=1}^{2n}[A_1,\ldots,A_{k-1},X_{f_k}+Y_{\bar{R}f_k},A_{k+1},\ldots,A_{2n}]\bigg)=0.
	\end{equation}
	
	As explained in Remark \ref{rem: difference btw Toeplitz and pseudodifferential}, it is hard to handle this trace at operator-theoretic level. Instead, we need to treat them as integral operators. Let us explain the idea of the proof. The operator in \eqref{eqn: 6} is a sum of compositions of $2n$ integral operators. By \cite[Lemma 2.5]{TWZ:semicommutator} and the definitions of $A_i, X_{f_k}, Y_{\bar{R}f_k}$, we can write the trace in \eqref{eqn: 6} into a $(2n+1)$-fold integral of the form
	\[
	\int_{\bn}\bigg[\int_{\bn^{2n}}\BKt_\xi(z_1)G(z_1,z_2,\ldots, z_{2n},\xi)\intd\lambda_{t+1}(z_1)\ldots\intd\lambda_{t+1}(z_{2n})\bigg]\intd\lambda_t(\xi),
	\]
	where $G(z_1,z_2,\ldots, z_{2n},\xi)$ is holomorphic in $\xi$. If the $(2n+1)$-fold integral converges absolutely, then we can apply Fubini's Theorem and get that the above integral is equal to
	\[
	\int_{\bn^{2n}}G(z_1,z_2,\ldots, z_{2n},\xi)\big|_{\xi=z_1}\intd\lambda_{t+1}(z_1)\ldots\intd\lambda_{t+1}(z_{2n}).
	\]
	The anti-symmetrization will then tell us that the above equals zero as $\xi=z_1$. Thus the proof of Proposition \ref{prop: B} reduces  to reorganize the parts in \eqref{eqn: 6} so that each part converges absolutely.
	
	The following lemma helps us recognizing absolutely integrable terms in a multi-fold integral.
	\begin{lem}\label{lem: k kernels absolutely integrable}
		Suppose $k$ is a positive integer and $t>-1$. Then
		\begin{equation}\label{eqn: k kernels absolutely integrable}
			\int_{\bn^k}\bigg|\BKt_{z_1}(z_2)\ldots\BKt_{z_{k-1}}(z_k)\bigg|\intd\lambda_t(z_1)\ldots\intd\lambda_t(z_k)<\infty;
		\end{equation}
		and for any $\epsilon>0$,
		\begin{equation}\label{eqn: k+1 kernels absolutely integrable}
			\int_{\bn^{k+1}}\bigg(\sum_{i,j}|1-\la z_i,z_j\ra|\bigg)^{n+\epsilon}\bigg|\BKt_{z_0}(z_1)\BKt_{z_1}(z_2)\ldots\BKt_{z_{k-1}}(z_k)\BKt_{z_k}(z_0)\bigg|\intd\lambda_t(z_0)\ldots\intd\lambda_t(z_k)<\infty.
		\end{equation}
	\end{lem}
	
	\begin{proof}
		First, we notice that \eqref{eqn: k kernels absolutely integrable} is a special case of \eqref{eqn: k+1 kernels absolutely integrable}: take $\epsilon=t+1$, then
		\begin{flalign*}
			&\int_{\bn^k}\bigg|\BKt_{z_1}(z_2)\ldots\BKt_{z_{k-1}}(z_k)\bigg|\intd\lambda_t(z_1)\ldots\intd\lambda_t(z_k)\\
			=&\int_{\bn^k}|1-\la z_1,z_k\ra|^{n+t+1}\bigg|\BKt_{z_1}(z_2)\ldots\BKt_{z_{k-1}}(z_k)\BKt_{z_k}(z_1)\bigg|\intd\lambda_t(z_1)\ldots\intd\lambda_t(z_k)\\
			\leq&\int_{\bn^k}\bigg(\sum_{i,j}|1-\la z_i,z_j\ra|\bigg)^{n+t+1}\bigg|\BKt_{z_1}(z_2)\ldots\BKt_{z_{k-1}}(z_k)\BKt_{z_k}(z_1)\bigg|\intd\lambda_t(z_1)\ldots\intd\lambda_t(z_k).
		\end{flalign*}
		Thus it suffices to prove \eqref{eqn: k+1 kernels absolutely integrable}. It is well-known that all $l^p$ norms on a finite set are equivalent. So by \cite[Proposition 5.1.2]{Rudinbookunitball} and Lemma \ref{lem: Mobius basics} (7), we obtain the following estimates,
		\[
		|1-\la z_i,z_j\ra|=\bigg(|1-\la z_i,z_j\ra|^{1/2}\bigg)^2\lesssim\bigg(\sum_{s=0}^{k-1}|1-\la z_s,z_{s+1}\ra|^{1/2}\bigg)^2\lesssim\sum_{s=0}^{k-1}|1-\la z_s,z_{s+1}\ra|, \forall i, j,
		\]
		and
		\[
		\bigg(\sum_{i,j}|1-\la z_i,z_j\ra|\bigg)^{n+\epsilon}\lesssim\bigg(\sum_{s=0}^{k-1}|1-\la z_s,z_{s+1}\ra|\bigg)^{n+\epsilon}\lesssim\sum_{s=0}^{k-1}|1-\la z_s,z_{s+1}\ra|^{n+\epsilon}.
		\]	
		Also, for any $s\geq1$, take the rotation in variables $z_s\mapsto z_0, z_{s+1}\mapsto z_1, \ldots, z_k\mapsto z_{k-s}, z_0\mapsto z_{k-s+1}, z_1\mapsto z_{k-s+2}, \ldots, z_{s-1}\mapsto z_k$, then we have the following equation,
		\begin{flalign*}
			&\int_{\bn^{k+1}}|1-\la z_s,z_{s+1}\ra|^{n+\epsilon}\bigg|\BKt_{z_0}(z_1)\BKt_{z_1}(z_2)\ldots\BKt_{z_{k-1}}(z_k)\BKt_{z_k}(z_0)\bigg|\intd\lambda_t(z_0)\ldots\intd\lambda_t(z_k)\\
			=&\int_{\bn^{k+1}}|1-\la z_0,z_{1}\ra|^{n+\epsilon}\bigg|\BKt_{z_0}(z_1)\BKt_{z_1}(z_2)\ldots\BKt_{z_{k-1}}(z_k)\BKt_{z_k}(z_0)\bigg|\intd\lambda_t(z_0)\ldots\intd\lambda_t(z_k).
		\end{flalign*}
		Notice that the part of the reproducing kernels is invariant under such a change of variable.
		Thus it suffices to prove the integral above is finite. Without loss of generality, assume $0<\epsilon<1+t$. We prove it by induction. For $k=1$, by Lemma \ref{lem: Rudin Forelli generalizations} (1), we compute the integral as follows.
		\begin{flalign*}
			&\int_{\bn^2}|1-\la z_0,z_1\ra|^{n+\epsilon}\bigg|\BKt_{z_0}(z_1)\BKt_{z_1}(z_0)\bigg|\intd\lambda_t(z_0)\intd\lambda_t(z_1)\\
			=&\int_{\bn^2}\frac{1}{|1-\la z_0, z_1\ra|^{2(n+1+t)-n-\epsilon}}\intd\lambda_t(z_0)\intd\lambda_t(z_1)\\
			\lesssim&\int_{\bn}(1-|z_1|^2)^{\epsilon-(1+t)}\intd\lambda_t(z_1)\\
			\lesssim&\int_{\bn}(1-|z_1|^2)^{\epsilon-1}\intd m(z_1)\\
			<&\infty.
		\end{flalign*}
		Thus \eqref{eqn: k+1 kernels absolutely integrable} holds for $k=1$. Suppose \eqref{eqn: k+1 kernels absolutely integrable} holds for $k-1$. Then by Lemma \ref{lem: Rudin Forelli generalizations} (2), we have the following estimates by the induction step.
		\begin{flalign*}
			&\int_{\bn^{k+1}}|1-\la z_0,z_1\ra|^{n+\epsilon}\bigg|\BKt_{z_0}(z_1)\BKt_{z_1}(z_2)\ldots\BKt_{z_{k-1}}(z_k)\BKt_{z_k}(z_0)\bigg|\intd\lambda_t(z_0)\ldots\intd\lambda_t(z_k)\\
			=&\int_{\bn^{k}}\bigg(\int_{\bn}\frac{1}{|1-\la z_0,z_1\ra|^{1+t-\epsilon}|1-\la z_0,z_k\ra|^{n+1+t}}\intd\lambda_t(z_0)\bigg)\\
			&~~~~~~~~~\cdot\bigg|\BKt_{z_1}(z_2)\ldots\BKt_{z_{k-1}}(z_k)\bigg|\intd\lambda_t(z_1)\ldots\intd\lambda_t(z_k)\\
			\lesssim&\int_{\bn^k}\frac{1}{|1-\la z_1,z_k\ra|^{1+t-\epsilon/2}}\bigg|\BKt_{z_1}(z_2)\ldots\BKt_{z_{k-1}}(z_k)\bigg|\intd\lambda_t(z_1)\ldots\intd\lambda_t(z_k)\\
			=&\int_{\bn^k}|1-\la z_1,z_k\ra|^{n+\epsilon/2}\bigg|\BKt_{z_1}(z_2)\ldots\BKt_{z_{k-1}}(z_k)\BKt_{z_k}(z_1)\bigg|\intd\lambda_t(z_1)\ldots\intd\lambda_t(z_k)\\
			<&\infty.
		\end{flalign*}
		 This completes the proof of Lemma \ref{lem: k kernels absolutely integrable}.
	\end{proof}

	\begin{proof}[{\bf Proof of Proposition \ref{prop: B}}]
		For simplicity of notation, in this proof let us write $A_i$ for $\hat{A}_i$, $B_i$ for $\hat{B}_i$, and $C_i$ for $\hat{C}_i$.
		The fact that $\{A_i, B_i\}_{i=1}^{2n}$ satisfy Hypotheses A follows from Lemma \ref{lem: B X Y} and Lemma \ref{lem: B commutator}. Thus by Proposition \ref{prop: A}, we have the following property, i.e. 
		\[
		[C_1, C_2,\ldots, C_{2n}]-[A_1, A_2,\ldots, A_{2n}]\in\mathcal{S}^1,
		\]
		and
		\[
		[C_1, C_2, \ldots, C_{2n}]-[A_1, A_2, \ldots, A_{2n}]-\sum_{l=1}^{2n}[A_1,\ldots,A_{l-1},B_l,A_{l+1},\ldots,A_{2n}]
		\]
		is a trace class operator with zero trace. Recall that
		\[
		B_i=X_{f_i}+Y_{\bar{R}f_i}+K_i,\quad i=1, 2, \ldots, 2n.
		\]
		By Corollary \ref{cor: 1 B improved}, each operator
		\[
		[A_1,\ldots,A_{l-1},K_l,A_{l+1},\ldots,A_{2n}]
		\]
		is in the trace class of zero trace. Denote
		\[
		X=\sum_{l=1}^{2n}[A_1,\ldots,A_{l-1},X_{f_l},A_{l+1},\ldots,A_{2n}],
		\]
		and
		\[
		Y=\sum_{l=1}^{2n}[A_1,\ldots,A_{l-1},Y_{\bar{R}f_l},A_{l+1},\ldots,A_{2n}].
		\]
		Then it follows from Proposition \ref{prop: A} that $X$ and $Y$ are trace class operators on $\bert$. It remains to show that they have zero trace.
		\begin{notation}
		For two functions $H(\xi), G(\xi)$ on $\bn$, write $H\simeq G$ if $\int_{\bn}(H(\xi)-G(\xi))\intd\lambda_t(\xi)=0$.
		\end{notation}
		By \cite[Lemma 2.5]{TWZ:semicommutator}, it suffices to show
		\begin{equation}\label{eqn: goal for prop B}
		\la X\BKt_\xi,\BKt_\xi\ra\simeq0,\quad\text{and}\quad\la Y\BKt_\xi,\BKt_\xi\ra\simeq0.
		\end{equation}
		A moment of reflection shows that
		\begin{equation*}
			X=\sum_{k=1}^{2n}\sum_{\tau\in S_{2n}}\sgn(\tau)A_{\tau_1}\ldots A_{\tau_{k-1}}X_{f_{\tau_k}}A_{\tau_{k+1}}\ldots A_{\tau_{2n}},
		\end{equation*}
	and
		\begin{equation*}
			Y=\sum_{k=1}^{2n}\sum_{\tau\in S_{2n}}\sgn(\tau)A_{\tau_1}\ldots A_{\tau_{k-1}}Y_{\bar{R}f_{\tau_k}}A_{\tau_{k+1}}\ldots A_{\tau_{2n}}.
		\end{equation*}
For $k=1,\ldots,2n$, denote
\[
X_k=\sum_{\tau\in S_{2n}}\sgn(\tau)A_{\tau_1}\ldots A_{\tau_{k-1}}X_{f_{\tau_k}}A_{\tau_{k+1}}\ldots A_{\tau_{2n}},
\]
and
\[
Y_k=\sum_{\tau\in S_{2n}}\sgn(\tau)A_{\tau_1}\ldots A_{\tau_{k-1}}Y_{\bar{R}f_{\tau_k}}A_{\tau_{k+1}}\ldots A_{\tau_{2n}}.
\]
Then we have the following expressions for $X$ and $Y$, 
\[
X=\sum_{k=1}^{2n}X_k,\quad Y=\sum_{k=1}^{2n}Y_k.
\]
		Define
		\begin{equation*}
			X_k'=
			\sum_{\tau\in S_{2n}}\sgn(\tau)A_{\tau_{2n}}A_{\tau_1}\ldots A_{\tau_{k-1}}X_{f_{\tau_k}}A_{\tau_{k+1}}\ldots A_{\tau_{2n-1}},\quad k=1,\ldots, 2n-1,
		\end{equation*}
		\[
		X_{2n}'=\sum_{\tau\in S_{2n}}\sgn(\tau)X_{f_{\tau_{2n}}}A_{\tau_1}\ldots A_{\tau_{2n-1}},
		\]
		and
		\begin{equation*}
			Y_k'=
			\sum_{\tau\in S_{2n}}\sgn(\tau)A_{\tau_{2n}}A_{\tau_1}\ldots A_{\tau_{k-1}}Y_{\bar{R}f_{\tau_k}}A_{\tau_{k+1}}\ldots A_{\tau_{2n-1}},\quad k=1,\ldots, 2n-1,
		\end{equation*}
		\[
		Y_{2n}'=\sum_{\tau\in S_{2n}}\sgn(\tau)Y_{\bar{R}f_{\tau_{2n}}}A_{\tau_1}\ldots A_{\tau_{2n-1}}.
		\]
		In other words, $X_k', Y_k'$ are obtained from $X_k, Y_k$ by moving each rightmost operator to the leftmost.
		Define
		\[
		X'=\sum_{k=1}^{2n}X_k',\quad Y'=\sum_{k=1}^{2n}Y_k'.
		\]
		The anti-symmetrization leads to
		\[
		X'=-X,\quad Y'=-Y.
		\]
		Below we show that
		\[
		\la X\BKt_\xi,\BKt_\xi\ra\simeq\la X'\BKt_\xi,\BKt_\xi\ra,\quad\la Y\BKt_\xi,\BKt_\xi\ra\simeq\la Y'\BKt_\xi,\BKt_\xi\ra.
		\]
		This leads to \eqref{eqn: goal for prop B}, completing the proof of Proposition \ref{prop: B}.\\ 
		
		Each $f_i$ has a decomposition $f_i=g_i+h_i$, where $g_i$ has compact support, and $h_i$ has support away from $0$. Since $X_{g_i}, Y_{\bar{R}g_i}$ are perturbations that belong to the trace class, we might as well assume that each $f_i$ have support away from the origin.

		For $k, i=1,\ldots,2n$, write
		\[
		X_{k,i}(w,z)=\begin{cases}
			X_w(z),&\text{ if }i=k,\\
			\BKtp_w(z),&\text{ otherwise}.
		\end{cases}
		\]
		By definition, for any $k=1,\ldots,2n$, we compute $\la X_k\BKt_\xi,\BKt_\xi\ra$ as follows.
		\begin{flalign}\label{eqn: X_k Berezin}
			&\la X_k\BKt_\xi,\BKt_\xi\ra\nonumber\\
			=&\bigg[\sum_{\tau\in S_{2n}}\sgn(\tau)\BTtpt_{f_{\tau_1}}\ldots\BTtpt_{f_{\tau_{k-1}}}X_{f_{\tau_k}}\BTtpt_{f_{\tau_{k+1}}}\ldots\BTtpt_{f_{\tau_{2n}}}\BKt_\xi\bigg](\xi)\nonumber\\
			=&\sum_{\tau\in S_{2n}}\sgn(\tau)\int_{\bn^{2n}}f_{\tau_1}(z_1)\ldots f_{\tau_{2n}}(z_{2n})\BKt_\xi(z_{2n})\\
			&~~~~~\cdot X_{k,2n}(z_{2n},z_{2n-1})X_{k,2n-1}(z_{2n-1},z_{2n-2})\ldots X_{k,2}(z_2,z_1)X_{k,1}(z_1,\xi)\intd\lambda_{t+1}(z_{2n})\ldots\intd\lambda_{t+1}(z_1)\nonumber\\
			=&\int_{\bn^{2n}}\det\big[f_i(z_j)\big]\BKt_\xi(z_{2n})\bigg(\prod_{i=2}^{2n}X_{k,i}(z_i,z_{i-1})\bigg)X_{k,1}(z_1,\xi)
			\intd\lambda_{t+1}(z_{2n})\ldots\intd\lambda_{t+1}(z_1).\nonumber
		\end{flalign}
		Denote $F$ the column vector of $2n$ functions,
		\[
		F(z)=[f_1(z)~f_2(z)~\ldots~f_{2n}(z)]^T.
		\]
		Then by Lemma \ref{lem: Mobius basics}, we estimate the determinant function,
		\begin{flalign*}
			\bigg|\det\big[f_i(z_j)\big]\bigg|=&\bigg|\det\big[F(z_1)~\ldots~F(z_{2n})\big]\bigg|\\
			=&\bigg|\det\big[F(z_1)~F(z_2)-F(z_1)~\ldots~F(z_{2n})-F(z_1)\big]\bigg|\\
			\lesssim&\bigg(\sum_{i, j}|z_i-z_j|\bigg)^{2n-1}\\
			\lesssim&\bigg(\sum_{i, j}|1-\la z_i,z_j\ra|\bigg)^{n-1/2}.
		\end{flalign*}
		Also, we obtain the following estimate
		\[
		|X_{k,i}(w,z)|\lesssim\begin{cases}
			|w|^{-2n+\epsilon}|\BKt_w(z)|,& i=k,\\
			|\BKtp_w(z)|=\frac{1}{|1-\la z,w\ra|}\cdot|\BKt_w(z)|,&i\neq k.
		\end{cases}
		\]
		Since we assume that each $f_i$ is supported away from the origin, we continue computing the inner product $\la X_k\BKt_\xi,\BKt_\xi\ra$ using Lemma \ref{lem: Mobius basics} and the above estimate.
		\begin{flalign*}
			&\int_{\bn^{2n+1}}\bigg|\det[f_i(z_j)]\BKt_\xi(z_{2n})\bigg(\prod_{i=2}^{2n}X_{k,i}(z_i,z_{i-1})\bigg)X_{k,1}(z_1,\xi)\bigg|\intd\lambda_{t+1}(z_{2n})\ldots\intd\lambda_{t+1}(z_1)\intd\lambda_t(\xi)\\
			\lesssim&\int_{\bn^{2n+1}}\bigg|\det[f_i(z_j)]\BKt_\xi(z_{2n})\bigg(\prod_{i\neq k,1}\frac{1-|z_i|^2}{|1-\la z_i,z_{i-1}\ra|}\bigg)(1-|z_k|^2)\frac{1-|z_1|^2}{|1-\la z_1,\xi\ra|}\\
			&~~~~~~~~~~~~~~~~~~~~~~~~~~~\cdot\bigg(\prod_{i=2}^{2n}\BKt_{z_i}(z_{i-1})\bigg) \BKt_{z_1}(\xi)\bigg|\intd\lambda_{t}(z_{2n})\ldots\intd\lambda_{t}(z_1)\intd\lambda_t(\xi)\\
			\lesssim&\int_{\bn^{2n+1}}\bigg(\sum_{i,j}|1-\la z_i,z_j\ra|\bigg)^{n+\frac{1}{2}}\bigg|\BKt_\xi(z_{2n})\bigg(\prod_{i=2}^{2n}\BKt_{z_i}(z_{i-1})\bigg)\BKt_{z_1}(\xi)\bigg|\intd\lambda_{t}(z_{2n})\ldots\intd\lambda_{t}(z_1)\intd\lambda_t(\xi)\\
			<&\infty.
		\end{flalign*}
	Here the last inequality follows from Lemma \ref{lem: k kernels absolutely integrable}.
		Thus the $(2n+1)$-fold integral obtained by plugging \eqref{eqn: X_k Berezin} into $\int_{\bn}\la X_k\BKt_\xi,\BKt_\xi\ra\intd\lambda_t(\xi)$ converges absolutely. By Fubini's Theorem, the variable $\xi$ can be integrated first, and, since each $X_{k,i}(w,z)$ is holomorphic in $z$, we compute the following integral
		\begin{flalign*}
			&\int_{\bn}\la X_k\BKt_\xi,\BKt_\xi\ra\intd\lambda_t(\xi)\\
			=&\int_{\bn}\bigg\{\int_{\bn^{2n}}\det\big[f_i(z_j)\big]\BKt_\xi(z_{2n})\\
			&\qquad \qquad \qquad \bigg(\prod_{i=2}^{2n}X_{k,i}(z_i,z_{i-1})\bigg)X_{k,1}(z_1,\xi)\intd\lambda_{t+1}(z_{2n})\ldots\intd\lambda_{t+1}(z_1)\bigg\}\intd\lambda_t(\xi)\\
			=&\int_{\bn^{2n}}\det\big[f_i(z_j)\big]\bigg(\prod_{i=2}^{2n}X_{k,i}(z_i,z_{i-1})\bigg)
			X_{k,1}(z_1,z_{2n})\intd\lambda_{t+1}(z_{2n})\ldots\intd\lambda_{t+1}(z_1).
		\end{flalign*}
		By a similar proof, we have a similar expression for the integral of $\la X_k'\BKt_\xi,\BKt_\xi\ra$.
		\begin{flalign*}
			&\int_{\bn}\la X_k'\BKt_\xi,\BKt_\xi\ra\intd\lambda_t(\xi)\\
			=&\int_{\bn}\bigg\{\int_{\bn^{2n}}\det[f_i(z_j)]\BKt_\xi(z_{2n-1})\\
			&~~~~~~~~~~~~~\cdot\bigg(\prod_{i=2}^{2n-1}X_{k,i}(z_{i},z_{i-1})\bigg) X_{k,1}(z_1,z_{2n})X_{k,2n}(z_{2n},\xi)\intd\lambda_{t+1}(z_{2n})\ldots\intd\lambda_{t+1}(z_1)\bigg\}\intd\lambda_t(\xi)\\
			=&\int_{\bn^{2n}}\det[f_i(z_j)]\bigg(\prod_{i=2}^{2n}X_{k,i}(z_i,z_{i-1})\bigg)X_{k,1}(z_1,z_{2n})\intd\lambda_{t+1}(z_{2n})\ldots\intd\lambda_{t+1}(z_1)\\
			=&\int_{\bn}\la X_k\BKt_\xi,\BKt_\xi\ra\intd\lambda_t(\xi).
		\end{flalign*}
		In other words, we have shown the following equation
		\begin{equation}\label{eqn: X rotation}
			\la X_k\BKt_\xi,\BKt_\xi\ra\simeq\la X_k'\BKt_\xi,\BKt_\xi\ra.
		\end{equation}

		The situation for $Y$ is more complicated. As in the case of $X$, for $k, i=1,\ldots,2n$, define
		\[
		Y_{k,i}(w,z)=\begin{cases}
			Y_w(z),&i=k,\\
			\BKtp_w(z),&\text{ otherwise}.
		\end{cases}
		\]
		Then each $Y_{k,i}(w,z)$ is holomorphic in $z$, and
		\begin{equation*}
			|Y_{k,i}(w,z)|\lesssim\begin{cases}
				|w|^{-2n-1+\epsilon}|\BKt_w(z)|,&i=k,\\
				\frac{1}{|1-\la w,z\ra|}\cdot|\BKt_w(z)|,&\text{ otherwise}.
			\end{cases}
		\end{equation*}
		By definition, we compute
		\begin{flalign}\label{eqn: Y_k Berezin}
			\la Y_k\BKt_\xi,\BKt_\xi\ra=&\int_{\bn^{2n}}\det\big[F(z_1)~\ldots~F(z_{k-1})~\bar{R}F(z_k)~F(z_{k+1})~\ldots~F(z_{2n})\big]\\
			&\BKt_\xi(z_{2n})\bigg(\prod_{i=2}^{2n}Y_{k,i}(z_i,z_{i-1})\bigg)Y_{k,1}(z_1,\xi)\intd\lambda_{t+1}(z_{2n})\ldots\intd\lambda_{t+1}(z_1).\nonumber
		\end{flalign}
		Here $\bar{R}F(z)$ is the column vector
		\[
		\bar{R}F(z)=[\bar{R}f_1(z)~\ldots~\bar{R}f_{2n}(z)]^T.
		\]
		Direct computation shows that
		\begin{flalign*}
			&\det\big[F(z_1)~\ldots~F(z_{k-1})~\bar{R}F(z_k)~F(z_{k+1})~\ldots~F(z_{2n})\big]\\
			=&\det\big[F(z_1)-F(z_k)~\ldots~F(z_{k-1})-F(z_k)~\bar{R}F(z_k)~F(z_{k+1})-F(z_k)~\ldots~F(z_{2n})-F(z_k)\big]\\
			&+\sum_{j\neq k}\det F_{j,k},
		\end{flalign*}
	where $F_{j,k}$ is the matrix function obtained by replacing the $j$-th column of
	\[
	\big[F(z_1)~\ldots~F(z_{k-1})~\bar{R}F(z_k)~F(z_{k+1})~\ldots~F(z_{2n})\big]
	\]
	into $F(z_k)$.
	Therefore we can compute its determinant as follows.
	\begin{flalign}\label{eqn: determinant equation}
		&\det\big[F(z_1)~\ldots~F(z_{k-1})~\bar{R}F(z_k)~F(z_{k+1})~\ldots~F(z_{2n})\big]\\
		=&\det\big[F(z_1)-F(z_k)~\ldots~F(z_{k-1})-F(z_k)~\bar{R}F(z_k)~F(z_{k+1})-F(z_k)~\ldots~F(z_{2n})-F(z_k)\big]\nonumber\\
		&+\sum_{j\neq k}\sum_{\tau\in S_{2n}}\sgn(\tau)f_{\tau_j}(z_k)\bar{R}f_{\tau_k}(z_k)\prod_{i\neq j,k}f_{\tau_i}(z_i)\nonumber\\
		:=&D_k(z_1,\ldots,z_{2n})+\sum_{j\neq k}E_{k,j}(z_1,\ldots,z_{2n}),\nonumber
	\end{flalign}
	where 
	\begin{eqnarray*}
	&&D_k(z_1,\ldots,z_{2n})=\\
	&&\qquad \det\big[F(z_1)-F(z_k)~\ldots~F(z_{k-1})-F(z_k)~\bar{R}F(z_k)~F(z_{k+1})-F(z_k)~\ldots~F(z_{2n})-F(z_k)\big],\\
	&&E_{k,j}(z_1,\ldots,z_{2n})=\sum_{\tau\in S_{2n}}\sgn(\tau)f_{\tau_j}(z_k)\bar{R}f_{\tau_k}(z_k)\prod_{i\neq j,k}f_{\tau_i}(z_i).
	\end{eqnarray*}
		Correspondingly, we write
		\[
		\la Y_k\BKt_\xi,\BKt_\xi\ra=I_k(\xi)+\sum_{j\neq k}II_{k,j}(\xi),
		\]
		where
		\[
		I_k(\xi)=\int_{\bn^{2n}}D_k(z_1,\ldots,z_{2n})\BKt_\xi(z_{2n})\bigg(\prod_{i=2}^{2n}Y_{k,i}(z_i,z_{i-1})\bigg)Y_{k,1}(z_1,\xi)\intd\lambda_{t+1}(z_{2n})\ldots\intd\lambda_{t+1}(z_1),
		\]
		and
		\[
		II_{k,j}(\xi)=\int_{\bn^{2n}}E_{k,j}(z_1,\ldots,z_{2n})\BKt_\xi(z_{2n})\bigg(\prod_{i=2}^{2n}Y_{k,i}(z_i,z_{i-1})\bigg)Y_{k,1}(z_1,\xi)\intd\lambda_{t+1}(z_{2n})\ldots\intd\lambda_{t+1}(z_1).
		\]
		Similarly, we have the following expression for $\la Y_k'\BKt_\xi,\BKt_\xi\ra$,
		\[
		\la Y_k'\BKt_\xi,\BKt_\xi\ra=I_k'(\xi)+\sum_{j\neq k}II_{k,j}'(\xi),
		\]
		where
		\begin{flalign*}
			&I_k'(\xi)\\
			=&\int_{\bn^{2n}}D_k(z_1,\ldots,z_{2n})\BKt_\xi(z_{2n-1})\bigg(\prod_{i=2}^{2n-1}Y_{k,i}(z_i,z_{i-1})\bigg)Y_{k,1}(z_1,z_{2n})Y_{k,2n}(z_{2n},\xi)\intd\lambda_{t+1}(z_{2n})\ldots\intd\lambda_{t+1}(z_1),
		\end{flalign*}
		and
		\begin{flalign*}
			&II_{k,j}'(\xi)\\
			=&\int_{\bn^{2n}}E_{k,j}(z_1,\ldots,z_{2n})\BKt_\xi(z_{2n-1})\bigg(\prod_{i=2}^{2n-1}Y_{k,i}(z_i,z_{i-1})\bigg)Y_{k,1}(z_1,z_{2n})Y_{k,2n}(z_{2n},\xi)\intd\lambda_{t+1}(z_{2n})\ldots\intd\lambda_{t+1}(z_1).
		\end{flalign*}
		Since
		\begin{flalign*}
		&	|D_k|\\
		=	&\bigg|\det\big[F(z_1)-F(z_k)~\ldots~F(z_{k-1})-F(z_k)~\bar{R}F(z_k)~F(z_{k+1})-F(z_k)~\ldots~F(z_{2n})-F(z_k)\big]\bigg|\\
			\lesssim&|z_k|\cdot\bigg(\sum_{i,j}|1-\la z_i,z_j\ra|\bigg)^{n+1/2},
		\end{flalign*}
		we can repeat the proof for \eqref{eqn: X rotation} and get the identity
		\begin{equation}\label{eqn: Ik simeq Ik'}
			I_k(\xi)\simeq I_k'(\xi).
		\end{equation}
		Define
		\[
		Z_{k,j}=\begin{cases}
			\sum_{\tau\in S_{2n}}\sgn(\tau)\BTtpt_{f_{\tau_1}}\ldots\widehat{\BTtpt_{f_{\tau_j}}}\ldots\BTtpt_{f_{\tau_{k-1}}}Y_{f_{\tau_j}\bar{R}f_{\tau_k}}\BTtpt_{f_{\tau_{k+1}}}\ldots\BTtpt_{f_{\tau_{2n}}},&j<k\\
			\sum_{\tau\in S_{2n}}\sgn(\tau)\BTtpt_{f_{\tau_1}}\ldots\BTtpt_{f_{\tau_{k-1}}}Y_{f_{\tau_j}\bar{R}f_{\tau_k}}\BTtpt_{f_{\tau_{k+1}}}\ldots\widehat{\BTtpt_{f_{\tau_j}}}\ldots\BTtpt_{f_{\tau_{2n}}},& j>k,
		\end{cases}
		\]
		and $Z_{k,j}'$ the operator obtained from $Z_{k,j}$ by moving each rightmost operator to the leftmost. Here $\widehat{A}$ means that $A$ is removed. Checking by definition we get
		\begin{equation*}
			II_{k,j}(\xi)=\la Z_{k,j}\BKt_\xi,\BKt_\xi\ra,\quad II_{k,j}'(\xi)=\la Z_{k,j}'\BKt_\xi,\BKt_\xi\ra.
		\end{equation*}
		Using antisymmetrization we see that
		\[
		Z_{k,j}=\begin{cases}
			(-1)^{k-j-1}Z_{k,k-1},&j<k,\\
			(-1)^{k-j-1}Z_{k,k+1},&j>k.
		\end{cases}
		\]
		Therefore we have the following calculation.
		\begin{flalign*}
			&\sum_{k=1}^{2n}\sum_{j\neq k}Z_{k,j}=\sum_{k\text{ odd}}Z_{k,k+1}+\sum_{k\text{ even}}Z_{k,k-1}=\sum_{m=1}^n\bigg(Z_{2m-1,2m}+Z_{2m,2m-1}\bigg)\\
			=&\sum_{m=1}^n\bigg\{\sum_{\tau\in S_{2n}}\sgn(\tau)\BTtpt_{f_{\tau_1}}\ldots\BTtpt_{f_{\tau_{2m-2}}}Y_{f_{\tau_{2m}}\bar{R}f_{\tau_{2m-1}}}\BTtpt_{f_{\tau_{2m+1}}}\ldots\BTtpt_{f_{\tau_{2n}}}\\
			&+\sum_{\tau\in S_{2n}}\sgn(\tau)\BTtpt_{f_{\tau_1}}\ldots\BTtpt_{f_{\tau_{2m-2}}}Y_{f_{\tau_{2m-1}}\bar{R}f_{\tau_{2m}}}\BTtpt_{f_{\tau_{2m+1}}}\ldots\BTtpt_{f_{\tau_{2n}}}\bigg\}\\
			=&\sum_{m=1}^n\sum_{\tau\in S_{2n}}\sgn(\tau)\BTtpt_{f_{\tau_1}}\ldots\BTtpt_{f_{\tau_{2m-2}}}Y_{f_{\tau_{2m}\bar{R}f_{\tau_{2m-1}}}+f_{\tau_{2m-1}}\bar{R}f_{\tau_{2m}}}\BTtpt_{f_{\tau_{2m+1}}}\ldots\BTtpt_{f_{\tau_{2n}}}\\
			=&\sum_{m=1}^n\sum_{\tau\in S_{2n}}\sgn(\tau)\BTtpt_{f_{\tau_1}}\ldots\BTtpt_{f_{\tau_{2m-2}}}Y_{\bar{R}[f_{\tau_{2m}}f_{\tau_{2m-1}}]}\BTtpt_{f_{\tau_{2m+1}}}\ldots\BTtpt_{f_{\tau_{2n}}}\\
			=&0.
		\end{flalign*}
		Similarly, $\sum_{k=1}^{2n}\sum_{j\neq k}Z_{k,j}'=0$.
		Thus we can verify the following equations. 
		\[
		\sum_{k=1}^{2n}\sum_{j\neq k}II_{k,j}(\xi)=\sum_{k=1}^{2n}\sum_{j\neq k}\la Z_{k,j}\BKt_\xi,\BKt_\xi\ra=\la \bigg(\sum_{k=1}^{2n}\sum_{j\neq k}Z_{k,j}\bigg)\BKt_\xi,\BKt_\xi\ra=0,
		\]
		\[
		\sum_{k=1}^{2n}\sum_{j\neq k}II_{k,j}'(\xi)=\sum_{k=1}^{2n}\sum_{j\neq k}\la Z_{k,j}'\BKt_\xi,\BKt_\xi\ra=\la \bigg(\sum_{k=1}^{2n}\sum_{j\neq k}Z_{k,j}'\bigg)\BKt_\xi,\BKt_\xi\ra=0.
		\]
		This leads to the following identities
		\[
		\la Y\BKt_\xi,\BKt_\xi\ra=\sum_{k=1}^{2n}\la Y_k\BKt_\xi,\BKt_\xi\ra=\sum_{k=1}^{2n}I_k(\xi)+\sum_{k=1}^{2n}\sum_{j\neq k}II_{k,j}(\xi)=\sum_{k=1}^{2n}I_k(\xi),
		\]
		and
		\[
		\la Y'\BKt_\xi,\BKt_\xi\ra=\sum_{k=1}^{2n}\la Y_k'\BKt_\xi,\BKt_\xi\ra=\sum_{k=1}^{2n}I_k'(\xi)+\sum_{k=1}^{2n}\sum_{j\neq k}II_{k,j}'(\xi)=\sum_{k=1}^{2n}I_k'(\xi).
		\]
		Combining with \eqref{eqn: Ik simeq Ik'}, we arrive at the following equation
		\begin{equation}\label{eqn: Y rotation}
			\la Y\BKt_\xi,\BKt_\xi\ra\simeq\la Y'\BKt_\xi,\BKt_\xi\ra.
		\end{equation}
		As explained before, \eqref{eqn: X rotation} and \eqref{eqn: Y rotation} implies \eqref{eqn: goal for prop B}, which completes the proof of Proposition \ref{prop: B}.
	\end{proof}

	\section{A Quantization Formula}\label{sec: Quantization}
	Toeplitz quantization, or Berezin-Toeplitz quantization \cite{Bau:BTquantization}, has been studied by many researchers on various types of domains. As an incomplete list, it was studied in \cite{Bau:BTquantization, Co:deformation, Co:BTquantization} on the Fock space on $\cn$; in \cite{En:asymptotics, Kl-Le-92} for planar domains; in \cite{En:weighted} for pseudoconvex domains; in \cite{En-Up:quantization-bsdomain, Up:TBquantization} on symmetric domains; in \cite{En:berezin, Ma-Mari:BTquantization-kahler, Ma-Zh:Tquantization, Sc:quantization-kahler} for {K}\"{a}hler manifolds. Also see \cite{Al-En:quantization, Englis2016Tquantization, Sc:quantization-kahler} for some very well-written surveys on this topic. In this section we give an explicit formula for the quantization of Toeplitz operators on the unit ball $\bn$. In addition to the standard norm estimates, we also obtain Schatten class membership and Schatten norm estimates. The formulas and its corollaries will be important tools in the proof of Theorem \ref{thm: AS partial trace dimension n}.

	Recall that we defined the functions $d_{\alpha,\beta}$ and $I^{\alpha,\beta}$ in Definition \ref{defn: d numbers I function}.
	
	For $i=1,\ldots,n$, denote $e_i$ the multi-index that equals $1$ at the $i$-th entry and $0$ elsewhere. For $i_1, i_2, \ldots, i_k$, denote
	\[
	e_{i_1,i_2,\ldots,i_k}=e_{i_1}+e_{i_2}+\ldots+e_{i_k}.
	\]

	\begin{thm}\label{thm: quantization bergman}
		Suppose $t>-1$, $k$ is a non-negative integer and $f, g\in\mathscr{C}^{k+1}(\overline{\bn})$. Then we have the decomposition
		\begin{equation}\label{eqn: quantization hardy and bergman}
			\BTt_f\BTt_g=\sum_{l=0}^kc_{l,t}\BTt_{C_l(f,g)}+R^{(t)}_{f,g,k+1},
		\end{equation}
		where
		\begin{equation}\label{eqn: c0 C0 C1}
			c_{0,t}=1,~~ c_{1,t}=nt^{-1}+O(t^{-2});\quad C_0(f,g)=fg,~~ C_1(f,g)-C_1(g,f)=\frac{-i}{n}\{f,g\},
		\end{equation}
		and
		\begin{equation}\label{eqn: Rfg norm}
			\|R_{f,g,k+1}^{(t)}\|\lesssim_k t^{-k-1}.
		\end{equation}
		Here $\{f,g\}$ is the Poisson bracket of $f$ and $g$.
		
		Moreover, the explicit formulas for $c_{l,t}$, $C_l(f,g)$ and $R_{f,g,k+1}^{(t)}$ are given as follows. For any $l\geq0$,
		\begin{equation}\label{eqn: clt formula}
			c_{l,t}=\frac{\BFt_{n+l}\Phi_{n,l}^{(t)}(0)}{B(n,t+1)}\approx t^{-l},
		\end{equation}
		\begin{equation}\label{eqn: Clfg formula}
			C_l(f,g)(z)=(-1)^l(1-|z|^2)^{-2l}\sum_{i_1,j_1,\ldots,i_l,j_l=1}^nd_{e_{i_1,\ldots,i_l},e_{j_1,\ldots,j_l}}(z)\bigg[D_{i_l,j_l}\ldots D_{i_1,j_1}(f(z)g(w))\bigg]\bigg|_{w=z},
		\end{equation}
		where
		\begin{equation}\label{eqn: Dij defn}
			D_{i,j}=(1-\la z,w\ra)^2\partial_{z_i}\bpartial_{w_j}.
		\end{equation}
		For any $h\in\bert$ and $\xi\in\bn$,
		\begin{flalign}\label{eqn: Rfg formula}
			R_{f,g,k+1}^{(t)}h(\xi)=&\int_{\bn}\int_{\bn}\Phi_{n,k+1}^{(t)}(|\varphi_z(w)|^2)S_{f,g,k+1}(z,w)h(w)\BKt_z(\xi)\BKt_w(z)\intd\lambda_t(w)\intd\lambda_t(z),
		\end{flalign}
		where
		\begin{flalign}\label{eqn: Sfg formula}
			S_{f,g,k+1}(z,w)=&\frac{(-1)^{k+1}}{|1-\la w,z\ra|^{2(k+1)}}\\
			&\cdot\sum_{i_1,j_1,\ldots,i_{k+1},j_{k+1}=1}^nI^{e_{i_1,\ldots,i_{k+1}},e_{j_1,\ldots,j_{k+1}}}(z-w)D_{i_{k+1},j_{k+1}}\ldots D_{i_1,j_1}[f(z)g(w)]. \nonumber
		\end{flalign}
	\end{thm}
	
	The formulas above lead to Schatten-class membership and Schatten norm estimates of the remainder term.
	
	\begin{cor}\label{cor: quantization schatten dim 1}
		If $k\geq0$ and $f, g\in\mathscr{C}^{k+1}(\overline{\dd})$, then for any $p\geq1$ and $t>-1$, $R^{(t)}_{f,g,k+1}\in\mathcal{S}^1$. Moreover, for such $p$,
		\[
		\|R^{(t)}_{f,g,k+1}\|_{\mathcal{S}^p}\lesssim_{k,p}t^{-(k+1)+\frac{n}{p}}.
		\]
	\end{cor}

	\begin{cor}\label{cor: quatization schatten dim n}
		If $k\geq0$ and $f, g\in\mathscr{C}^{k+1}(\overline{\bn})$, then for any $p>n$ and $t>-1$, $R^{(t)}_{f,g,k+1}\in\mathcal{S}^p$. Moreover, for such $p$,
		\[
		\|R^{(t)}_{f,g,k+1}\|_{\mathcal{S}^p}\lesssim_{k,p}t^{-(k+1)+\frac{n}{p}}.
		\]
	\end{cor}
Recall that
	\begin{equation*}
	d_{e_i,e_j}(z)=\int_{\sn}\big(A_z\zeta\big)_i\overline{\big( A_z\zeta\big)_j}\frac{\intd\sigma(\zeta)}{\sigma_{2n-1}}.
	\end{equation*}
Therefore, we have the following formula for $C_1(f,g)$.
	\begin{equation*}
	\begin{split}
	C_1(f,g)(z)=&-\sum_{i,j=1}^nd_{e_i,e_j}(z)\partial_if(z)\bpartial_jg(z)\\
	=&-\int_{\sn}\bigg(\sum_{i=1}^n\big(A_z\zeta\big)_i\partial_if(z)\bigg)\bigg(\sum_{j=1}^n\overline{\big(A_z(\zeta)\big)_j}\bpartial_j g(z)\bigg)\frac{\intd\sigma(\zeta)}{\sigma_{2n-1}}
	\end{split}
	\end{equation*}
	As in the proof of \cite[Equation (4.8)]{TWZ:semicommutator}, the functions 
	\[
	\sum_{i=1}^n\big(A_z\zeta\big)_i\partial_if(z)\quad\text{and}\quad\sum_{j=1}^n\overline{\big(A_z(\zeta)\big)_j}\bpartial_j g(z)=\overline{\sum_{j=1}^n\big(A_z(\zeta)\big)_j\partial_j\bar{g}(z)}
	\]
	are independent of the choice of basis.
	Thus $C_1(f,g)$ is independent of the choice of basis. A similar computation shows that $C_l(f,g)$ ($l=1,2,\cdots$) are all independent of the choice of basis for general $l$.
	\begin{rem}\label{rem: CN CT}
		At $z\in\bn$, $z\neq0$, choose an orthonormal basis $\mathbf{e}_z=\{e_{z,1},e_{z,2},\ldots,e_{z,n}\}$ of $\cn$ under which $z$ has coordinates $(z_1,0,\ldots,0)$. Then under $\mathbf{e}_z$, by \eqref{eqn: d at (1,0...0)},
		\begin{equation*}
			C_1(f,g)(z)=C_N(f,g)(z)+C_T(f,g)(z),
		\end{equation*}
		where
		\[
		C_N(f,g)(z)=-\frac{1}{n}(1-|z|^2)^2\partial_1f(z)\bpartial_1g(z),\quad C_T(f,g)(z)=-\frac{1}{n}(1-|z|^2)\sum_{i=2}^n\partial_i f(z)\bpartial_i g(z).
		\]
		The functions $C_N(f,g)$ and $C_T(f,g)$ represent parts of $C_1(f,g)$ involving derivatives of $f, g$ in the complex normal and tangential directions, respectively. It is easy to see that the definitions of $C_N(f,g)$ and $C_T(f,g)$ do not depend on the choice of $\mathbf{e}_z$ (as long as $e_{z, 1}$ is in $z$-direction). Locally, we can choose $\mathbf{e}_z$ so that the vectors vary smoothly with respect to $z$. If $f, g$ are $\mathscr{C}^1$ except possibly at the origin. Also, it is easy to see from the definition that $C_N(f,g)$, $C_T(f,g)$ are bounded. This implies
		\[
		|C_N(f,g)(z)|\lesssim(1-|z|^2)^2,\quad|C_T(f,g)(z)|\lesssim1-|z|^2.
		\]
		By Lemma \ref{lem: Toeplitz Schatten Class}, for $t$ large enough, we obtain
		\[
		T^{(t)}_{C_N(f,g)}\in\mathcal{S}^p,\quad\|\BTt_{C_N(f,g)}\|_{\mathcal{S}^p}\lesssim_p t^{\frac{n}{p}},\quad\forall p>\frac{n}{2}, p\geq1,
		\]
		and
		\[
		T^{(t)}_{C_T(f,g)}\in\mathcal{S}^p,\quad\|\BTt_{C_T(f,g)}\|_{\mathcal{S}^p}\lesssim_p t^{\frac{n}{p}},\quad\forall p>n, p\geq1.
		\]	
	\end{rem}

\begin{rem}\label{rem: C1 formula for any basis}
Continuing with Remark \ref{rem: CN CT}, we can take $e_{z,1}=\frac{z}{|z|}$. Denote $\mathbf{e}=\{e_1,\ldots,e_n\}$ the canonical basis. Denote $\zeta_i=\zeta_i(\xi)$
to be the $i$-th coordinate of $\xi$ under the basis $\mathbf{e}_z$. Then compute
\[
\frac{\partial\xi_i}{\partial\zeta_1}=\overline{\la e_i,e_{z,1}\ra}=\big(e_{z,1}\big)_i=\frac{z_i}{|z|}.
\]
Thus, we compute the following expressions, 
\begin{flalign*}
\bigg[\frac{\partial f(\xi)}{\partial\zeta_1}\frac{\partial g(\xi)}{\partial\bar{\zeta}_1}\bigg]\bigg|_{\xi=z}=&\bigg[\frac{\partial f(\xi)}{\partial\zeta_1}\frac{\partial g(\xi)}{\partial\bar{\zeta}_1}\bigg]\bigg|_{\xi=z}\\
=&\bigg[\sum_{i,j=1}^n\frac{\partial f(\xi)}{\partial\xi_i}\frac{\partial\xi_i}{\partial\zeta_1}\frac{\partial g(\xi)}{\partial\bar{\xi}_j}\overline{\bigg(\frac{\partial\xi_j}{\partial\zeta_1}\bigg)}\bigg]\bigg|_{\xi=z}\\
=&\bigg[\sum_{i,j=1}^n\frac{\partial f(\xi)}{\partial\xi_i}\frac{z_i}{|z|}\frac{\partial g(\xi)}{\partial\bar{\xi}_j}\overline{\bigg(\frac{z_j}{|z|}\bigg)}\bigg]\bigg|_{\xi=z}\\
=&|z|^{-2}Rf(z)\bar{R}g(z),
\end{flalign*}
and
\begin{flalign*}
\bigg[\sum_{i=2}^n\frac{\partial f(\xi)}{\partial\zeta_i}\frac{\partial g(\xi)}{\partial\bar{\zeta}_j}\bigg]\bigg|_{\xi=z}=&\la\partial f(z),\bpartial g(z)\ra-|z|^{-2}Rf(z)\bar{R}g(z)\\
=&\sum_{i,j=1}^n\big(\delta_{i,j}-\frac{z_i\bar{z}_j}{|z|^2}\big)\partial_if(z)\bpartial_jg(z).
\end{flalign*}
In other words, we arrive at the following expression
\begin{equation*}
C_N(f,g)(z)=-\frac{1}{n}(1-|z|^2)^2|z|^{-2}Rf(z)\bar{R}g(z),
\end{equation*}
and
\begin{equation*}
C_T(f,g)(z)=-\frac{1}{n}(1-|z|^2)\sum_{i,j=1}^n\big(\delta_{i,j}-\frac{z_i\bar{z}_j}{|z|^2}\big)\partial_if(z)\bpartial_jg(z).
\end{equation*}
Adding up the two equations gives the following formula,
\begin{equation*}
C_1(f,g)(z)=-\frac{1}{n}(1-|z|^2)\bigg[\sum_{i=1}^n\partial_if(z)\bpartial_ig(z)-Rf(z)\bar{R}g(z)\bigg].
\end{equation*}
\end{rem}
	
	Motivated by Remark \ref{rem: CN CT}, we further decompose $R^{(t)}_{f,g,k+1}$ according to the normal and tangential derivatives.
	\begin{defn}\label{defn: Sabfgk}
		For $z\in\bn, z\neq0$, let $\mathbf{e}_z=\{e_{z,1}, e_{z,2},\ldots,e_{z,n}\}$ be as in Remark \ref{rem: CN CT}. Then $e_{z,1}$ represents the complex normal direction at $z$, and $e_{z,2},\ldots, e_{z,n}$ represents the complex tangential directions at $z$.
		
		Under the basis $\mathbf{e}_z$, by \eqref{eqn: Sfg formula}, $S_{f,g,k+1}$ decomposes into
		\[
		S_{f,g,k+1}(z,w)=\sum_{1\leq|\alpha|, |\beta|\leq k+1}V_{k+1}^{\alpha,\beta}(z,w)\partial^\alpha f(z)\bpartial^\beta g(w).
		\]
		For integers $0\leq a, b\leq k+1$, define
		\begin{equation*}
			S_{f,g,k+1}^{a,b}(z,w)=\sum_{\stackrel{1\leq|\alpha|,|\beta|\leq k+1}{\alpha_1=a, \beta_1=b}}V_{k+1}^{\alpha,\beta}(z,w)\partial^\alpha f(z)\bpartial^\beta g(w).
		\end{equation*}
		A moment of reflection shows that the function $S^{a,b}_{f,g,k+1}$ does not depend on the choice of $\mathbf{e}_z$.
		Define the corresponding operator on $\bert$,
		\begin{flalign*}
			R^{(t) a,b}_{f,g,k+1}h(\xi)=\int_{\bn}\int_{\bn}\Phi_{n,k+1}^{(t)}(|\varphi_z(w)|^2)S^{a,b}_{f,g,k+1}(z,w)h(w)\BKt_z(\xi)\BKt_w(z)\intd\lambda_t(w)\intd\lambda_t(z).
		\end{flalign*}
		Then we write
		\[
		R^{(t)}_{f,g,k+1}=\sum_{a,b=0,\ldots,k+1}R^{(t)a,b}_{f,g,k+1}.
		\]
	\end{defn}
	
	\begin{cor}\label{cor: Rabfgk membership}
		Suppose $f, g\in\mathscr{C}^{k+1}(\overline{\bn})$. Then for any $0\leq a, b\leq k+1$ any $p\geq1$, $p>\max\{\frac{n}{1+\frac{a+b}{2}},\frac{n}{k+1+\frac{t+1}{2}}\}$, $R^{(t) a,b}_{f,g,k+1}\in\mathcal{S}^p$. Moreover, for such $p$, and $t$ large enough,
		\[
		\|R^{(t) a,b}_{f,g,k+1}\|_{\mathcal{S}^p}\lesssim_{k,p}t^{-k-1+\frac{n}{p}}.
		\]
		In particular,
		\begin{itemize}
			\item[(1)] if one of $f, g$ has the form $\phi(|z|^2)$, where $\phi\in\mathscr{C}^{k+1}([0,1])$, then
			\[
			R^{(t)}_{f,g,k+1}\in\mathcal{S}^p,\quad\forall p>\max\{\frac{2n}{3},\frac{n}{k+1+\frac{t+1}{2}}\};
			\]
			\item[(2)] if both $f, g$ are of the form $\phi(|z|^2)$, $\phi\in\mathscr{C}^{k+1}([0,1])$, then
			\[
			R^{(t)}_{f,g,k+1}\in\mathcal{S}^p,\quad\forall p>\max\{\frac{n}{2},\frac{n}{k+1+\frac{t+1}{2}}\}.
			\]
		\end{itemize}
		For $p>\frac{2n}{3}$ in case (1) and $p>\frac{n}{2}$ in case (2) and $t$ large enough,
		\[
		\|R^{(t)}_{f,g,k+1}\|_{\mathcal{S}^p}\lesssim_{k,p}t^{-k-1+\frac{n}{p}}.
		\]
	\end{cor}
	
	\begin{cor}\label{cor: semicom with radial functions}
		Suppose $x, y$ are positive integers, and $f\in\mathscr{C}^1(\overline{\bn})$. Then the following hold.
		\begin{itemize}
			\item[(1)] For $p>\max\{\frac{n}{x+\frac{1}{2}},\frac{n}{1+\frac{t+1}{2}}\}$, $\BTt_{(1-|z|^2)^x}\BTt_f-\BTt_{(1-|z|^2)^xf}$ and $ \BTt_f\BTt_{(1-|z|^2)^x}-\BTt_{(1-|z|^2)^xf}$ are in $\mathcal{S}^p$. For $p>\frac{n}{x+\frac{1}{2}}$ and $t$ large enough,
			\[
			\|\BTt_{(1-|z|^2)^x}\BTt_f-\BTt_{(1-|z|^2)^xf}\|_{\mathcal{S}^p}\lesssim_p t^{-1+\frac{n}{p}},\quad\|\BTt_f\BTt_{(1-|z|^2)^x}-\BTt_{(1-|z|^2)^xf}\|_{\mathcal{S}^p}\lesssim_p t^{-1+\frac{n}{p}}.
			\]
			\item[(2)] For $p>\max\{\frac{n}{x+y},\frac{n}{1+\frac{t+1}{2}}\}$, $\BTt_{(1-|z|^2)^x}\BTt_{(1-|z|^2)^y}-\BTt_{(1-|z|^2)^{x+y}}\in\mathcal{S}^p$. For $p>\frac{n}{x+y}$ and $t$ large enough,
			\[
			\|\BTt_{(1-|z|^2)^x}\BTt_{(1-|z|^2)^y}-\BTt_{(1-|z|^2)^{x+y}}\|_{\mathcal{S}^p}\lesssim_p t^{-1+\frac{n}{p}}.
			\]
		\end{itemize}
	\end{cor}
In general, we would expect that $\BTt_f\BTt_g-\BTt_{fg}$ have better Schatten class membership compared to $\BTt_f\BTt_g$: for arbitrary $f, g\in\mathscr{C}^1(\overline{\bn})$, $\BTt_f\BTt_g$ is only bounded, whereas the semi-commutator $\BTt_f\BTt_g-\BTt_{fg}\in\mathcal{S}^p$ for any $p>n$. Also in case (1) of the corollary above, for a general function $f\in\mathscr{C}^1(\overline{\bn})$, $\BTt_{(1-|z|^2)^x}\BTt_f$ is in $\mathcal{S}^p, \forall p>\frac{n}{x}$, while the semi-commutator $\BTt_{(1-|z|^2)^x}\BTt_f-\BTt_{(1-|z|^2)^xf}$ is in $\mathcal{S}^p, \forall p>\frac{n}{x+\frac{1}{2}}.$ This is no longer true in case (2): both $\BTt_{(1-|z|^2)^x}\BTt_{(1-|z|^2)^y}$ and $\BTt_{(1-|z|^2)^x}\BTt_{(1-|z|^2)^y}-\BTt_{(1-|z|^2)^{x+y}}$ are in $\mathcal{S}^p$ for $p>\frac{n}{x+y}$. Intuitively, this has to do with the fact that functions of the form $(1-|z|^2)^x$ already vanishes along the radial direction to some order.\\

	In the rest of this section, we prove Theorem \ref{thm: quantization bergman} and Corollary \ref{cor: Rabfgk membership}.

	Functions of the form
	\[
	\sum_{i_1,\ldots,i_l,j_1,\ldots,j_l=1}^nI^{e_{i_1,\ldots,i_l},e_{j_1,\ldots,j_l}}(z-w)D_{i_l,j_l}D_{i_{l-1},j_{l-1}}\ldots D_{i_1,j_1}[f(z)g(w)]
	\]
	appear in the formula of $S_{f,g,k+1}(z,w)$ in Theorem \ref{thm: quantization bergman}. We need to estimate its absolute value. To start with, we write the above sum in terms of the standard derivation. Recall that
	\[
	D_{i,j}=(1-\la z,w\ra)^2\partial_{z_i}\bpartial_{w_j}.
	\]
	\begin{defn}\label{defn: ABC functions}
		Denote
		\[
		A_{x,y}(z,w)=\sum_{i_1,\ldots,i_x=1}^n\sum_{j_1,\ldots,j_y=1}^nI^{e_{i_1,\ldots,i_x},e_{j_1,\ldots,j_y}}(z-w)\partial_{i_1}\ldots\partial_{i_x}f(z)\bpartial_{j_1}\ldots\bpartial_{j_y}g(w);
		\]
		\[
		B_1(z,w)=\sum_{i=1}^n(z_i-w_i)\partial_{z_i}(1-\la z,w\ra)=-\sum_{i=1}^n(z_i-w_i)\bw_i=\la w-z,w\ra;
		\]
		\[
		B_2(z,w)=\sum_{j=1}^n\overline{(z_j-w_j)}\bpartial_{w_j}(1-\la z,w\ra)=\la z,w-z\ra;
		\]
		\[
		C(z,w)=\sum_{i,j=1}^nI^{e_i,e_j}(z-w)\partial_{z_i}\bpartial_{w_j}(1-\la z,w\ra)=-|z-w|^2.
		\]
		In $D_{i_l,j_l}D_{i_{l-1},j_{l-1}}\ldots D_{i_2,j_2}$, the partial derivations $\partial_{z_{i_l}}$, $\bpartial_{w_{j_l}},\ldots,\partial_{z_{i_2}}$, $\bpartial_{w_{j_2}}$ fall either on $f(z)g(w)$ or a copy of $(1-\la z,w\ra)$. Thus the summation
		\[
		\sum_{i_1,\ldots,i_l,j_1,\ldots,j_l=1}^nI^{e_{i_1,\ldots,i_l},e_{j_1,\ldots,j_l}}(z-w)D_{i_l,j_l}D_{i_{l-1},j_{l-1}}\ldots D_{i_1,j_1}[f(z)g(w)]
		\]
		can be reorganized into sums of functions of the form
		\[
		(1-\la z,w\ra)^{s_1}\big[C(z,w)\big]^{s_2}\big[B_1(z,w)\big]^{s_3}\big[B_2(z,w)\big]^{s_4}A_{x,y}(z,w).
		\]
		In total, there are $2l$ steps of taking partial derivatives, and there are $2l$ copies of $(1-\la z,w\ra)$ in the above. So
		\[
		2s_2+s_3+s_4+x+y=s_1+s_2+s_3+s_4=2l.
		\]
		Also, the partial derivatives of the first operator, $D_{i_1,j_1}$, always apply on $f(z)g(w)$. So
		\[
		x, y\geq1.
		\]
	\end{defn}
	
	From definition and Lemma \ref{lem: Mobius basics}, the following estimates are obvious.
	\begin{lem}\label{lem: ABC estimates}
		Suppose $f, g\in\mathscr{C}^l(\overline{\bn})$. Then
		\begin{itemize}
			\item[(1)] $|A_{x,y}(z,w)|\lesssim|\varphi_z(w)|^{x+y}|1-\la z,w\ra|^{\frac{x+y}{2}}$;
			\item[(2)] $|B_1(z,w)|\lesssim|\varphi_z(w)||1-\la z,w\ra|$,\quad$B_2(z,w)|\lesssim|\varphi_z(w)||1-\la z,w\ra|;$
			\item[(3)] $|C(z,w)|\lesssim|\varphi_z(w)|^2|1-\la z,w\ra|$.
		\end{itemize}
	\end{lem}

	\begin{lem}\label{lem: IG estimates}
		Suppose $k$ is a non-negative integer and $f, g\in\mathscr{C}^{k+1}(\overline{\bn})$. For any $l=1,\ldots,k+1$, set
		\[
		G_{i_1,\ldots,i_l; j_1,\ldots,j_l}(z,w)=D_{i_l,j_l}D_{i_{l-1},j_{l-1}}\ldots D_{i_1,j_1}[f(z)g(w)].
		\]		
		Then  the following estimates hold.
		\begin{equation}\label{eqn: IG estimate 1}
			\bigg|\sum_{i_1,\ldots,i_l,j_1,\ldots,j_l=1}^nI^{e_{i_1,\ldots,i_l},e_{j_1,\ldots,j_l}}(z-w)G_{i_1,\ldots,i_l; j_1,\ldots,j_l}(z,w)\bigg|\lesssim|\varphi_z(w)|^{2l}|1-\la z,w\ra|^{2l+1},
		\end{equation}
		\begin{equation}\label{eqn: IG estimate 2}
			\bigg|\sum_{i_1,\ldots,i_l,j_1,\ldots,j_l,j_{l+1}=1}^nI^{e_{i_1,\ldots,i_l},e_{j_1,\ldots,j_l,j_{l+1}}}(z-w)\bpartial_{w_{j_{l+1}}}G_{i_1,\ldots,i_l; j_1,\ldots,j_l}(z,w)\bigg|\lesssim|\varphi_z(w)|^{2l+1}|1-\la z,w\ra|^{2l+1}.
		\end{equation}
	\end{lem}
	\begin{proof}
		As explained in Definition \ref{defn: ABC functions}, the following sum
		\[
		\sum_{i_1,\ldots,i_l,j_1,\ldots,j_l=1}^nI^{e_{i_1,\ldots,i_l},e_{j_1,\ldots,j_l}}(z-w)G_{i_1,\ldots,i_l; j_1,\ldots,j_l}(z,w)
		\]
		splits into sums of functions of the form
		\[
		(1-\la z,w\ra)^{s_1}\big[C(z,w)\big]^{s_2}\big[B_1(z,w)\big]^{s_3}\big[B_2(z,w)\big]^{s_4}A_{x,y}(z,w),
		\]
		with
		\[
		x, y\geq1,\quad2s_2+s_3+s_4+x+y=s_1+s_2+s_3+s_4=2l.
		\]
		Similarly, the sum 
		\[
		\sum_{i_1,\ldots,i_l,j_1,\ldots,j_l,j_{l+1}=1}^nI^{e_{i_1,\ldots,i_l},e_{j_1,\ldots,j_l,j_{l+1}}}(z-w)\bpartial_{w_{j_{l+1}}}G_{i_1,\ldots,i_l; j_1,\ldots,j_l}(z,w)
		\]
		splits into sums of functions of above form, with
		\[
		x, y\geq1,\quad 2s_2+s_3+s_4+x+y=2l+1,\quad s_1+s_2+s_3+s_4=2l.
		\]
		By Lemma \ref{lem: ABC estimates}, we have the following estimate
		\begin{flalign*}
			&\bigg|(1-\la z,w\ra)^{s_1}\big[C(z,w)\big]^{s_2}\big[B_1(z,w)\big]^{s_3}\big[B_2(z,w)\big]^{s_4}A_{x,y}(z,w)\bigg|\\
			\lesssim&|\varphi_z(w)|^{2s_2+s_3+s_4+x+y}|1-\la z,w\ra|^{s_1+s_2+s_3+s_4+\frac{x+y}{2}}.
		\end{flalign*}
		Plugging in the equations for $s_i, x, y$ gives the inequalities \eqref{eqn: IG estimate 1} and \eqref{eqn: IG estimate 2}. This completes the proof of Lemma \ref{lem: IG estimates}.
	\end{proof}

	\begin{proof}[{\bf Proof of Theorem \ref{thm: quantization bergman}}]
		Let $G_{i_1,\ldots,i_l;j_1,\ldots,j_l}$ be defined as in Lemma \ref{lem: IG estimates}. Suppose $h\in\Hol(\overline{\bn})$ and $\xi\in\bn$. Write $F=f(z)g(w)h(w)\BKt_z(\xi)$ and
		\[
		F_{i_1,\ldots,i_l;j_1,\ldots,j_l}=D_{i_l,j_l}\ldots D_{i_1,j_1}[f(z)g(w)h(w)\BKt_z(\xi)]=G_{i_1,\ldots,i_l;j_1,\ldots,j_l}h(w)\BKt_z(\xi).
		\]
		Then we compute $\BTt_f\BTt_{g}$ as follows.
		\begin{flalign*}
			&\BTt_f\BTt_gh(\xi)\\
			=&\int_{\bn^2}f(z)g(w)h(w)\BKt_z(\xi)\BKt_w(z)\intd\lambda_t(w)\intd\lambda_t(z)\\
			=&\int_{\bn^2}\Phi^{(t)}_{n,0}F(z,w)\BKt_w(z)\intd\lambda_t(w)\intd\lambda_t(z)\\
			\xlongequal{\eqref{eqn: formula F induction}}&\frac{\BFt_n\Phi^{(t)}_{n,0}(0)}{B(n,t+1)}\int_{\bn}d_{0,0}(z)F(z,z)\intd\lambda_t(z)\\
			&-\int_{\bn^2}\Phi^{(t)}_{n,1}(|\varphi_z(w)|^2)\frac{\sum_{i,j=1}^nI^{e_i,e_j}(z-w)D_{i,j}F(z,w)}{|1-\la z,w\ra|^2}\BKt_w(z)\intd\lambda_t(z)\intd\lambda_t(w)\\
			=&\BTt_{fg}h(\xi)-\int_{\bn^2}\Phi^{(t)}_{n,1}(|\varphi_z(w)|^2)\frac{\sum_{i,j=1}^nI^{e_i,e_j}(z-w)F_{i,j}(z,w)}{|1-\la z,w\ra|^2}\BKt_w(z)\intd\lambda_t(z)\intd\lambda_t(w)\\
			=&\BTt_{fg}h(\xi)+R^{(t)}_{f,g,1}.
		\end{flalign*}
		The condition for applying Lemma \ref{lem: formula F induction} is verified by Lemma \ref{lem: IG estimates}. In general, we have the following computation for $R^{(t)}_{f,g,l}$.
		\begin{flalign*}
			&R^{(t)}_{f,g,l}h(\xi)\\
			=&(-1)^l\int_{\bn^2}\Phi^{(t)}_{n,l}(|\varphi_z(w)|^2)\frac{\sum_{i_1,\ldots,i_l,j_1,\ldots,j_l=1}^nI^{e_{i_1,\ldots,i_l},e_{j_1,\ldots,j_l}}(z-w)F_{i_l,\ldots,i_1;j_l,\ldots,j_1}}{|1-\la z,w\ra|^{2l}}\BKt_w(z)\intd\lambda_t(w)\intd\lambda_t(z)\\
			=&(-1)^l\frac{\BFt_{n+l}\Phi^{(t)}_{n,l}(0)}{B(n,t+1)}\int_{\bn}(1-|z|^2)^{-2l}\sum_{i_1,\ldots,i_l,j_1,\ldots,j_l=1}^nd_{e_{i_1,\ldots,i_l},e_{j_1,\ldots,j_l}}(z)F_{i_l,\ldots,i_1;j_l,\ldots,j_1}(z,z)\intd\lambda_t(z)\\
			&+(-1)^{l+1}\int_{\bn^2}\intd\lambda_t(z)\intd\lambda_t(w)\Phi^{(t)}_{n,l+1}(|\varphi_z(w)|^2)\BKt_w(z)\\
			&\hspace{2cm}\cdot \frac{\sum_{i_1,\ldots,i_{l+1},j_1,\ldots,j_{l+1}=1}^nI^{e_{i_1,\ldots,i_{l+1},e_{j_1,\ldots,j_{l+1}}}}(z-w)D_{i_{l+1},j_{l+1}}F_{i_l,\ldots,i_1;j_l,\ldots,j_1}(z,w)}{|1-\la z,w\ra|^{2(l+1)}}\\
			=&c_{l,t}\BTt_{C_l(f,g)}h(\xi)+R^{(t)}_{f,g,l+1}h(\xi).
		\end{flalign*}
		This proves the formulas in Theorem \ref{thm: quantization bergman} for $h\in\Hol(\overline{\bn})$. By \cite[Lemma 8.4]{TWZ:semicommutator}, we have the following identity,
		\[
		c_{0,t}=\frac{\BFt_n1(0)}{B(n,t+1)}=1.
		\]
		The estimate (a) for $c_{l,t}$ follows from Lemma \ref{lem: Phi esitmates}. The formulas for $C_0(f,g)$ follows from direct computation. We can also see from the formula that $C_l(f,g)$ does not depend on the choice of an orthonormal basis. Using \ref{eqn: d at (1,0...0)} we directly verify that $C_1(f,g)-C_1(g,f)=\frac{-i}{n}\{f,g\}$. By Lemma \ref{lem: IG estimates}, we have the following estimates,
		\begin{equation}\label{eqn: temp Sfgk+1 estimate}
		\begin{split}
			&|S_{f,g,k+1}(z,w)|\\
			=&|1-\la w,z\ra|^{-2(k+1)}\bigg|\sum_{i_1,\ldots,i_{k+1},j_{k+1},\ldots,j_l=1}^nI^{e_{i_1,\ldots,i_{k+1}},e_{j_1,\ldots,j_{k+1}}}(z-w)G_{i_1,\ldots,i_{k+1}; j_1,\ldots,j_{k+1}}(z,w)\bigg|\\
			\lesssim&|\varphi_z(w)|^{2(k+1)}|1-\la z,w\ra|.
		\end{split}
		\end{equation}
		Thus it follows from Theorem \ref{thm: integral Schatten membership} that $\|R^{(t)}_{f,g,k+1}\|\lesssim_k t^{-k-1}$.
		
		By \cite[Lemma 5.3]{TWZ:semicommutator},
		\begin{equation*}
		\BFt_{n+1}\Phi^{(t)}_{n,1}(0)=n!t^{-n-1}+o(t^{-n-2})
		\end{equation*}
	as $t$ tends to infinity.
		Therefore, we arrive at the following estimate
		\[
		c_{1,t}=\frac{\BFt_{n+1}\Phi_{n,1}^{(t)}(0)}{B(n,t+1)}=\frac{B(n+1,t+1)}{B(n,t+1)}+O(\frac{B(n+2,t+1)}{B(n,t+1)})=nt^{-1}+O(t^{-2}).
		\]
		This completes the proof of Theorem \ref{thm: quantization bergman}.
	\end{proof}
	
	\begin{proof}[{\bf Proof of Corollary \ref{cor: quantization schatten dim 1}}]
		In the case when $n=1$, the estimates in the proof of Lemma \ref{lem: ABC estimates} are improved into
		\begin{itemize}
			\item[(1)] $|A_{x,y}(z,w)|\lesssim|\varphi_z(w)|^{x+y}|1-\la z,w\ra|^{x+y}$;
			\item[(2)] $|B_1(z,w)|\lesssim|\varphi_z(w)||1-\la z,w\ra|$,\quad$B_2(z,w)|\lesssim|\varphi_z(w)||1-\la z,w\ra|$;
			\item[(3)] $|C(z,w)|\lesssim|\varphi_z(w)|^2|1-\la z,w\ra|^2$.
		\end{itemize}
		This leads to
		\begin{equation}\label{eqn: IG estimate 1 for dim 1}
			\bigg|\sum_{i_1,\ldots,i_l,j_1,\ldots,j_l=1}^nI^{e_{i_1,\ldots,i_l},e_{j_1,\ldots,j_l}}(z-w)G_{i_1,\ldots,i_l; j_1,\ldots,j_l}(z,w)\bigg|\lesssim|\varphi_z(w)|^{2l}|1-\la z,w\ra|^{2l+2}
		\end{equation}
		and then
		\[
		|S_{f,g,k+1}(z,w)|\lesssim|\varphi_z(w)|^{2(k+1)}|1-\la z,w\ra|^2.
		\]
		The corollary follows from Theorem \ref{thm: integral Schatten membership}.
	\end{proof}
	
	\begin{proof}[{\bf Proof of Corollary \ref{cor: quatization schatten dim n}}]
		The corollary follows from \eqref{eqn: temp Sfgk+1 estimate} and Theorem \ref{thm: integral Schatten membership}.
	\end{proof}

	\begin{proof}[{\bf Proof of Corollary \ref{cor: Rabfgk membership}}]
		As in Definition \ref{defn: Sabfgk}, for $z\in\bn, z\neq0$, let $\mathbf{e}_z$ be an orthonormal basis of $\cn$ so that $e_{z,1}=\frac{z}{|z|}$. Under the basis $\mathbf{e}_z$, for $0\leq a, b\leq k+1$, $S^{a,b}_{f,g,k+1}(z,w)$ consists of the part of $S_{f,g,k+1}$ that contains $\partial^\alpha f(z)\bpartial^\beta g(w)$ with $\alpha_1=a, \beta_1=b$.

		Since $R^{(t)a,b}_{f,g,k+1}=\BPt T^{(t)a,b}_{f,g,k+1}$, where
		\[
		T^{(t)a,b}_{f,g,k+1}h(z)=\int_{\bn}\Phi^{(t)}_{n,k+1}(|\varphi_z(w)|^2)S^{a,b}_{f,g,k+1}(z,w)h(w)\BKt_w(z)\intd\lambda_t(w).
		\]
		To prove the Schatten class membership of the $R^{(t)a,b}_{f,g,k+1}$ operators, it amounts to prove the corresponding estimates for the kernel $S^{a,b}_{f,g,k+1}(z,w)$, and apply Theorem \ref{thm: integral Schatten membership}.
		
		Locally choose the basis $\mathbf{e}_z$ so that it varies smooth with respect to $z$.
		Under the basis $\mathbf{e}_z$, define
		\[
		A_{\alpha,\beta}(z,w)=I^{\alpha,\beta}(z-w)\partial^\alpha f(z)\bpartial^\beta g(w).
		\]
		Then by Lemma \ref{lem: Mobius basics}, we have the following bound,
		\[
		|A_{\alpha,\beta}(z,w)|\lesssim|\varphi_z(w)|^{|\alpha|+|\beta|}|1-\la z,w\ra|^{\frac{|\alpha|+|\beta|+\alpha_1+\beta_1}{2}}.
		\]
		Then $S^{(t)a,b}_{f,g,k+1}(z,w)$ is a finite linear combination of terms like:
		\[
		\frac{1}{|1-\la w,z\ra|^{2(k+1)}}\cdot (1-\la z,w\ra)^{s_1}\big[C(z,w)\big]^{s_2}\big[B_1(z,w)\big]^{s_3}\big[B_2(z,w)\big]^{s_4}A_{\alpha,\beta}(z,w),
		\]
		where $|\alpha|, |\beta|\geq1, 2s_2+s_3+s_4+|\alpha|+|\beta|=s_1+s_2+s_3+s_4=2(k+1),$ and $\alpha_1=a, \beta_1=b$. Therefore we obtain the following estimate
		\begin{flalign*}
			|S^{a,b}_{f,g,k+1}(z,w)|\lesssim& |\varphi_z(w)|^{2s_2+s_3+s_4+|\alpha|+|\beta|}|1-\la z,w\ra|^{-2(k+1)+s_1+s_2+s_3+s_4+\frac{|\alpha|+|\beta|+\alpha_1+\beta_1}{2}}\\
			=&|\varphi_z(w)|^{2(k+1)}|1-\la z,w\ra|^{\frac{2+a+b}{2}}.
		\end{flalign*}
		Thus by Theorem \ref{thm: integral Schatten membership}, $R^{(t)a,b}_{f,g,k+1}=\BPt T^{(t)a,b}_{f,g,k+1}\in\mathcal{S}^p, \forall p>\max\{\frac{n}{1+\frac{a+b}{2}},\frac{n}{k+1+\frac{t+1}{2}}\}$, and for such $p$, we have
		\[
		\|R^{(t) a,b}_{f,g,k+1}\|_{\mathcal{S}^p}\lesssim_{k,p}t^{-k-1+\frac{n}{p}}.
		\]
		If $f=\phi(|z|^2)$ for some $\phi$, then $R^{(t)0,b}_{f,g,k+1}=0$ for any $b$. If both $f, g$ are of such form then $R^{(t)a,0}_{f,g,k+1}, R^{(t)0,b}_{f,g,k+1}=0$ for any $a,b$. This gives the improved Schatten-class membership in (1) and (2), and proves Corollary \ref{cor: Rabfgk membership}.
		
	\end{proof}
	
	\begin{proof}[{\bf Proof of Corollary \ref{cor: semicom with radial functions}}]
		By Theorem \ref{thm: quantization bergman}, for $f, g\in\mathscr{C}^1(\overline{\bn})$, $\BTt_f\BTt_g-\BTt_{fg}=R^{(t)}_{f,g,1}$, where
		\[
		R^{(t)}_{f,g,1}h(z)=\int_{\bn^2}\Phi^{(t)}_{n,1}(|\varphi_z(w)|^2)S_{f,g,1}(z,w)h(w)\BKt_z(\xi)\BKt_w(z)\intd\lambda_t(w)\intd\lambda_t(z),
		\]
		with
		\[
		S_{f,g,1}(z,w)=\frac{-(1-\la z,w\ra)^2}{|1-\la z,w\ra|^2}\la\partial f(z),\overline{z-w}\ra\la\bpartial g(w),z-w\ra.
		\]
		If $f(z)=(1-|z|^2)^x$ then we have the following estimate
		\[
		\big|\la\partial f(z),\overline{z-w}\ra\big|=\big|(x-1)(1-|z|^2)^{x-1}\la\bar{z},\overline{z-w}\ra\big|\lesssim(1-|z|^2)^{x-1}|1-\la z,w\ra|\lesssim|1-\la z,w\ra|^x.
		\]
		Thus we arrive at the estimate,
		\[
		\big|\la\partial f(z),\overline{z-w}\ra\la\bpartial g(w),z-w\ra\big|\lesssim\begin{cases}
			|1-\la z,w\ra|^{x+1/2},&\text{ if one of $f$ or $g$ equals }(1-|z|^2)^x\\
			|1-\la z,w\ra|^{x+y},&\text{ if }f=(1-|z|^2)^x, g(z)=(1-|z|^2)^y.
		\end{cases}
		\]
		Corollary \ref{cor: semicom with radial functions} follows from the above inequality and Theorem \ref{thm: integral Schatten membership}. This completes the proof of Corollary \ref{cor: semicom with radial functions}.
	\end{proof}

	\section{Even and Odd Antisymmetrizations}\label{sec: even and odd}
	As explained in the introduction, the goal of this section is to prove the trace class membership of the antisymmetric sum $[\BTt_{f_1},\BTt_{f_2},\ldots,\BTt_{f_{2n}}]$ as well as the asymptotic trace formula \eqref{eqn:trace limit}. To begin with, we define even and odd partial antisymmetrizations, which are generalizations of semi-commutators.
	\begin{defn}\label{defn: AS partial}
		For $f, g\in C(\overline{\bn})$ and $t\geq-1$, denote the semi-commutator on $\bert$,
		\begin{equation}\label{eqn: semi-commutator sigma notation}
			\sigma_t(f,g)=\BTt_f\BTt_g-\BTt_{fg}.
		\end{equation}
		For $f_1, \ldots, f_n, g_1, \ldots, g_n\in C(\overline{\bn})$ and $t\geq-1$, define the following partial anti-symmetric sums.
		\begin{equation}\label{eqn: AS odd defn}
			[f_1, g_1,\ldots, f_n, g_n]^{\odd}_t=\sum_{\tau\in S_n}\sgn(\tau)\sigma_t(f_{\tau_1}, g_1)\ldots\sigma_t(f_{\tau_n}, g_n),
		\end{equation}
		and
		\begin{equation}\label{eqn: AS even defn}
			[f_1,g_1,\ldots,f_n,g_n]^{\even}_t=\sum_{\tau\in S_n}\sgn(\tau)\sigma_t(f_1, g_{\tau_1})\ldots\sigma_t(f_n,g_{\tau_n}).
		\end{equation}
	\end{defn}
	In the case when $n=1$, the operators above both agree with the semi-commutator $\sigma_t(f,g)$. In higher dimensions, the partial anti-symmetric sum \eqref{eqn: AS odd defn} and \eqref{eqn: AS even defn} generalize the semi-commutator \eqref{defn: AS partial}.
	
	\begin{rem}\label{rem: semi-commutator}
	Take the case $t>-1$ for example. Under the decomposition $L^2(\lambda_t)=\bert\oplus\bert^\perp$, we can write a multiplication operator $M_f$ as a block matrix
	\[
	M_f=\begin{bmatrix}
	\BTt_f&H^{(t)*}_{\bar{f}}\\
	\BHt_f&S^{(t)}_f
	\end{bmatrix},
	\]
	where $\BTt_f, \BHt_f$ and $S^{(t)}_f$ are Toeplitz operator, Hankel operator, and dual Toeplitz operator, respectively. From the equations
	\[
	\sigma_t(f,g)=-H^{(t)*}_{\bar{f}}H^{(t)}_g,
	\]
	and
	\[
	[\BTt_f,\BTt_g]=\sigma_t(f,g)-\sigma_t(g,f)=H^{(t)*}_{\bar{g}}H^{(t)}_f-H^{(t)*}_{\bar{f}}H^{(t)}_g,
	\]
	we can see that both operators come from products of off-diagonal terms under the above block matrix representation. Under this point of view, the even and odd antisymmetrizations defined above are linear combinations of alternating products of off-diagonal terms, with $f_i$ appearing in the top right, and $g_j$ appearing in the bottom left.
	\end{rem}
	\begin{thm}\label{thm: AS partial trace dimension n}
		Suppose $t\geq-1$ and $f_1, g_1, \ldots, f_n, g_n\in\mathscr{C}^{2}(\overline{\bn})$.
		Then the partial antisymmetrizations $[f_1,g_1,\ldots,f_n, g_n]_t^{\odd}$ and $[f_1,g_1,\ldots,f_n, g_n]_t^{\even}$ are in the trace class. Moreover,
		\begin{equation}\label{eqn: trace limit dim n}
			\lim_{t\to\infty}\Tr[f_1,g_1,\ldots,f_n,g_n]_t^{\odd}=\lim_{t\to\infty}\Tr[f_1,g_1,\ldots,f_n,g_n]_t^{\even}=\frac{1}{(2\pi i)^n}\int_{\bn}\partial f_1\wedge\bpartial g_1\wedge\ldots\wedge\partial f_n\wedge\bpartial g_n.
		\end{equation}
	\end{thm}
As an application of the above theorem, we obtain the following asymptotic estimate for the Schatten-$4$ norm of Hankel operators at dimension $2$. We can view it as a higher uncertainty principle.
\begin{cor}\label{cor: Hankel S4 estimate}
Suppose $f_1, f_2\in\mathscr{C}^2(\overline{\mathbb{B}_2})$ are supported inside $\mathbb{B}_2$. Then the infimum limit $$\varliminf\limits_{t\to\infty}\|H^{(t)}_{\bar{f}_1}\|_{\mathcal{S}^4}^2\|H^{(t)}_{\bar{f}_2}\|_{\mathcal{S}^4}^2$$ satisfies
\[
\varliminf_{t\to\infty}\|H^{(t)}_{\bar{f}_1}\|_{\mathcal{S}^4}^2\|H^{(t)}_{\bar{f}_2}\|_{\mathcal{S}^4}^2\geq\frac{1}{4\pi^2}\int_{\mathbb{B}_2}\partial f_1\wedge\partial f_2\wedge\overline{\partial f_1\wedge\partial f_2}.
\]
\end{cor}
\begin{proof}
Take $g_1=\bar{f}_1$, $g_2=\bar{f}_2$. By definition, we compute the expression of $[f_1,\bar{f}_1,f_2,\bar{f}_2]^{\even}$,
\begin{flalign*}
[f_1,\bar{f}_1,f_2,\bar{f}_2]^{\even}=&\sigma_t(f_1,\bar{f}_1)\sigma_t(f_2,\bar{f}_2)-\sigma_t(f_1,\bar{f}_2)\sigma_t(f_2,\bar{f}_1)\\
=&H^{(t)*}_{\bar{f}_1}H^{(t)}_{\bar{f}_1}H^{(t)*}_{\bar{f}_2}H^{(t)}_{\bar{f}_2}-H^{(t)*}_{\bar{f}_1}H^{(t)}_{\bar{f}_2}H^{(t)*}_{\bar{f}_2}H^{(t)}_{\bar{f}_1}\\
=&|H^{(t)}_{\bar{f}_1}|^2|H^{(t)}_{\bar{f}_2}|^2-\bigg(H^{(t)*}_{\bar{f}_2}H^{(t)}_{\bar{f}_1}\bigg)^*H^{(t)*}_{\bar{f}_2}H^{(t)}_{\bar{f}_1}.
\end{flalign*}
Then by Theorem \ref{thm: AS partial trace dimension n}, we have the following limit, 
\begin{flalign*}
\Tr [f_1,\bar{f}_1,f_2,\bar{f}_2]^{\even}=&\Tr |H^{(t)}_{\bar{f}_1}|^2|H^{(t)}_{\bar{f}_2}|^2-\Tr\bigg(H^{(t)*}_{\bar{f}_2}H^{(t)}_{\bar{f}_1}\bigg)^*H^{(t)*}_{\bar{f}_2}H^{(t)}_{\bar{f}_1}\\
\to&\frac{1}{(2\pi i)^2}\int_{\mathbb{B}_2}\partial f_1\wedge\bpartial\bar{f}_1\wedge\partial f_2\wedge\bpartial\bar{f}_2\\
=&\frac{1}{4\pi^2}\int_{\mathbb{B}_2}\partial f_1\wedge\partial f_2\wedge\overline{\partial f_1\wedge\partial f_2},\quad t\to\infty.
\end{flalign*}
Notice that $\bigg(H^{(t)*}_{\bar{f}_2}H^{(t)}_{\bar{f}_1}\bigg)^*H^{(t)*}_{\bar{f}_2}H^{(t)}_{\bar{f}_1}\geq0$. The Cauchy-Schwartz inequality gives the following bound,
\[
\bigg|\Tr|H^{(t)}_{\bar{f}_1}|^2|H^{(t)}_{\bar{f}_2}|^2\bigg|=\bigg|\la|H^{(t)}_{\bar{f}_2}|^2,|H^{(t)}_{\bar{f}_1}|^2\ra_{\mathcal{S}^2}\bigg|\leq\la|H^{(t)}_{\bar{f}_2}|^2,|H^{(t)}_{\bar{f}_2}|^2\ra_{\mathcal{S}^2}^{1/2}\la|H^{(t)}_{\bar{f}_1}|^2,|H^{(t)}_{\bar{f}_1}|^2\ra_{\mathcal{S}^2}^{1/2}= \|H^{(t)}_{\bar{f}_2}\|_{\mathcal{S}^4}^2\|H^{(t)}_{\bar{f}_1}\|_{\mathcal{S}^4}^2.
\] 
We have the following bound from the above computation 
\[
\varliminf_{t\to\infty}\|H^{(t)}_{\bar{f}_2}\|_{\mathcal{S}^4}^2\|H^{(t)}_{\bar{f}_1}\|_{\mathcal{S}^4}^2\geq \frac{1}{4\pi^2}\int_{\mathbb{B}_2}\partial f_1\wedge\partial f_2\wedge\overline{\partial f_1\wedge\partial f_2}.
\]
This completes the proof of Corollary \ref{cor: Hankel S4 estimate}.
\end{proof}
	
	\begin{rem}\label{rem: even odd AS}
		Both $[f_1, g_1,\ldots, f_n, g_n]^{\even}_t$ and $[f_1, g_1,\ldots, f_n, g_n]^{\odd}_t$ are sums of compositions of $n$ semi-commutators, each belonging to $\mathcal{S}^p, \forall p>n$. This proves that they belong to $\mathcal{S}^p$, $\forall p>1$. However, the trace class membership of these operators relies on some higher order cancellation and is therefore nontrivial. As an example, we can show by direct computation that each $\sigma_t(z_i,\bz_i)$ is the diagonal operator under the basis $\{\frac{z^\alpha}{\|z^\alpha\|}\}_{\alpha\in\ind}$, with entry $-\frac{|\alpha|-\alpha_i+n+t}{(n+|\alpha|+t)(n+|\alpha|+t+1)}$ at $\alpha\in\ind$. Therefore $\sigma_t(z_i,\bz_i)\notin\mathcal{S}^n$. (In fact, in \cite{Zhu:schattenHankel} it was proved that for $f\in H^\infty(\bn)$, $\sigma_0(f,\bar{f})\in\mathcal{S}^n$ if and only if $f$ is constant.) However, the theorem states that
		\[
		[z_1,\bz_1,\ldots,z_n,\bz_n]^\odd_t=\sum_{\tau\in S_n}\sgn(\tau)\sigma_t(z_{\tau_1},\bz_1)\ldots\sigma_t(z_{\tau_n},\bz_n)\in\mathcal{S}^1.
		\]
		The partial anti-symmetrization needs to be carefully chosen so that the higher cancellation works. For example,
		\[
		\sigma_t(z_1,\bz_1)\ldots\sigma_t(z_n,\bz_n)-\sigma_t(\bz_1,z_2)\sigma_t(\bz_2, z_3)\ldots\sigma_t(\bz_n,z_1)=\sigma_t(z_1,\bz_1)\ldots\sigma_t(z_n,\bz_n)\notin\mathcal{S}^1.
		\]
		Therefore simply taking anti-symmetrization over a rotation generally does not give a trace class operator. The above example also shows that the Connes-Chern character \eqref{eq:chern char} for the Toeplitz extension is in general not well-defined at $p=n$. At $p=n+1$, by Theorem \ref{thm: quantization bergman}, for $f, g\in\mathscr{C}^2(\overline{\bn})$,
		\[
		\|\sigma_t(f,g)\|_{\mathcal{S}^{n+1}}=\|R^{(t)}_{f,g,1}\|_{\mathcal{S}^{n+1}}\lesssim t^{-1+\frac{n}{n+1}}=t^{-\frac{1}{n+1}}.
		\]
		Therefore for $f_0,\ldots, f_{2n+1}\in\mathscr{C}^2(\overline{\bn})$, the Connes-Chern character at $p=n+1$ satisfies
		\[
		|\tau_t(f_0,\ldots,f_{2p-1})|\lesssim t^{-1}\to0,\quad t\to\infty.
		\]
		In particular, its value depends on $t$. In Subsection \ref{subsec: connes chern}, we consider the Connes-Chern character at $p>n$ after multiplying a suitable power of $t$.
	\end{rem}
	
	Theorem \ref{thm: AS partial trace dimension n} immediately leads to the trace class membership and asymptotic trace formula for the full anti-symmetric sum, which consist an important part of the proof of our main result, Theorem \ref{thm: main}.
	\begin{cor}\label{cor: AS sum trace membership and trace limit}
		Suppose $f_1, f_2, \ldots, f_{2n}\in\mathscr{C}^2(\overline{\bn})$ and $t\geq-1$. Then $[\BTt_{f_1}, \BTt_{f_2}, \ldots, \BTt_{f_{2n}}]$ is in the trace class $\mathcal{S}^1$. Moreover,
		\begin{equation}\label{eqn: AS sum trace limit}
			\lim_{t\to\infty}\Tr[\BTt_{f_1}, \BTt_{f_2}, \ldots, \BTt_{f_{2n}}]=\frac{n!}{(2\pi i)^n}\int_{\bn}\intd f_1\wedge\intd f_2\wedge\ldots\wedge\intd f_{2n}.
		\end{equation}
	\end{cor}
	\begin{proof}
		By Remark \ref{rem: semi-commutator}, it is easy to see the following identity
		\[
		[\BTt_{f_1}, \BTt_{f_2},\ldots,\BTt_{f_{2n}}]=\frac{1}{n!}\sum_{\tau\in S_{2n}}\sgn(\tau)[f_{\tau_1}, f_{\tau_2},\ldots,f_{\tau_{2n}}]^{\odd}.
		\]
		Therefore it follows from Theorem \ref{thm: AS partial trace dimension n} that $[\BTt_{f_1},\BTt_{f_2},\ldots,\BTt_{f_{2n}}]\in\mathcal{S}^1$. This proves the trace class membership.
		
		Let $H$ be the collection of subsets of $\{1,\ldots,2n\}$ consisting of $n$ elements. For each $a\in H$, let $[a]$ be the subset of $S_{2n}$ that sends $\{1,3,\ldots,2n-1\}$ to $a$. Then $\#[a]=(n!)^2$ for any $a\in H$. In each $[a]$ there is a unique permutation $\tau_a$ that satisfies
		\[
		\tau_a(1)<\tau_a(3)<\ldots<\tau_a(2n-1),\quad\tau_a(2)<\tau_a(4)<\ldots<\tau_a(2n).
		\]
		Then we compute the limit of the trace of full antisymmetrization as follows,
		\begin{flalign*}
			&\lim_{t\to\infty}\Tr[\BTt_{f_1},\BTt_{f_2},\ldots,\BTt_{f_{2n}}]\\
			=&\frac{1}{n!}\sum_{\tau\in S_{2n}}\sgn(\tau)\frac{1}{(2\pi i)^n}\int_{\bn}\partial f_{\tau_1}\wedge\bpartial f_{\tau_2}\wedge\ldots\wedge\partial f_{\tau_{2n-1}}\wedge\bpartial f_{\tau_{2n}}\\
			=&\frac{1}{(2\pi i)^nn!}\sum_{a\in H}\sum_{\tau\in[a]}\sgn(\tau)\int_{\bn}\partial f_{\tau_1}\wedge\bpartial f_{\tau_2}\wedge\ldots\wedge\partial f_{\tau_{2n-1}}\wedge\bpartial f_{\tau_{2n}}\\
			=&\frac{1}{(2\pi i)^nn!}\sum_{a\in H}(n!)^2\sgn(\tau_a)\int_{\bn}\partial f_{\tau_a(1)}\wedge\bpartial f_{\tau_a(2)}\wedge\ldots\wedge\partial f_{\tau_a(2n-1)}\wedge\bpartial f_{\tau_a(2n)}\\
			=&\frac{n!}{(2\pi i)^n}\int_{\bn}\intd f_1\wedge\intd f_2\wedge\ldots\wedge\intd f_{2n}.
		\end{flalign*}
		This proves \eqref{eqn: AS sum trace limit}.
	\end{proof}
	\subsection{Proof of Theorem \ref{thm: AS partial trace dimension n} ($t>-1$)}\label{subsec: pf of thm t>-1}
		Since $\sigma_t(f,g)=\bigg[\sigma_t(\bar{g},\bar{f})\bigg]^*$, we can verify that
		\[
		[f_1,g_1,\ldots,f_n,g_n]_t^{\odd}=\bigg([\bar{g}_n,\bar{f_n},\ldots,\bar{g}_1,\bar{f_1}]_t^{\even }\bigg)^*.
		\]
		Thus it suffices to prove the results for the odd partial anti-symmetric sums.
		\begin{notation}
		For two operators $A$ and $B$, temporarily denote $A\sim_w B$ if $A-B$ is a trace class operator. If $A_t$ and $B_t$ are parameterized families of operators on $\bert$, temporarily denote $A_t\sim_a B_t$ if $A_t-B_t$ are trace class operators on $\bert$ with trace norm tending to zero as $t\to\infty$.
		\end{notation}
		~
		
		\noindent{\bf Part 1:} We prove 
		\[
		[f_1,g_1,\ldots,f_n,g_n]_t^{\odd}\in\mathcal{S}^1, \forall t>-1.
		\]
		
		By Theorem \ref{thm: quantization bergman} and Corollary \ref{cor: Rabfgk membership}, for $f, g\in\mathscr{C}^2(\overline{\bn})$, recall 
		\[
		\sigma_t(f,g)=R^{(t)}_{f,g,1},
		\]
		and
		\[
		R^{(t)}_{f,g,1}=\sum_{a,b=0,1}R^{(t)a,b}_{f,g,1},
		\]
		where
		\[
		R^{(t)a,b}_{f,g,1}\in\mathcal{S}^p, \forall p>\max\bigg\{\frac{n}{1+\frac{a+b}{2}},\frac{n}{1+\frac{t+1}{2}}\bigg\}.
		\]
		In particular, if one of $a, b$ is non-zero, then $R^{(t)a,b}_{f,g,1}\in\mathcal{S}^p$ for some $p<n$. It follows immediately that for $f_1, g_1, \ldots, f_n, g_n\in\mathscr{C}^2(\overline{\bn})$, $\tau\in S_n$, 
		\[
		\sigma_t(f_{\tau_1},g_1)\sigma_t(f_{\tau_2},g_2)\ldots\sigma_t(f_{\tau_n},g_n)\sim_w R^{(t)0,0}_{f_{\tau_1},g_1,1}R^{(t)0,0}_{f_{\tau_2},g_2,1}\ldots R^{(t)0,0}_{f_{\tau_n},g_n,1}.
		\]
		Thus in order to prove $[f_1,g_1,\ldots,f_n,g_n]^{\odd}_t\in\mathcal{S}^1$, it remains to prove
		\begin{equation}\label{eqn: Rfg00 AS trace class}
			\sum_{\tau\in S_n}\sgn(\tau)R^{(t)0,0}_{f_{\tau_1},g_1,1}R^{(t)0,0}_{f_{\tau_2},g_2,1}\ldots R^{(t)0,0}_{f_{\tau_n},g_n,1}\in\mathcal{S}^1.
		\end{equation}
		Recall that by definition the integral kernel of $R^{(t)0,0}_{f,g,1}$ consists of the part that involves only complex tangential derivatives of $f$ and $g$. Tracing back in Theorem \ref{thm: quantization bergman} we have the following formula
		\begin{flalign*}
			R^{(t)0,0}_{f,g,1}h(\xi)=&-\int_{\bn^2}\Phi^{(t)}_{n,1}(|\varphi_z(w)|^2)V_1(z,w)\la Q_{\bar{z}}\partial f(z),\overline{z-w}\ra\la Q_z\bpartial g(w),z-w\ra \\
			&~~~~~~~\cdot h(w)\BKt_z(\xi)\BKt_w(z)\intd\lambda_t(w)\intd\lambda_t(z),
		\end{flalign*}
		where $V_1(z,w)=\frac{1-\la z,w\ra}{1-\la w,z\ra}$ is bounded.
		For any $j=1,\ldots,n$ and $f, g\in\mathscr{C}^2(\overline{\bn})$, define
		\[
		F_{\Lambda_{g,1,j}}(z,w)=-V_1(z,w)(z_j-w_j)\la Q_z\bpartial g(w),z-w\ra \BKt_w(z),\quad z, w\in\bn.
		\]
		Define the operator
		\begin{flalign}\label{eqn: temp 1}
			\Lambda_{g,1,j}h(z)=&\int_{\bn}\Phi_{n,1}^{(t)}(|\varphi_z(w)|^2)F_{\Lambda_{g,1,j}}(z,w)h(w)\intd\lambda_t(w).
		\end{flalign}
		Let $[Q_{\bar{z}}\partial f(z)]_j$ be the $j$-th entry of the vector $Q_{\bar{z}}\partial f(z)$. Then by the above expressions, we have the following expression
		\[
		R^{(t)0,0}_{f,g,1}=\sum_{j=1}^n\BPt M^{(t)}_{[Q_{\bar{z}}\partial f]_j}\Lambda_{g,1,j}.
		\]
		By Lemma \ref{lem: Mobius basics} we have the bound
		\[
		\big|F_{\Lambda_{g,1,i}}(z,w)\big|\lesssim|\varphi_z(w)|^2\frac{1}{|1-\la z,w\ra|^{n+1+t-1}},\quad z, w\in\bn.
		\]
		Thus by Theorem \ref{thm: integral Schatten membership} and Corollary \ref{cor: commutator schatten membership with phi}, for any Lipschitz function $u$ on $\bn$, $j=1,\ldots,n$, the following hold.
		\begin{itemize}
			\item Both $\BPt\Lambda_{g,1,j}$ and $\Lambda_{g,1,j}\BPt$ are in $\mathcal{S}^p$ for any $p>n$, with Schatten-$p$ norm $\lesssim_{p}t^{-1+\frac{n}{p}}$.
			\item For any Lipschitz function $u$, both $\BPt[\Lambda_{g,1,j},M^{(t)}_u]$ and $[\Lambda_{g,1,j}, M^{(t)}_u]\BPt$ are in $\mathcal{S}^p$ for any $p>\max\{\frac{2n}{3},\frac{n}{1+\frac{t+1}{2}}\}$, with Schatten-$p$ norm $\lesssim_{p}t^{-1+\frac{n}{p}}$. In particular, these operators belong to $\mathcal{S}^p$ for some $p<n$.
		\end{itemize}
		Also, by Corollary \ref{cor: commutator schatten membership no phi}, we have 
		\begin{itemize}
			\item For any Lipschitz function $u$, $[\BPt, M_u^{(t)}]\in\mathcal{S}^p$ for any $p>2n$. 
		\end{itemize}
		Denote
		\[
		u_{i,j}(z)=[Q_{\bar{z}}\partial f_i(z)]_j,\quad i, j=1, \ldots, n.
		\]
		Then by the above discussion, we compute the product of semi-commutators as follows.
		\begin{flalign*}
			&\sigma_t(f_1, g_1)\sigma_t(f_2,g_2)\ldots\sigma_t(f_n, g_n)\nonumber\\
			\sim_w&R^{(t)0,0}_{f_1,g_1,1}R^{(t)0,0}_{f_2,g_2,1}\ldots R^{(t)0,0}_{f_n,g_n,1}\\
			=&\sum_{j_1,\ldots,j_n=1}^n\BPt M^{(t)}_{u_{1,j_1}}\Lambda_{g_1,1,j_1}\BPt M^{(t)}_{u_{2,j_2}}\Lambda_{g_2,1,j_2}\ldots\BPt M^{(t)}_{u_{n,j_n}}\Lambda_{g_n,1,i_n}\BPt\nonumber\\
			=&\sum_{j_1,\ldots,j_n=1}^n\BPt \Lambda_{g_1,1,j_1}\BPt M^{(t)}_{u_{1,j_1}} M^{(t)}_{u_{2,j_2}}\Lambda_{g_2,1,j_2}\ldots\BPt M^{(t)}_{u_{n,j_n}}\Lambda_{g_n,1,i_n}\BPt\nonumber\\
			&+\sum_{j_1,\ldots,j_n=1}^n\BPt [M^{(t)}_{u_{1,j_1}},\Lambda_{g_1,1,j_1}]\BPt M^{(t)}_{u_{2,j_2}}\Lambda_{g_2,1,j_2}\ldots\BPt M^{(t)}_{u_{n,j_n}}\Lambda_{g_n,1,i_n}\BPt\nonumber\\
			&+\sum_{j_1,\ldots,j_n=1}^n\BPt \Lambda_{g_1,1,j_1}[M^{(t)}_{u_{1,j_1}},\BPt] M^{(t)}_{u_{2,j_2}}\Lambda_{g_2,1,j_2}\ldots\BPt M^{(t)}_{u_{n,j_n}}\Lambda_{g_n,1,i_n}\BPt\nonumber\\
			\sim_w&\sum_{j_1,\ldots,j_n=1}^n\BPt \Lambda_{g_1,1,j_1}\BPt M^{(t)}_{u_{1,j_1}} M^{(t)}_{u_{2,j_2}}\Lambda_{g_2,1,j_2}\ldots\BPt M^{(t)}_{u_{n,j_n}}\Lambda_{g_n,1,i_n}\BPt\nonumber.
		\end{flalign*}
		Continuing like this, we obtain the following expression
		\begin{flalign}\label{eqn: temp 3}
			&\sigma_t(f_1, g_1)\sigma_t(f_2,g_2)\ldots\sigma_t(f_n, g_n)\nonumber\\
			\sim_w&\sum_{j_1,\ldots,j_n=1}^n\BPt\Lambda_{g_1,1,j_1}\BPt \Lambda_{g_2,1,j_2}\ldots\BPt\Lambda_{g_n,1,j_n} M^{(t)}_{u_{1,j_1}}M^{(t)}_{u_{2,j_2}}\ldots M^{(t)}_{u_{n,j_n}}\BPt\\
			:=&\Theta^{(t)}_{f_1,g_1,\ldots,f_n,g_n}.\nonumber
		\end{flalign}
		Writing the operator above in integral form, for any $h\in\bert$,
		\begin{flalign*}
			\Theta^{(t)}_{f_1,g_1,\ldots,f_n,g_n}h(\xi)=&\int_{\bn^{2n}}\bigg[\prod_{i=1}^n\Phi_{n,1}^{(t)}(|\varphi_{z_i}(w_i)|^2)\bigg]\cdot\bigg[\prod_{i=1}^n\la Q_{z_i}\bpartial g_i(w_i),z_i-w_i\ra\bigg]\cdot\bigg[\prod_{i=1}^nV_1(z_i,w_i)\bigg]\\
			&\cdot\bigg[\prod_{i=1}^n\la Q_{\bar{w}_n}\partial f_i(w_n),\overline{z_i-w_i}\ra\bigg]h(w_n)\\
			&\cdot\BKt_{z_1}(\xi)\BKt_{w_1}(z_1)\ldots\BKt_{z_n}(w_{n-1})\BKt_{w_n}(z_n)\intd\lambda_t(w_n)\intd\lambda_t(z_n)\ldots\intd\lambda_t(w_1)\intd\lambda_t(z_1).
		\end{flalign*}
		We claim that $\Theta^{(t)}_{f_1,g_1,\ldots,f_n,g_n}$ vanishes after antisymmetrization over the $f$-symbols. To show this property, it suffices to show that
		\begin{equation}\label{eqn: Qpartial antisym =0}
			\sum_{\tau\in S_n}\sgn(\tau)\prod_{i=1}^n\la Q_{\bar{w}_n}\partial f_{\tau_i}(w_n),\overline{z_i-w_i}\ra=0.
		\end{equation}
	Each $\la Q_{\bar{w}_n}\partial f_{\tau_i}(w_n),\overline{z_i-w_i}\ra$ is independent of the choice of coordinates. Thus we may assume without loss of generality that $w_n=(r,0,\ldots,0)$. In this case, we have the expression
	\[
	\la Q_{\bar{w}_n}\partial f_{\tau_i}(w_n),\overline{z_i-w_i}\ra=\sum_{j=2}^n(z_{i,j}-w_{i,j})\partial_j f_{\tau_i}(w_n).
	\]
	Then the above equals
		\begin{flalign*}
			&\sum_{j_1,\ldots,j_n=2}^n(z_{1,j_1}-w_{1,j_1})\ldots(z_{n,j_n}-w_{n,j_n})\sum_{\tau\in S_n}\sgn(\tau)\partial_{j_1}f_{\tau_1}(w_n)\partial_{j_2}f_{\tau_2}(w_n)\ldots\partial_{j_n}f_{\tau_n}(w_n).
		\end{flalign*}
		Since $j_1,\ldots,j_n$ takes value in $\{2,\ldots,n\}$, at least two indices are equal. This implies that
		\[
		\sum_{\tau\in S_n}\sgn(\tau)\partial_{j_1}f_{\tau_1}(w_n)\partial_{j_2}f_{\tau_2}(w_n)\ldots\partial_{j_n}f_{\tau_n}(w_n)=0.
		\]
		Therefore \eqref{eqn: Qpartial antisym =0} holds. We conclude from this fact that the anti-symmetric sum
		\[
		\sum_{\tau\in S_n}\sgn(\tau)\Theta^{(t)}_{f_{\tau_1},g_1,\ldots,f_{\tau_n},g_n}
		\]
		equals zero, and complete the proof of 
		\[
		[f_1,g_1,\ldots,f_n,g_n]^{\odd}_t\in\mathcal{S}^1.
		\]
		
		~
		
		\noindent{\bf Part 2:}  We prove 
		\[
		\lim_{t\to\infty}\Tr[f_1,g_1,\ldots,f_n,g_n]^{\odd}_t=\frac{1}{(2\pi i)^n}\int_{\bn}\partial f_1\wedge\bpartial g_1\wedge\ldots\wedge\partial f_n\wedge\bpartial g_n.
		\]
		
		For this part, we assume that $t$ is large enough.
		We use the quantization formula at $k=1$. By Theorem \ref{thm: quantization bergman}, Remark \ref{rem: CN CT} and Corollary \ref{cor: Rabfgk membership}, for $f, g\in\mathscr{C}^2(\overline{\bn})$, we have the following decomposition
		\begin{equation}\label{eqn: temp sigma t decomp R1 R2}
			\sigma_t(f,g)=R^{(t)}_{f,g,1}=c_{1,t}\BTt_{C_1(f,g)}+R^{(t)}_{f,g,2}.
		\end{equation}
		The following hold.
		\begin{enumerate}
			\item \begin{equation*}
				c_{1,t}=nt^{-1}+O(t^{-2}),
			\end{equation*}
			\item
			\begin{equation*}
				C_1(f,g)=C_N(f,g)+C_T(f,g),
			\end{equation*}
			where $C_N(f,g)$ ($C_T(f,g)$) denotes the part involving complex normal (tangential) derivatives of $f, g$,
			and
			\begin{equation*}
				\BTt_{C_N(f,g)}\in\mathcal{S}^p, \forall p>\frac{n}{2},\text{ with }\|\BTt_{C_N(f,g)}\|_{\mathcal{S}^p}\lesssim_p t^{\frac{n}{p}},
			\end{equation*}
			\begin{equation*}
				\BTt_{C_T(f,g)}\in\mathcal{S}^p, \forall p>n,\text{ with }\|\BTt_{C_T(f,g)}\|_{\mathcal{S}^p}\lesssim_p t^{\frac{n}{p}}.
			\end{equation*}
			\item
			\begin{equation*}
				R^{(t)}_{f,g,1}=\sum_{a,b=0,1}R^{(t)a,b}_{f,g,1},\quad R^{(t)}_{f,g,2}=\sum_{a,b=0,1,2}R^{(t)a,b}_{f,g,2},
			\end{equation*}
			where $a, b$ denote the order of derivatives on $f, g$ in the complex normal directions, and for large $t$,
			\begin{equation*}
				R^{(t)a,b}_{f,g,i}\in\mathcal{S}^p, \forall p>\frac{n}{1+\frac{a+b}{2}},\quad i=1,2,
			\end{equation*}
			\begin{equation*}
				\|R^{(t)a,b}_{f,g,i}\|_{\mathcal{S}^p}\lesssim_{p} t^{-i+\frac{n}{p}}.
			\end{equation*}
		\end{enumerate}
		The above will be the main tool for Part 2 and will be repeatedly used without reference. We will prove Part 2 by establishing the following properties.
		\begin{flalign}\label{eqn: sigma sim_a TC1}
			&[f_1,g_1,\ldots,f_n,g_n]_t^{\odd}\nonumber\\
			\sim_a&c_{1,t}^n\sum_{\tau\in S_n}\sgn(\tau)\BTt_{C_1(f_{\tau_1},g_1)}\BTt_{C_1(f_{\tau_2},g_2)}\ldots\BTt_{C_1(f_{\tau_n},g_n)},
		\end{flalign}
		
		\begin{flalign}\label{eqn: TC1 into one}
			&c_{1,t}^n\sum_{\tau\in S_n}\sgn(\tau)\BTt_{C_1(f_{\tau_1},g_1)}\BTt_{C_1(f_{\tau_2},g_2)}\ldots\BTt_{C_1(f_{\tau_n},g_n)}\nonumber\\
			\sim_a&c_{1,t}^n\sum_{\tau\in S_n}\sgn(\tau)\BTt_{C_1(f_{\tau_1},g_1)C_1(f_{\tau_2},g_2)C_1(f_{\tau_n},g_n)},
		\end{flalign}
		\begin{flalign}\label{eqn: Tr TC1 sum limit}
			&\Tr\bigg[c_{1,t}^n\sum_{\tau\in S_n}\sgn(\tau)\BTt_{C_1(f_{\tau_1},g_1)C_1(f_{\tau_2},g_2)C_1(f_{\tau_n},g_n)}\bigg]\nonumber\\
			\to& \frac{1}{(2\pi i)^n}\int_{\bn}\partial f_1\wedge\bpartial g_1\wedge\ldots\wedge\partial f_n\wedge\bpartial g_n, t\to\infty.
		\end{flalign}
		
		~
		
		\noindent{\bf Proof of \eqref{eqn: sigma sim_a TC1}.}
		By \eqref{eqn: temp sigma t decomp R1 R2}, we compute the product of semi-commutators
		\begin{flalign*}
			&\sigma_t(f_1,g_1)\sigma_t(f_2,g_2)\ldots\sigma_t(f_n,g_n)-c_{1,t}^n\BTt_{C_1(f_1,g_1)}\BTt_{C_1(f_2,g_2)}\ldots\BTt_{C_1(f_n,g_n)}\\
			=&R^{(t)}_{f_1,g_1,1}R^{(t)}_{f_2,g_2,1}\ldots R^{(t)}_{f_n,g_n,1} -\bigg(R^{(1)}_{f_1,g_1,1}-R^{(t)}_{f_1,g_1,2}\bigg)\bigg(R^{(1)}_{f_2,g_2,1}-R^{(t)}_{f_2,g_2,2}\bigg)\ldots\bigg(R^{(1)}_{f_n,g_n,1}-R^{(t)}_{f_n,g_n,2}\bigg)\\
			=&\sum_{(i_1,i_2,\ldots,i_n)\in X}\pm R^{(t)}_{f_1,g_1,i_1}R^{(t)}_{f_2,g_2,i_2}\ldots R^{(t)}_{f_n,g_n,i_n},
		\end{flalign*}
		where
		\[
		X=\{(i_1,i_2,\ldots,i_n): i_j=1, 2,\text{ and at least one }i_j=2\}.
		\]
		Therefore
		\begin{equation*}
			[f_1,g_1,\ldots,f_n,g_n]_t^{\odd}-c_{1,t}^n\sum_{\tau\in S_n}\sgn(\tau)\BTt_{C_1(f_{\tau_1},g_1)}\BTt_{C_1(f_{\tau_2},g_2)}\ldots\BTt_{C_1(f_{\tau_n},g_n)}
		\end{equation*}
		is a linear combination of operators of the form
		\[
		\sum_{\tau\in S_n}\sgn(\tau)R^{(t)}_{f_{\tau_1},g_1,i_1}R^{(t)}_{f_{\tau_2},g_2,i_2}\ldots R^{(t)}_{f_{\tau_n},g_n,i_n},
		\]
		where $(i_1,i_2,\ldots,i_n)\in X$. For (\ref{eqn: sigma sim_a TC1}), it suffices to prove
		\begin{equation}\label{eqn: R sum sima 0}
			\sum_{\tau\in S_n}\sgn(\tau)R^{(t)}_{f_{\tau_1},g_1,i_1}R^{(t)}_{f_{\tau_2},g_2,i_2}\ldots R^{(t)}_{f_{\tau_n},g_n,i_n}\sim_a0,\quad\forall (i_1,\ldots,i_n)\in X.
		\end{equation}
		We show the case when $(i_1,i_2,\ldots, i_n)=(2,1,\ldots,1)$, i.e.,
		\[
		\sum_{\tau\in S_n}\sgn(\tau)R^{(t)}_{f_{\tau_1},g_1,2}R^{(t)}_{f_{\tau_2},g_2,1}\ldots R^{(t)}_{f_{\tau_n},g_n,1}\sim_a0.
		\]
		First, each $R^{(t)}_{f,g,i}$ decomposes into the sum of $R^{(t)a,b}_{f,g,i}$. The operator $R^{(t)0,0}_{f,g,i}\in\mathcal{S}^p, \forall p>n,$ and
		\[
		\|R^{(t)0,0}_{f,g,1}\|_{\mathcal{S}^p}\lesssim t^{-1+\frac{n}{p}},\quad\|R^{(t)0,0}_{f,g,2}\|_{\mathcal{S}^p}\lesssim t^{-2+\frac{n}{p}}.
		\]
		If $(a,b)\neq(0,0)$ then the operator belongs to $\mathcal{S}^p$ for some $p<n$, with asymptotic Schatten-norm estimates
		\[
		\|R^{(t)a,b}_{f,g,1}\|_{\mathcal{S}^p}\lesssim t^{-1+\frac{n}{p}},\quad\|R^{(t)a,b}_{f,g,2}\|_{\mathcal{S}^p}\lesssim t^{-2+\frac{n}{p}}.
		\]
		Therefore for each $\tau\in S_n,$ we have the following equation
		\[
		R^{(t)}_{f_{\tau_1},g_1,2}R^{(t)}_{f_{\tau_2},g_2,1}\ldots R^{(t)}_{f_{\tau_n},g_n,1}=\sum_{\substack{a_1,b_1=0,1,2\\a_2,b_2,\ldots,a_n,b_n=0,1}} R^{(t)a_1,b_1}_{f_{\tau_1},g_1,2}R^{(t)a_2,b_2}_{f_{\tau_2},g_2,1}\ldots R^{(t)a_n,b_n}_{f_{\tau_n},g_n,1}.
		\]
		Suppose  that there is some $(a_k,b_k)\neq(0,0)$. Then we can choose $1\leq p_1,\ldots,p_n<\infty$ such that $\sum_j\frac{1}{p_j}=1$, and $R^{(t)a_j,b_j}_{f_{\tau_j},g_j,i_j}\in\mathcal{S}^{p_i},\ i=1,\ldots, n$. So
		\[
		R^{(t)a_1,b_1}_{f_{\tau_1},g_1,2}R^{(t)a_2,b_2}_{f_{\tau_2},g_2,1}\ldots R^{(t)a_n,b_n}_{f_{\tau_n},g_n,1}\in\mathcal{S}^1,
		\]
		with trace norm
		\[
		\lesssim t^{-2+\frac{n}{p_1}}\cdot t^{-1+\frac{n}{p_2}}\cdot\ldots\cdot t^{-1+\frac{n}{p_n}}=t^{-1}.
		\]
		It follows that $R^{(t)a_1,b_1}_{f_{\tau_1},g_1,2}R^{(t)a_2,b_2}_{f_{\tau_2},g_2,1}\ldots R^{(t)a_n,b_n}_{f_{\tau_n},g_n,1}\sim_a0$. And we conclude
		\begin{equation}\label{eqn: Rfg sima Rfg00}
			R^{(t)}_{f_{\tau_1},g_1,2}R^{(t)}_{f_{\tau_2},g_2,1}\ldots R^{(t)}_{f_{\tau_n},g_n,1}\sim_a R^{(t)0,0}_{f_{\tau_1},g_1,2}R^{(t)0,0}_{f_{\tau_2},g_2,1}\ldots R^{(t)0,0}_{f_{\tau_n},g_n,1}.
		\end{equation}
		Next, we show
		\begin{equation}\label{eqn: Rfg00 sum sima 0}
			\sum_{\tau\in S_n}\sgn(\tau)R^{(t)0,0}_{f_{\tau_1},g_1,2}R^{(t)0,0}_{f_{\tau_2},g_2,1}\ldots R^{(t)0,0}_{f_{\tau_n},g_n,1}\sim_a0.
		\end{equation}
		
		The proof is an almost verbatim repetition of the proof of \eqref{eqn: Rfg00 AS trace class}. Tracing back the definition, for $f, g\in\mathscr{C}^2(\overline{\bn})$ and $h\in\bert$, we compute $R_{f,g,2}^{(t)0,0}$,
		\begin{flalign}\label{eqn: R00fg2}
			R_{f,g,2}^{(t)0,0}h(\xi)=&\int_{\bn^2}\Phi_{n,2}^{(t)}(|\varphi_z(w)|^2)V_2(z,w)\la Q_{\bar{z}}\partial f(z),\overline{z-w}\ra\la Q_z\bpartial g(w),z-w\ra\\
			&~~~~~\cdot h(w)\BKt_z(\xi)\BKt_w(z)\intd\lambda_t(w)\intd\lambda_t(z),\nonumber
		\end{flalign}
		with
		\begin{flalign}\label{eqn: V2}
			V_2(z,w)=&\frac{\sum_{i_2,j_2=1}^nI^{e_{i_2},e_{j_2}}(z-w)\partial_{z_{i_2}}\bpartial_{w_{j_2}}(1-\la z,w\ra)^2}{(1-\la w,z\ra)^2}\\
			=&2\frac{1-\la z,w\ra}{(1-\la w,z\ra)^2}\cdot C(z,w)+\frac{B_1(z,w)B_2(z,w)}{(1-\la w,z\ra)^2},
		\end{flalign}
		where $C(z,w), B_1(z,w), B_2(z,w)$ are as in the proof of Lemma \ref{lem: IG estimates}. By the estimates in the proof, we get
		\begin{equation}\label{eqn: V2 estimate}
			|V_2(z,w)|\lesssim |\varphi_z(w)|^2.
		\end{equation}
		Similar as in the proof of \eqref{eqn: Rfg00 AS trace class}, we write
		\[
		R^{(t)0,0}_{f,g,2}=\sum_{j=1}^n\BPt M^{(t)}_{[Q_z\partial f]_j}\Lambda_{g,2,j},
		\]
		where
		\[
		\Lambda_{g,2,j}h(z)=\int_{\bn}\Phi_{n,2}^{(t)}(|\varphi_z(w)|^2)F_{\Lambda_{g,2,j}}(z,w)h(w)\intd\lambda_t(w),
		\]
		with
		\[
		F_{\Lambda_{g,2,j}}(z,w)=V_2(z,w)(z_j-w_j)\la Q_z\bpartial g(w),z-w\ra\BKt_w(z).
		\]
		By \eqref{eqn: V2 estimate} and Lemma \ref{lem: Mobius basics}, we have the following estimate
		\[
		\big|F_{\Lambda_{g,2,j}}(z,w)\big|\lesssim|\varphi_z(w)|^2\frac{1}{|1-\la z,w\ra|^{n+1+t-1}}. 
		\]
		Again, by Theorem \ref{thm: integral Schatten membership} and Corollary \ref{cor: commutator schatten membership with phi}, for any Lipschitz function $u$ on $\bn$, we have the following properties. 
		\begin{itemize}
			\item Both $\BPt\Lambda_{g,2,j}$ and $\Lambda_{g,2,j}\BPt$ are in $\mathcal{S}^p$ for any $p>n$, with Schatten-$p$ norm $\lesssim_{p}t^{-2+\frac{n}{p}}$.
			\item Both $\BPt[\Lambda_{g,2,j},M^{(t)}_u]$ and $[\Lambda_{g,2,j}, M^{(t)}_u]\BPt$ are in $\mathcal{S}^p$ for any $p>\frac{2n}{3}$, with Schatten-$p$ norm $\lesssim_{p}t^{-2+\frac{n}{p}}$.
			\item For any Lipschitz function $u$, $[\BPt, M_u^{(t)}]=[\BPt, M_u^{(t)}]\BPt-\BPt[\BPt, M_u^{(t)}]\in\mathcal{S}^p$ for any $p>2n$, with $\|[\BPt, M_u^{(t)}]\|_{\mathcal{S}^p}\lesssim_p t^{\frac{n}{p}}$.
		\end{itemize}
		Then as in the proof of \eqref{eqn: temp 3}, we compute the product of $R^{(t)0,0}_{f_{\tau_1}, f_1,2}$ and $R^{(t)0,0}_{f_{\tau_i}, g_i, 1}$ ($i=2, \cdots, n$)
		\begin{flalign*}
			&R^{(t)0,0}_{f_{\tau_1},g_1,2}R^{(t)0,0}_{f_{\tau_2},g_2,1}\ldots R^{(t)0,0}_{f_{\tau_n},g_n,1}\\
			=&\sum_{j_1,\ldots,j_n=1}^n\BPt M^{(t)}_{u_{\tau_1, j_1}}\Lambda_{g_1,2,j_1}\BPt M^{(t)}_{u_{\tau_2,j_2}}\Lambda_{g_2,1,j_2}\ldots\BPt M^{(t)}_{u_{\tau_n,j_n}}\Lambda_{g_n,1,j_n}\BPt\\
			\sim_a&\sum_{j_1,\ldots,j_n=1}^n\BPt\Lambda_{g_1,2,j_1}\BPt \Lambda_{g_2,1,j_2}\ldots\BPt\Lambda_{g_n,1,j_n} M^{(t)}_{u_{\tau_1,j_1}}M^{(t)}_{u_{\tau_2,j_2}}\ldots M^{(t)}_{u_{\tau_n,j_n}}\BPt\\
			:=&\Theta^{'(t)}_{f_{\tau_1},g_1,\ldots,f_{\tau_n},g_n}.
		\end{flalign*}
		Again, by \eqref{eqn: Qpartial antisym =0}, we conclude with the following equation
		\[
		\sum_{\tau\in S_n}\sgn(\tau)\Theta^{'(t)}_{f_{\tau_1},g_1,\ldots,f_{\tau_n},g_n}=0.
		\]
		This proves \eqref{eqn: Rfg00 sum sima 0}. And together with \eqref{eqn: Rfg sima Rfg00}, it proves \eqref{eqn: R sum sima 0} for $(i_1,\ldots, i_n)=(2,1,\ldots,1)$. The proof for general $(i_1,\ldots, i_n)\in X$ is an almost verbatim repetition of the above. This finishes the proof of \eqref{eqn: sigma sim_a TC1}.
		
		~
		
		\noindent{\bf Proof of \eqref{eqn: TC1 into one}.}
		Denote
		\[
		\phi_i(z)=(1-|z|^2)^i,\quad i=1, 2, \ldots.
		\]
		Suppose $f, g\in\mathscr{C}^2(\overline{\bn})$.
		For $z\neq 0$, let $\mathbf{e}_z=\{e_{z,1}, e_{z,2}, \ldots, e_{z, n}\}$ be an orthonormal basis of $\cn$ such that $e_{z,1}=\frac{z}{|z|}$. Let $\varphi:\cn\to[0,1]$ be a smooth function that equals $1$ inside $\frac{1}{4}\bn$ and vanishes outside $\frac{1}{2}\bn$. Under the basis $\mathbf{e}_z$, let
		\begin{equation}\label{eqn: Dfg1}
			D_{f,g,1}(z)=-\frac{1}{n}\bigg[\sum_{i=2}^n\partial_i f(z)\bpartial_i g(z)\bigg](1-\varphi(z)),
		\end{equation}
		and
		\begin{equation}\label{eqn: Dfg2}
		D_{f,g,2}(z)=-\frac{1}{n}\partial_1g(z)\bpartial_1g(z)(1-\varphi(z))+\frac{C_1(f,g)(z)\varphi(z)}{(1-|z|^2)^2}.
		\end{equation}
		Then by Remark \ref{rem: CN CT} and direct computation, we have the following decomposition
		\[
		C_1(f,g)=\phi_1D_{f,g,1}+\phi_2D_{f,g,2}.
		\]
		By Remark \ref{rem: C1 formula for any basis}, $D_{f,g,1}, D_{f,g,2}\in\mathscr{C}^1(\overline{\bn})$. With the decomposition we have the following formula
		\begin{flalign}\label{eqn: 4}
			&c_{1,t}^n\BTt_{C_1(f_1,g_1)}\BTt_{C_1(f_2,g_2)}\ldots\BTt_{C_1(f_n,g_n)}- \BTt_{c_{1,t}^nC_1(f_1,g_1)C_1(f_2,g_2)\ldots C_1(f_n,g_n)}\\
			=&c_{1,t}^n\sum_{i_1,\ldots,i_n=1,2}\bigg[\BTt_{\phi_{i_1}D_{f_1,g_1,i_1}}\BTt_{\phi_{i_2}D_{f_2,g_2,i_2}}\ldots\BTt_{\phi_{i_n}D_{f_n,g_n,i_n}}-\BTt_{\phi_{i_1}D_{f_1,g_1,i_1}\phi_{i_2}D_{f_2,g_2,i_2}\ldots\phi_{i_n}D_{f_n,g_n,i_n}}\bigg].\nonumber
		\end{flalign}
		By Theorem \ref{thm: quantization bergman}, Corollary \ref{cor: semicom with radial functions} and Lemma \ref{lem: Toeplitz Schatten Class} for $t$ large enough, $i, j=1, 2, \ldots$, and $u, v\in\mathscr{C}^1(\overline{\bn})$, we obtain the following estimates. 
		\[
		\BTt_u\BTt_v-\BTt_{uv}\in\mathcal{S}^p,~\forall p>n,\quad\text{ and }\|\BTt_u\BTt_v-\BTt_{uv}\|_{\mathcal{S}^p}\lesssim_p t^{-1+\frac{n}{p}},
		\]
		\[
		\BTt_{\phi_i}\BTt_{u}-\BTt_{\phi_iu}\in\mathcal{S}^p,~\forall p>\frac{n}{i+\frac{1}{2}}, \quad\text{ and }\|\BTt_{\phi_i}\BTt_{u}-\BTt_{\phi_iu}\|_{\mathcal{S}^p}\lesssim_p t^{-1+\frac{n}{p}},
		\]
		\[
		\BTt_{u}\BTt_{\phi_i}-\BTt_{\phi_iu}\in\mathcal{S}^p,~\forall p>\frac{n}{i+\frac{1}{2}}, \quad\text{ and }\|\BTt_{u}\BTt_{\phi_i}-\BTt_{\phi_iu}\|_{\mathcal{S}^p}\lesssim_p t^{-1+\frac{n}{p}},
		\]
		and
		\[
		\BTt_{\phi_i}\BTt_{\phi_j}-\BTt_{\phi_i\phi_j}\in\mathcal{S}^p,~\forall p>\frac{n}{i+j},\quad\text{ and }\|\BTt_{\phi_i}\BTt_{\phi_j}-\BTt_{\phi_i\phi_j}\|_{\mathcal{S}^p}\lesssim_p t^{-1+\frac{n}{p}},
		\]
		\[
		\BTt_{\phi_iu}\in\mathcal{S}^p,~\forall p>\frac{n}{i},\quad\text{ and }\|\BTt_{\phi_iu}\|_{\mathcal{S}^p}\lesssim_p t^{\frac{n}{p}},
		\]
		and
		\[
		c_{1,t}\approx t^{-1}.
		\]
		Therefore by the above estimates, we compute the product of $T^{(t)}_{\phi_{i_j}D_{f_j, g_j, i_j}}$.
		\begin{flalign*}
			&c_{1,t}^n\BTt_{\phi_{i_1}D_{f_1,g_1,i_1}}\BTt_{\phi_{i_2}D_{f_2,g_2,i_2}}\ldots\BTt_{\phi_{i_n}D_{f_n,g_n,i_n}}\\
			=&c_{1,t}^n\BTt_{\phi_{i_1}}\BTt_{D_{f_1,g_1,i_1}}\BTt_{\phi_{i_2}D_{f_2,g_2,i_2}}\ldots\BTt_{\phi_{i_n}D_{f_n,g_n,i_n}}\\
			&\hspace{4cm}-c_{1,t}^n\big(\BTt_{\phi_{i_1}}\BTt_{D_{f_1,g_1,i_1}}-\BTt_{\phi_{i_1}D_{f_1,g_1,i_1}}\big)\BTt_{\phi_{i_2}D_{f_2,g_2,i_2}}\ldots\BTt_{\phi_{i_n}D_{f_n,g_n,i_n}}\\
			\sim_a&c_{1,t}^n\BTt_{\phi_{i_1}}\BTt_{D_{f_1,g_1,i_1}}\BTt_{\phi_{i_2}D_{f_2,g_2,i_2}}\ldots\BTt_{\phi_{i_n}D_{f_n,g_n,i_n}}\\
			&\ldots\\
			\sim_a&c_{1,t}^n\BTt_{\phi_{i_1}}\BTt_{D_{f_1,g_1,i_1}}\BTt_{\phi_{i_2}}\BTt_{D_{f_2,g_2,i_2}}\ldots\BTt_{\phi_{i_n}}\BTt_{D_{f_n,g_n,i_n}}\\
			=&c_{1,t}^n\BTt_{\phi_{i_1}}\BTt_{\phi_{i_2}}\BTt_{D_{f_1,g_1,i_1}}\BTt_{D_{f_2,g_2,i_2}}\ldots\BTt_{\phi_{i_n}}\BTt_{D_{f_n,g_n,i_n}}\\
			&\hspace{4cm}+c_{1,t}^n\BTt_{\phi_{i_1}}\big[\BTt_{D_{f_1,g_1,i_1}},\BTt_{\phi_{i_2}}\big]\BTt_{D_{f_2,g_2,i_2}}\ldots\BTt_{\phi_{i_n}}\BTt_{D_{f_n,g_n,i_n}}\\
			\sim_a&c_{1,t}^n\BTt_{\phi_{i_1}}\BTt_{\phi_{i_2}}\BTt_{D_{f_1,g_1,i_1}}\BTt_{D_{f_2,g_2,i_2}}\ldots\BTt_{\phi_{i_n}}\BTt_{D_{f_n,g_n,i_n}}\\
			&\ldots\\
			\sim_a&c_{1,t}^n\BTt_{\phi_{i_1}}\ldots\BTt_{\phi_{i_n}}\BTt_{D_{f_1,g_1,i_1}}\BTt_{D_{f_2,g_2,i_2}}\ldots\BTt_{D_{f_n,g_n,i_n}}\\
			=&c_{1,t}^n\BTt_{\phi_{i_1}}\ldots\BTt_{\phi_{i_n}}\BTt_{D_{f_1,g_1,i_1}D_{f_2,g_2,i_2}}\ldots\BTt_{D_{f_n,g_n,i_n}}\\
			&\hspace{2cm}+c_{1,t}^n\BTt_{\phi_{i_1}}\ldots\BTt_{\phi_{i_n}}\bigg(\BTt_{D_{f_1,g_1,i_1}}\BTt_{D_{f_2,g_2,i_2}}-\BTt_{D_{f_1,g_1,i_1}D_{f_2,g_2,i_2}}\bigg)\BTt_{\phi_{i_3}}\ldots\BTt_{D_{f_n,g_n,i_n}}\\
			\sim_a&c_{1,t}^n\BTt_{\phi_{i_1}}\ldots\BTt_{\phi_{i_n}}\BTt_{D_{f_1,g_1,i_1}D_{f_2,g_2,i_2}}\BTt_{\phi_{i_3}}\ldots\BTt_{D_{f_n,g_n,i_n}}\\
			&\ldots\\
			\sim_a&c_{1,t}^n\BTt_{\phi_{i_1}}\ldots\BTt_{\phi_{i_n}}\BTt_{D_{f_1,g_1,i_1}D_{f_2,g_2,i_2}\ldots D_{f_n,g_n,i_n}}.
		\end{flalign*}
		If some $i_k=2$ in the above, we continue the above computation as follows,
		\begin{flalign*}
			\sim_a&c_{1,t}^n\BTt_{\phi_{i_1}\phi_{i_2}}\BTt_{\phi_{i_3}}\ldots\BTt_{\phi_{i_n}}\BTt_{D_{f_1,g_1,i_1}D_{f_2,g_2,i_2}\ldots D_{f_n,g_n,i_n}}\\
			&\ldots\\
			\sim_a&c_{1,t}^n\BTt_{\phi_{i_1}\phi_{i_2}\ldots\phi_{i_n}}\BTt_{D_{f_1,g_1,i_1}D_{f_2,g_2,i_2}\ldots D_{f_n,g_n,i_n}}\\
			\sim_a&c_{1,t}^n\BTt_{\phi_{i_1}\phi_{i_2}\ldots\phi_{i_n}D_{f_1,g_1,i_1}D_{f_2,g_2,i_2}\ldots D_{f_n,g_n,i_n}}.
		\end{flalign*}
		This proves that if  there is some $i_k=2$, then we have the following equation
		\begin{equation}\label{eqn: 1}
			c_{1,t}^n\BTt_{\phi_{i_1}D_{f_1,g_1,i_1}}\BTt_{\phi_{i_2}D_{f_2,g_2,i_2}}\ldots\BTt_{\phi_{i_n}D_{f_n,g_n,i_n}}\sim_a c_{1,t}^n\BTt_{\phi_{i_1}\phi_{i_2}\ldots\phi_{i_n}D_{f_1,g_1,i_1}D_{f_2,g_2,i_2}\ldots D_{f_n,g_n,i_n}}.
		\end{equation}
		If all $i_k=1$, in the steps right above \eqref{eqn: 1}, the difference operator may not belong to the trace class. We need to take the anti-symmetrization into consideration. Similar as in the proof of \eqref{eqn: Qpartial antisym =0}, the odd anti-symmetrization is over $n$ symbols, but $Q_z\partial f$ has only $n-1$ entries under the basis $\mathbf{e}_z$. Thus we have the identities
		\begin{equation}\label{eqn: 5}
			\sum_{\tau\in S_n}\sgn(\tau)D_{f_{\tau_1},g_1,1}D_{f_{\tau_2},g_2,1}\ldots D_{f_{\tau_n},g_n,1}=0,
		\end{equation}
		and 
		\begin{flalign}\label{eqn: 2}
			&\sum_{\tau\in S_n}\sgn(\tau)c_{1,t}^n\BTt_{\phi_1D_{f_{\tau_1},g_1,1}}\BTt_{\phi_1D_{f_{\tau_2},g_2,1}}\ldots\BTt_{\phi_1D_{f_{\tau_n},g_n,1}}\nonumber\\
			\sim_a&c_{1,t}^n\bigg[\BTt_{\phi_1}\bigg]^n\BTt_{\sum_{\tau\in S_n}\sgn(\tau)D_{f_{\tau_1},g_1,1}D_{f_{\tau_2},g_2,1}\ldots D_{f_{\tau_n},g_n,1}}\\
			=&0.\nonumber
		\end{flalign}
		Also by \eqref{eqn: 5}, we have the following equation
		\begin{equation}\label{eqn: 3}
			\sum_{\tau\in S_n}\sgn(\tau)c_{1,t}^n\BTt_{\phi_1D_{f_{\tau_1},g_1,1}\phi_1D_{f_{\tau_2},g_2,1}\ldots\phi_1 D_{f_{\tau_n},g_n,1}}=\BTt_{\phi_1^n\sum_{\tau\in S_n}\sgn(\tau)D_{f_{\tau_1},g_1,1}D_{f_{\tau_2},g_2,1}\ldots D_{f_{\tau_n},g_n,1}}=0.
		\end{equation}
		Altogether, \eqref{eqn: 4}, \eqref{eqn: 1}, \eqref{eqn: 2} and \eqref{eqn: 3} imply \eqref{eqn: TC1 into one}.
		
		~

		\noindent{\bf Proof of \eqref{eqn: Tr TC1 sum limit}.}
		
		Denote
		\[
		F(z)=\sum_{\tau\in S_n}\mathrm{sgn}(\tau)C_1(f_{\tau_1},g_1)\ldots C_1(f_{\tau_n},g_n).
		\]
		Under the basis $\mathbf{e}_z$, we compute the following volume form
		\begin{flalign*}
			&(-1)^nn^n(1-|z|^2)^{-n-1}F\intd z_1\wedge\ldots\wedge\intd z_n\wedge\intd \bz_1\wedge\ldots\wedge\intd\bz_n\\
			=&\sum_{\tau,\varsigma\in S_n}\mathrm{sgn}(\tau)\partial_{\varsigma(1)}f_{\tau_1}\bpartial_{\varsigma(1)}g_1\ldots\partial_{\varsigma(n)}f_{\tau_n}\bpartial_{\varsigma(n)}g_n\intd z_1\wedge\ldots\wedge\intd z_n\wedge\intd \bz_1\wedge\ldots\wedge\intd\bz_n\\
			=&\sum_{\tau,\varsigma\in S_n}\sgn(\tau)\big(\partial_{\varsigma(1)}f_{\tau_1}\intd z_{\varsigma(1)}\big)\wedge\ldots\wedge\big(\partial_{\varsigma(n)}f_{\tau_n}\intd z_{\varsigma(n)}\big)\wedge\big(\bpartial_{\varsigma(1)}g_1\intd\bz_{\varsigma(1)}\big)\wedge\ldots\wedge\big(\bpartial_{\varsigma(n)}g_n\intd\bz_{\varsigma(n)}\big)\\
			=&\sum_{\tau,\varsigma\in S_n}\big(\partial_{\varsigma\tau^{-1}(1)}f_1\intd z_{\varsigma\tau^{-1}(1)}\big)\wedge\ldots\wedge\big(\partial_{\varsigma\tau^{-1}(n)}f_n\intd z_{\varsigma\tau^{-1}(n)}\big)\wedge\big(\bpartial_{\varsigma(1)}g_1\intd\bz_{\varsigma(1)}\big)\wedge\ldots\wedge\big(\bpartial_{\varsigma(n)}g_n\intd\bz_{\varsigma(n)}\big)\\
			=&\sum_{\iota,\varsigma\in S_n}\big(\partial_{\iota(1)}f_1\intd z_{\iota(1)}\big)\wedge\ldots\wedge\big(\partial_{\iota(n)}f_n\intd z_{\iota(n)}\big)\wedge\big(\bpartial_{\varsigma(1)}g_1\intd\bz_{\varsigma(1)}\big)\wedge\ldots\wedge\big(\bpartial_{\varsigma(n)}g_n\intd\bz_{\varsigma(n)}\big)\\
			=&\partial f_1\wedge\ldots\wedge\partial f_n\wedge\bpartial g_1\wedge\ldots\wedge\bpartial g_n.
		\end{flalign*}
		Therefore, we obtain the following formula of the trace
		\begin{flalign*}
			&\Tr \BTt_{\sum_{\tau\in S_n}\sgn(\tau)c_{1,t}^nC_1(f_{\tau_1},g_1)C_1(f_{\tau_2},g_2)C_1(f_{\tau_n},g_n)}\\
			=&c_{1,t}^n\Tr\BTt_F\\
			=&\frac{(n-1)!c_{1,t}^n}{\pi^nB(n,t+1)}\int_{\bn}\frac{F(z)}{(1-|z|^2)^{n+1}}\intd m(z)\\
			=&\frac{(n-1)!c_{1,t}^n}{\pi^nB(n,t+1)}\frac{(-1)^n}{n^n}\frac{1}{(-2i)^n}\int_{\bn}\partial f_1\wedge\bpartial g_1\wedge\ldots\wedge\partial f_n\wedge\bpartial g_n.
		\end{flalign*}
		Since
		\[
		c_{1,t}=nt^{-1}+O(t^{-2}),\quad B(n,t+1)=(n-1)!t^{-n}+O(t^{-n-1}),
		\]
		we have
		\[
		\frac{c_{1,t}^n}{B(n,t+1)}\to\frac{n^n}{(n-1)!},\quad t\to\infty.
		\]
		Simplifying the above gives \eqref{eqn: Tr TC1 sum limit}.
		
		In summary, we have proved \eqref{eqn: sigma sim_a TC1}, \eqref{eqn: TC1 into one} and \eqref{eqn: Tr TC1 sum limit}. Altogether they finish Part 2. we have completed the proof of Theorem \ref{thm: AS partial trace dimension n} for $t>-1$.

	\subsection{Proof of Theorem \ref{thm: AS partial trace dimension n}: $t=-1$}
To simplify notations, we write
	\[
	\sigma(f,g)=\sigma_{-1}(f,g),
	\]
	and
	\[
	[f_1,g_1,\ldots,f_n,g_n]^\odd=[f_1,g_1,\ldots,f_n,g_n]^\odd_{-1},
	\]
	\[
	[f_1,g_1,\ldots,f_n,g_n]^\even=[f_1,g_1,\ldots,f_n,g_n]^\even_{-1}.
	\]
	\begin{lem}\label{lem: Tg Hardy}
		Suppose $g\in\mathscr{C}^1(\overline{\bn})$. Then for $h\in H^2(\sn)$,
		\begin{equation}\label{eqn: Tg Hardy}
			T_gh(z)=g(z)h(z)-\frac{1}{n}\int_{\bn}|\varphi_z(w)|^{-2n}\frac{\la\bpartial g(w),z-w\ra}{(1-\la w,z\ra)}h(w)K_w(z)\intd\lambda_0(w),\quad\forall z\in\bn.
		\end{equation}
		Moreover, for $z\in\sn$ almost everywhere,
		\begin{equation}\label{eqn: Tg Hardy boundary}
			T_gh(z)=g(z)h(z)-\frac{1}{n}\int_{\bn}\frac{\la\bpartial g(w),z-w\ra}{1-\la w,z\ra}h(w)K_w(z)\intd\lambda_0(w).
		\end{equation}
	\end{lem}
	
	\begin{proof}
		Suppose $h\in\Hol(\overline{\bn})$. For $z\in\bn$, apply Lemma \ref{lem: BM hardy} with $\alpha=\beta=0$, $v(w)=g(w)h(w)$. Then we compute $T_g$ as follows,
		\begin{flalign*}
			T_gh(z)=&\int_{\sn}g(w)h(w)K_w(z)\frac{\intd\sigma(w)}{\sigma_{2n-1}}\\
			=&d_{0,0}(z)g(z)h(z)-\frac{1}{n}\int_{\bn}|\varphi_z(w)|^{-2n}\frac{\sum_{j=1}^n\bpartial_j\big[g(w)h(w)\big]\overline{(z_j-w_j)}}{1-\la w,z\ra}K_w(z)\intd\lambda_0(w)\\
			=&g(z)h(z)-\frac{1}{n}\int_{\bn}|\varphi_z(w)|^{-2n}\frac{\la\bpartial g(w),z-w\ra}{1-\la w,z\ra}h(w)K_w(z)\intd\lambda_0(w).
		\end{flalign*}
		This proves \eqref{eqn: Tg Hardy} for $h\in\Hol(\overline{\bn})$. The equation for general $h\in H^2(\sn)$ follows from approximation.
		
		For $h\in \Hol(\overline{\bn})$, $z\in\sn$ and $\frac{1}{2}<r<1$, we compute $T_gh(rz)$ 
		\[
		T_gh(rz)=g(rz)h(rz)-\frac{1}{n}\int_{\bn}|\varphi_{rz}(w)|^{-2n}\frac{\la\bpartial g(w),rz-w\ra}{1-\la w,rz\ra}h(w)K_w(rz)\intd\lambda_0(w).
		\]
		The first term on the right hand side tends to $g(z)h(z)$. In order to prove \eqref{eqn: Tg Hardy boundary}, it suffices to prove the following identity
		\begin{flalign}\label{eqn: temp 7}
			&\lim_{r\to1^-}\int_{\bn}|\varphi_{rz}(w)|^{-2n}\frac{\la\bpartial g(w),rz-w\ra}{1-\la w,rz\ra}h(w)K_w(rz)\intd\lambda_0(w)\\
			=&\int_{\bn}\frac{\la\bpartial g(w),z-w\ra}{1-\la w,z\ra}h(w)K_w(z)\intd\lambda_0(w).\nonumber
		\end{flalign}
		By Lemma \ref{lem: Mobius basics}, we have the following estimate
		\[
		\bigg|\frac{\la\bpartial g(w),rz-w\ra}{1-\la w,rz\ra}h(w)K_w(rz)\bigg|\lesssim\frac{|\varphi_{rz}(w)|}{|1-\la rz,w\ra|^{n+1/2}}.
		\]
		Therefore we have the following estimate of the integrals
		\begin{flalign*}
			&\bigg|\int_{\bn}|\varphi_{rz}(w)|^{-2n}\frac{\la\bpartial g(w),rz-w\ra}{1-\la w,rz\ra}h(w)K_w(rz)\intd\lambda_0(w)-\int_{\bn}\frac{\la\bpartial g(w),rz-w\ra}{1-\la w,rz\ra}h(w)K_w(rz)\intd\lambda_0(w)\bigg|\\
			\lesssim&\int_{\bn}\big(|\varphi_{rz}(w)|^{-2n}-1\big)\cdot\frac{|\varphi_{rz}(w)|}{|1-\la rz,w\ra|^{n+1/2}}\intd\lambda_0(w)\\
			\approx&\int_{\bn}|\varphi_{rz}(w)|^{-2n+1}\big(1-|\varphi_{rz}(w)|^2\big)\frac{1}{|1-\la rz,w\ra|^{n+1/2}}\intd\lambda_0(w)\\
			\lesssim&(1-|rz|^2)^{1/2}\to0,\quad r\to1^-.
		\end{flalign*}
		Here the last inequality follows from Lemma \ref{lem: Rudin Forelli generalizations} (3). Therefore, we have the following limit computation,
		\begin{flalign}\label{eqn: temp 8}
			&\lim_{r\to1^-}\int_{\bn}|\varphi_{rz}(w)|^{-2n}\frac{\la\bpartial g(w),rz-w\ra}{1-\la w,rz\ra}h(w)K_w(rz)\intd\lambda_0(w)\\
			=&\lim_{r\to1^-}\int_{\bn}\frac{\la\bpartial g(w),rz-w\ra}{1-\la w,rz\ra}h(w)K_w(rz)\intd\lambda_0(w).\nonumber
		\end{flalign}
		For $\frac{1}{2}<r<1$ and $|a|<1$, it is easy to verify that
		\[
		\frac{1}{|1-ra|}\leq\frac{2}{|1-a|}.
		\]
		So we have the following bounds
		\[
		\bigg|\frac{\la\bpartial g(w),rz-w\ra}{1-\la w,rz\ra}h(w)K_w(rz)\bigg|\lesssim\frac{|\varphi_{rz}(w)|}{|1-\la rz,w\ra|^{n+1/2}}\lesssim\frac{1}{|1-\la z,w\ra|^{n+1/2}}.
		\]
		Thus by the Dominated Convergence Theorem, we arrive at the following equation
		\begin{equation}\label{eqn: temp 10}
			\lim_{r\to1^-}\int_{\bn}\frac{\la\bpartial g(w),rz-w\ra}{1-\la w,rz\ra}h(w)K_w(rz)\intd\lambda_0(w)=\int_{\bn}\frac{\la\bpartial g(w),z-w\ra}{1-\la w,z\ra}h(w)K_w(z)\intd\lambda_0(w).
		\end{equation}
		Combining \eqref{eqn: temp 8} and \eqref{eqn: temp 10} gives \eqref{eqn: temp 7}. Thus \eqref{eqn: Tg Hardy boundary} holds pointwise for $h\in\Hol(\overline{\bn})$. The general case follows from approximation. This completes the proof of Lemma \ref{lem: Tg Hardy}.
	\end{proof}
	
	\begin{lem}\label{lem:Rfg1}
		Suppose $f, g\in\mathscr{C}^2(\overline{\bn})$. Then
		\[
		\sigma(f, g)=R_{f,g,1},
		\]
		where the operator $R_{f,g,1}$ is defined as follows.
		\[
		R_{f,g,1}h(\xi)=-\frac{1}{n^2}\int_{\bn}\int_{\bn}|\varphi_z(w)|^{-2n}\frac{\la\partial f(z),\overline{z-w}\ra\la\bpartial g(w),z-w\ra}{(1-\la w,z\ra)^2}K_w(z)K_z(\xi)\intd\lambda_0(w)\intd\lambda_0(z).
		\]
	\end{lem}
	
	\begin{proof}
		Suppose $h\in\Hol(\overline{\bn})$. By Lemma \ref{lem: Tg Hardy}, for $z\in\sn$, we compute $T_gh$
		\[
		T_gh(z)=g(z)h(z)-\frac{1}{n}\int_{\bn}\frac{\la\bpartial g(w),z-w\ra}{1-\la w,z\ra}h(w)K_w(z)\intd\lambda_0(w).
		\]
		Therefore for $\xi\in\bn$, we have the following computation of $T_fT_g$,
		\begin{flalign*}
			&T_fT_gh(\xi)\\
			=&\int_{\sn}\bigg\{g(z)h(z)-\frac{1}{n}\int_{\bn}\frac{\la\bpartial g(w),z-w\ra}{1-\la w,z\ra}h(w)K_w(z)\intd\lambda_0(w)\bigg\}f(z)K_z(\xi)\frac{\intd\sigma(z)}{\sigma_{2n-1}}\\
			=&\int_{\sn}g(z)f(z)h(z)K_z(\xi)\frac{\intd\sigma(z)}{\sigma_{2n-1}}-\frac{1}{n}\int_{\sn}\int_{\bn}\frac{\la\bpartial g(w),z-w\ra}{1-\la w,z\ra}h(w)K_w(z)f(z)K_z(\xi)\intd\lambda_0(w)\frac{\intd\sigma(z)}{\sigma_{2n-1}}\\
			=&T_{fg}h(\xi)-\frac{1}{n}\int_{\sn}\int_{\bn}\frac{\la\bpartial g(w),z-w\ra}{1-\la w,z\ra}h(w)K_w(z)f(z)K_z(\xi)\intd\lambda_0(w)\frac{\intd\sigma(z)}{\sigma_{2n-1}}.
		\end{flalign*}
		Since $f, g, h$ are bounded, using Lemma \ref{lem: Rudin Forelli generalizations} we see that the double integral on the right converges absolutely. By Fubini's Theorem, we have
		\begin{equation*}
			\sigma(f,g)h(\xi)=-\frac{1}{n}\int_{\bn}\bigg\{\int_{\sn}\frac{f(z)\la\bpartial g(w),z-w\ra}{1-\la w,z\ra}K_w(z)K_z(\xi)\frac{\intd\sigma(z)}{\sigma_{2n-1}}\bigg\}h(w)\intd\lambda_0(w).
		\end{equation*}
		For fixed $w, \xi\in\bn$, applying \eqref{eqn: BM hardy d'} with $\alpha=\beta=0$, $v(z)=\frac{f(z)\la\bpartial g(w),z-w\ra}{1-\la w,z\ra}K_z(\xi)$ gives
		\begin{flalign*}
			&\int_{\sn}\frac{f(z)\la\bpartial g(w),z-w\ra}{1-\la w,z\ra}K_w(z)K_z(\xi)\frac{\intd\sigma(z)}{\sigma_{2n-1}}\\
			=&\frac{1}{n}\int_{\bn}|\varphi_z(w)|^{-2n}\frac{\la\partial f(z),\overline{z-w}\ra\la\bpartial g(w),z-w\ra}{(1-\la w,z\ra)^2}K_w(z)K_z(\xi)\intd\lambda_0(z).
		\end{flalign*}
		Therefore we get the following formula for $\sigma(f,g)$
		\begin{flalign*}
			&\sigma(f,g)h(\xi)\\
			=&-\frac{1}{n}\int_{\bn}\bigg\{\frac{1}{n}\int_{\bn}|\varphi_z(w)|^{-2n}\frac{\la\partial f(z),\overline{z-w}\ra\la\bpartial g(w),z-w\ra}{(1-\la w,z\ra)^2}K_w(z)K_z(\xi)\intd\lambda_0(z)\bigg\}h(w)\intd\lambda_0(w)\\
			=&-\frac{1}{n^2}\int_{\bn^2}|\varphi_z(w)|^{-2n}\frac{\la\partial f(z),\overline{z-w}\ra\la\bpartial g(w),z-w\ra}{(1-\la w,z\ra)^2}K_w(z)K_z(\xi)h(w)\intd\lambda_0(w)\intd\lambda_0(z).
		\end{flalign*}
		This completes the proof of Lemma \ref{lem:Rfg1}.
	\end{proof}

	\begin{defn}\label{defn: hardy Rabfg}
		Suppose $f, g\in\mathscr{C}^2(\overline{\bn})$ and $h\in L^2(\lambda_0)$.
		\begin{enumerate}
			\item Define
			\[
			Ph(\xi)=\int_{\bn}h(z)K_z(\xi)\intd\lambda_0(z),
			\]
			and
			\[
			\Gamma_{f,g}h(z)=-\frac{1}{n^2}\int_{\bn}|\varphi_z(w)|^{-2n}\frac{\la\partial f(z),\overline{z-w}\ra\la\bpartial g(w),z-w\ra}{(1-\la w,z\ra)^2}K_w(z)h(w)\intd\lambda_0(w).
			\]
			Then
			\[
			R_{f,g,1}=P\Gamma_{f,g}.
			\]
			\item Define
			\[
			A^0=\la Q_{\bar{z}}\partial f(z),\overline{(z-w)}\ra,\quad A^1=\la P_{\bar{z}}\partial f(z),\overline{(z-w)}\ra,
			\]
			\[
			B^0=\la Q_z\bpartial g(w), z-w\ra,\quad B^1=\la P_z\bpartial g(w), z-w\ra,
			\]
			and
			\[
			C^{a,b}(z,w)=A^aB^b,\quad a, b=0, 1.
			\]
			Then
			\begin{equation}\label{eqn: Cab estimate}
				|C^{a,b}(z,w)|\lesssim|1-\la z,w\ra|^{1+\frac{a+b}{2}},
			\end{equation}
			and
			\[
			\sum_{a,b=0,1}C^{a,b}(z,w)=\la\partial f(z),\overline{z-w}\ra\la\bpartial g(w),z-w\ra.
			\]
			\item Define
			\[
			\Gamma^{a,b}_{f,g}h(z)=-\frac{1}{n^2}\int_{\bn}|\varphi_z(w)|^{-2n}\frac{C^{a,b}(z,w)}{(1-\la w,z\ra)^2}K_w(z)h(w)\intd\lambda_0(w).
			\]
			Then
			\[
			\Gamma_{f,g}=\sum_{a,b=0,1}\Gamma^{a,b}_{f,g}.
			\]
			\item Define for $j=1,\ldots, n$,
			\[
			\Lambda_{g,j}h(z)=-\frac{1}{n^2}\int_{\bn}|\varphi_z(w)|^{-2n}(z_j-w_j)\frac{\la Q_z\bpartial g(w),z-w\ra}{(1-\la w,z\ra)^2}K_w(z)h(w)\intd\lambda_0(w).
			\]
			Then
			\[
			\Gamma^{0,0}_{f,g}=\sum_{j=1}^nM_{[Q_{\bar{z}}\partial f]_j}\Lambda_{g,j}.
			\]
		\end{enumerate}
	\end{defn}
	
	\begin{lem}\label{lem: hardy schatten memberships}
		Suppose $n\geq2$, $f, g\in\mathscr{C}^2(\overline{\bn})$, $u\in\mathscr{C}^1(\overline{\bn})$, and $a, b=0, 1$. Consider $P, \Gamma^{a,b}_{f,g}, \Lambda_{g,j}$ as operators on $L^2(\lambda_0)$. Then the following hold.
		\begin{enumerate}
			\item[(1)] $[P^{(0)}, M_u]\in\mathcal{S}^p$ for any $p>2n$.
			\item[(2)] $P\in\mathcal{S}^p$ for any $p>n$.
			\item[(3)] $[P, M_u]\in\mathcal{S}^p$ for any $p>\frac{2n}{3}$.
			\item[(4)] If $(a,b)\neq(0,0)$, then $\Gamma^{a,b}_{f,g}P^{(0)}\in\mathcal{S}^p$ for $p$ large enough, and $\Gamma^{a,b}_{f,g}P\in\mathcal{S}^p$ for some $p<n$.
			\item[(5)] For each $j$, $\Lambda_{g,j}$ is bounded.
			\item[(6)] For each $j$, $[\Lambda_{g,j}, M_u]P^{(0)}\in\mathcal{S}^p$ for any $p>2n$.
			\item[(7)] For each $j$, $[\Lambda_{g,j}, M_u]P\in\mathcal{S}^p$ for any $p>\frac{2n}{3}$.
		\end{enumerate}
	\end{lem}
	\begin{proof}
		By the integral formula of $P$ and the Bergman projection $P^{(0)}$, it is easy to see that (2) follows from Corollary \ref{cor: integral schatten membership no phi}, and (1) follows from Corollary \ref{cor: commutator schatten membership no phi}. If $n\geq3$ then (3) also follows from Corollary \ref{cor: commutator schatten membership no phi}. At $n=2$, notice that $P$ is self-adjoint and has range in $L_{a,0}^2(\bn)$. Thus we compute the commutator $[P, M_u]$
		\[
		[P, M_u]=P^{(0)}[P, M_u]+(1-P^{(0)})[P,M_u]=\bigg([M_{\bar{u}},P]P^{(0)}\bigg)^*-H^{(0)}_uP.
		\]
		By (2) and Corollary \ref{cor: hankel schatten membership}, the second term on the right belongs to $\mathcal{S}^p$ for any $p>\frac{2n}{3}$. For any $h\in\Hol(\overline{\bn})$, apply \eqref{eqn: formula Bn d''} with $t=0$, $\alpha=\beta=0$, $\phi=1$, and $v(w)=\big(\bar{u}(z)-\bar{u}(w)\big)h(w)(1-\la z,w\ra)$. Then we get the following expression
		\begin{flalign}\label{eqn: temp 11}
			&[M_{\bar{u}},P]h(z)\nonumber\\
			=&\int_{\bn}\big(\bar{u}(z)-\bar{u}(w)\big)h(w)K_w(z)\intd\lambda_0(w)\nonumber\\
			=&-\int_{\bn}\mathcal{G}^{(0)}_n1(|\varphi_z(w)|^2)\frac{(1-|w|^2)\la\bpartial v(w),z-w\ra}{1-\la w,z\ra}K^{(0)}_w(z)\intd\lambda_0(w)\\
			=&\int_{\bn}\mathcal{G}^{(0)}_n1(|\varphi_z(w)|^2)\bigg((1-\la z,w\ra)\la\bpartial\bar{u}(w),z-w\ra+\big(\bar{u}(z)-\bar{u}(w)\big)\la z,z-w\ra\bigg)\nonumber\\
			&~~~~~~~~\cdot\frac{1-|w|^2}{1-\la w,z\ra}h(w)K^{(0)}_w(z)\intd\lambda_0(w).\nonumber
		\end{flalign}
		Take
		\[
		F(z,w)=\frac{2}{n+1}\bigg((1-\la z,w\ra)\la\bpartial\bar{u}(w),z-w\ra+\big(\bar{u}(z)-\bar{u}(w)\big)\la z,z-w\ra\bigg)\frac{1}{1-\la w,z\ra}K^{(0)}_w(z).
		\]
		Substituting $F$ into the integral on the right hand side of Equation (\ref{eqn: temp 11}) gives  the following formula
		\[
		[M_{\bar{u}},P]h(z)=\int_{\bn}\mathcal{G}_n^{(0)}1(|\varphi_z(w)|^2)F(z,w)h(w)\intd\lambda_1(w).
		\]
		By Lemma \ref{lem: formula Bn R 0}, we have the following estimate 
		\[
		|\mathcal{G}^{(0)}_n1(s)|\lesssim s^{-n}.
		\]
		By Lemma \ref{lem: Mobius basics} and the fact that $\bar{u}\in\mathscr{C}^1(\overline{\bn})$, we obtain the following bound
		\[
		\big|F(z,w)\big|\lesssim|\varphi_z(w)|\frac{1}{|1-\la z,w\ra|^{n+1/2}}.
		\]
		For any $\epsilon>0$, split the map as follows.
		\[
		[M_{\bar{u}},P]P^{(0)}: L_{a,0}^2(\bn)\xrightarrow{E_{0,3-\epsilon}}L_{a,3-\epsilon}^2(\bn)\xrightarrow{T}L^2(\lambda_0),
		\]
		where $Th(z)$ is defined as in the last line of Equation \eqref{eqn: temp 11}. Then by the estimates above and Lemma \ref{lem: bounded operator using Schur's test}, $T$ is bounded. Thus by Lemma \ref{lem: Embedding Schatten membership}, $[M_{\bar{u}},P]P^{(0)}$ is in $\mathcal{S}^p$ for any $p>\frac{2n}{3}$. This proves (3).
		
		Let
		\[
		F_{\Gamma^{a,b}}(z,w)=-\frac{1}{n^2}\frac{C^{a,b}(z,w)}{(1-\la w,z\ra)^2}K_w(z),\quad z, w\in\bn,
		\]
		and
		\[
		F_{\Lambda_{g,j}}(z,w)=-\frac{1}{n^2}(z_j-w_j)\frac{\la Q_z\bpartial g(w),z-w\ra}{(1-\la w,z\ra)^2}K_w(z),\quad z, w\in\bn, j=1,\ldots, n.
		\]
		Then write
		\[
		\Gamma_{f,g}^{a,b}h(z)=\int_{\bn}|\varphi_z(w)|^{-2n}F_{\Gamma^{a,b}}(z,w)h(w)\intd\lambda_0(w),
		\]
		and
		\[
		\Lambda_{g,j}h(z)=\int_{\bn}|\varphi_z(w)|^{-2n}F_{\Lambda_{g,j}}(z,w)h(w)\intd\lambda_0(w).
		\]
		By \eqref{eqn: Cab estimate} and Lemma \ref{lem: Mobius basics}, we have the following bound
		\[
		\big|F_{\Gamma^{a,b}}(z,w)\big|\lesssim|\varphi_z(w)|^2\frac{1}{|1-\la z,w\ra|^{n+1-\frac{a+b}{2}}},
		\]
		and
		\[
		\big|F_{\Lambda_{g,j}}(z,w)\big|\lesssim|\varphi_z(w)|^2\frac{1}{|1-\la z,w\ra|^{n+1}}.
		\]
		Thus by Lemma \ref{lem: bounded operator using Schur's test}, $\Lambda_{g,j}$ is bounded. This proves (5). By Lemma \ref{lem: integral operator Schatten membership}, $\Gamma_{f,g}^{a,b}P^{(0)}\in\mathcal{S}^p$ for $p$ large enough. Since $P\in\mathcal{S}^p$ for any $p>n$, and $\Gamma^{a,b}_{f,g}P=\Gamma_{f,g}^{a,b}P^{(0)}P$, we have $\Gamma^{a,b}_{f,g}P\in\mathcal{S}^p$ for some $p<n$. This proves (4). Also, (6) follows from the above estimates and Corollary \ref{cor: commutator schatten membership with phi}, and (7) follows again from (6) and the equation $P=P^{(0)}P$. This completes the proof.
	\end{proof}
	
	\begin{lem}\label{lem: Phat}
		Denote $\hat{P}$ the operator from $L^2(\lambda_0)$ to $H^2(\sn)$ defined by the same integral formula as $P$. Then $\hat{P}=E_{-1,1}^*$.
	\end{lem}
	\begin{proof}
		For any $f\in L^2(\lambda_0)$, $g\in H^2(\sn)$, we compute $\la \hat{P}f, g\ra_{H^2(\sn)}$
		\begin{flalign*}
			\la \hat{P}f, g\ra_{H^2(\sn)}=&\int_{\sn}\int_{\bn}f(z)K_z(\xi)\intd\lambda_0(z)\overline{g(\xi)}\frac{\intd\sigma(\xi)}{\sigma_{2n-1}}\\
			=&\int_{\bn}\int_{\sn}\overline{g(\xi)}K_z(\xi)\frac{\intd\sigma(\xi)}{\sigma_{2n-1}}f(z)\intd\lambda_0(z)\\
			=&\int_{\bn}\overline{g(z)}f(z)\intd\lambda_0(z)\\
			=&\la f, E_{-1,1}g\ra_{L^2(\lambda_0)}.
		\end{flalign*}
		This completes the proof.
	\end{proof}
	
	\begin{proof}[{\bf Proof of Theorem \ref{thm: AS partial trace dimension n} ($t=-1$)}]
		For any $\tau\in S_n$, split the map
		\[
		\sigma(f_{\tau_1},g_1)\ldots\sigma(f_{\tau_n},g_n)=P\Gamma_{f_{\tau_1},g_1}P\Gamma_{f_{\tau_2},g_2}\ldots P\Gamma_{f_{\tau_n},g_n}
		\]
		as follows.
		\begin{equation}\label{eqn: temp 4}
			H^2(\sn)\xrightarrow{E_{-1,1}}L_{a,0}^2(\bn)\xrightarrow{\Gamma_{f_{\tau_1},g_1}P\Gamma_{f_{\tau_2},g_2}\ldots P\Gamma_{f_{\tau_n},g_n}P^{(0)}}L^2(\lambda_0)\xrightarrow{\hat{P}}H^2(\sn).
		\end{equation}
		By Lemma \ref{lem: Embedding Schatten membership} and Lemma \ref{lem: Phat}, the operators on the two ends of \eqref{eqn: temp 4} are in $\mathcal{S}^p$ for any $p>2n$. Thus it suffices to show that
		\[
		\sum_{\tau\in S_n}\sgn(\tau)\Gamma_{f_{\tau_1},g_1}P\Gamma_{f_{\tau_2},g_2}\ldots P\Gamma_{f_{\tau_n},g_n}P^{(0)}
		\]
		defines an operator in $\mathcal{S}^p(L^2(\lambda_0))$ for some $p<\frac{n}{n-1}$.
		\begin{notation}
		For operators $A, B$, temporarily write $A\sim _pB$ when $A-B$ is in $\mathcal{S}^p(L^2(\lambda_0))$ for some $p<\frac{n}{n-1}$. 
		\end{notation}
		Denote
		\[
		u_{i,j}(z)=[Q_{\bar{z}}\partial f_i(z)]_j.
		\]
		Then by Lemma \ref{lem: hardy schatten memberships}, we compute an element in the above sum
		\begin{flalign*}
			&\Gamma_{f_{\tau_1},g_1}P\Gamma_{f_{\tau_2},g_2}\ldots P\Gamma_{f_{\tau_n},g_n}P^{(0)}\\
			=&\bigg(\sum_{a,b=0,1}\Gamma^{a,b}_{f_{\tau_1},g_1}\bigg)P\bigg(\sum_{a,b=0,1}\Gamma^{a,b}_{f_{\tau_2},g_2}\bigg)\ldots P\bigg(\sum_{a,b=0,1}\Gamma^{a,b}_{f_{\tau_n},g_n}\bigg)P^{(0)}\\
			\sim_p&\Gamma^{0,0}_{f_{\tau_1},g_1}P\Gamma^{0,0}_{f_{\tau_2},g_2}\ldots P\Gamma^{0,0}_{f_{\tau_n},g_n}P^{(0)}\\
			=&\sum_{j_1,\ldots,j_n=1}^nM_{u_{\tau_1,j_1}}\Lambda_{g_1,j_1}PM_{u_{\tau_2,j_2}}\Lambda_{g_2,j_2}P\ldots M_{u_{\tau_{n},j_n}}\Lambda_{g_n,j_n}P^{(0)}\\
			=&\sum_{j_1,\ldots,j_n=1}^n\Lambda_{g_1,j_1}M_{u_{\tau_1,j_1}}PM_{u_{\tau_2,j_2}}\Lambda_{g_2,j_2}P\ldots M_{u_{\tau_{n},j_n}}\Lambda_{g_n,j_n}P^{(0)}\\
			&+\sum_{j_1,\ldots,j_n=1}^n[M_{u_{\tau_1,j_1}}, \Lambda_{g_1,j_1}]PM_{u_{\tau_2,j_2}}\Lambda_{g_2,j_2}P\ldots M_{u_{\tau_{n},j_n}}\Lambda_{g_n,j_n}P^{(0)}.
		\end{flalign*}
		Each operator in the second term contains $n-2$ copies of $P$ and one $[M_{u_{\tau_1,j_1}}, \Lambda_{g_1,j_1}]P$. By Lemma \ref{lem: hardy schatten memberships} (2)  and (7), it belongs to $\mathcal{S}^p$ for some $p<\frac{n}{n-1}$. Thus we compute the following sum,
		\begin{flalign*}
			&\sum_{j_1,\ldots,j_n=1}^nM_{u_{\tau_1,j_1}}\Lambda_{g_1,j_1}PM_{u_{\tau_2,j_2}}\Lambda_{g_2,j_2}P\ldots M_{u_{\tau_{n},j_n}}\Lambda_{g_n,j_n}P^{(0)}\\
			\sim_p&\sum_{j_1,\ldots,j_n=1}^n\Lambda_{g_1,j_1}M_{u_{\tau_1,j_1}}PM_{u_{\tau_2,j_2}}\Lambda_{g_2,j_2}P\ldots M_{u_{\tau_{n},j_n}}\Lambda_{g_n,j_n}P^{(0)}\\
			=&\sum_{j_1,\ldots,j_n=1}^n\Lambda_{g_1,j_1}PM_{u_{\tau_1,j_1}}M_{u_{\tau_2,j_2}}\Lambda_{g_2,j_2}P\ldots M_{u_{\tau_{n},j_n}}\Lambda_{g_n,j_n}P^{(0)}\\
			&+\sum_{j_1,\ldots,j_n=1}^n\Lambda_{g_1,j_1}[M_{u_{\tau_1,j_1}},P]M_{u_{\tau_2,j_2}}\Lambda_{g_2,j_2}P\ldots M_{u_{\tau_{n},j_n}}\Lambda_{g_n,j_n}P^{(0)}.
		\end{flalign*}
		By Lemma \ref{lem: hardy schatten memberships} (2) and (3), the last term belongs to $\mathcal{S}^p$ for some $p<\frac{n}{n-1}$. Therefore we have the following equation
		\begin{flalign*}
			&\sum_{j_1,\ldots,j_n=1}^nM_{u_{\tau_1,j_1}}\Lambda_{g_1,j_1}PM_{u_{\tau_2,j_2}}\Lambda_{g_2,j_2}P\ldots M_{u_{\tau_{n},j_n}}\Lambda_{g_n,j_n}P^{(0)}\\
			\sim_p&\sum_{j_1,\ldots,j_n=1}^n\Lambda_{g_1,j_1}PM_{u_{\tau_1,j_1}}M_{u_{\tau_2,j_2}}\Lambda_{g_2,j_2}P\ldots M_{u_{\tau_{n},j_n}}\Lambda_{g_n,j_n}P^{(0)}.
		\end{flalign*}
		Continuing like this, we obtain the following equation
		\begin{flalign*}
			&\sum_{j_1,\ldots,j_n=1}^nM_{u_{\tau_1,j_1}}\Lambda_{g_1,j_1}PM_{u_{\tau_2,j_2}}\Lambda_{g_2,j_2}P\ldots M_{u_{\tau_{n},j_n}}\Lambda_{g_n,j_n}P^{(0)}\\
			\sim_p&\sum_{j_1,\ldots,j_n=1}^n\Lambda_{g_1,j_1}P\Lambda_{g_2,j_2}P\ldots \Lambda_{g_n,j_n}M_{u_{\tau_1,j_1}}M_{u_{\tau_2,j_2}}\ldots M_{u_{\tau_{n},j_n}}P^{(0)}\\
			=:&\Theta_{f_{\tau_1},g_1,\ldots,f_{\tau_n},g_n}.
		\end{flalign*}
		As in the proof of Theorem \ref{thm: AS partial trace dimension n}, by writing $\Theta_{f_{\tau_1},g_1,\ldots,f_{\tau_n},g_n}$ as an integral operator, and by \eqref{eqn: Qpartial antisym =0}, we can show that
		\[
		\sum_{\tau\in S_n}\sgn(\tau)\Theta_{f_{\tau_1},g_1,\ldots,f_{\tau_n},g_n}=0.
		\]
		Therefore we conclude
		\[
		\sum_{\tau\in S_n}\sgn(\tau)\Gamma_{f_{\tau_1},g_1}P\Gamma_{f_{\tau_2},g_2}\ldots P\Gamma_{f_{\tau_n},g_n}P^{(0)}\in\mathcal{S}^p
		\]
		for some $p<\frac{n}{n-1}$. Thus by \eqref{eqn: temp 4}, we obtain
		\[
		[f_1,g_1,\ldots,f_n,g_n]^{\odd}=P\bigg(\sum_{\tau\in S_n}\sgn(\tau)\Gamma_{f_{\tau_1},g_1}P\Gamma_{f_{\tau_2},g_2}\ldots P\Gamma_{f_{\tau_n},g_n}P^{(0)}\bigg)E_{-1,1}
		\]
		is in the trace class. Since
		\[
		[f_1,g_1,\ldots,f_n,g_n]^{\odd}=\bigg([\bar{g_n},\bar{f_n},\ldots,\bar{g_1},\bar{f_1}]^{\even }\bigg)^*,
		\]
		the even anti-symmetric sum is also in the trace class. Finally, as in the proof of Corollary \ref{cor: AS sum trace membership and trace limit}, we obtain
		\[
		[T_{f_1}, T_{f_2},\ldots,T_{f_{2n}}]=\frac{1}{n!}\sum_{\tau\in S_{2n}}\sgn(\tau)[f_{\tau_1}, f_{\tau_2},\ldots,f_{\tau_{2n}}]^{\odd}\in\mathcal{S}^1.
		\]
		This completes the proof of Theorem \ref{thm: AS partial trace dimension n} for $t=-1$. 
	\end{proof}

	\section{Main Theorems}\label{sec: main}
	In this section, we prove the main results of this article. 
	\subsection{Helton-Howe cocycle}\label{sec:Helton-Howe}
	We are ready to prove the following main theorem.
	\begin{thm}\label{thm: main}
		Suppose $f_1, f_2, \ldots, f_{2n}\in\mathscr{C}^2(\overline{\bn})$ and $t\geq -1$. Then the following hold.
		\begin{enumerate}
			\item The antisymmetric sum $[\BTt_{f_1}, \BTt_{f_2}, \ldots, \BTt_{f_{2n}}]$ is in the trace class $\mathcal{S}^1$.
			\item
			\begin{equation}\label{eq:trace integral}
				\Tr[\BTt_{f_1}, \BTt_{f_2}, \ldots, \BTt_{f_{2n}}]=\frac{n!}{(2\pi i)^n}\int_{\bn}\intd f_1\wedge\intd f_2\wedge\ldots\wedge\intd f_{2n},
			\end{equation}
			which is independent of $t$.
		\end{enumerate}
	\end{thm}
	\begin{proof}
		By Lemma \ref{lem: AS sum trace t-tp hardy}, the operator on $H^2(\sn)$,
		\[
		[T_{f_1},T_{f_2},\ldots,T_{f_{2n}}]-[T^{(1,-1)}_{f_1},T^{(1,-1)}_{f_2},\ldots,T^{(1,-1)}_{f_{2n}}]
		\]
		is a trace class operator of zero trace. By Corollary \ref{cor: AS sum trace membership and trace limit}, the operator
		\[
		[T_{f_1},T_{f_2},\ldots,T_{f_{2n}}]
		\]
		is itself in the trace class. Thus we have the following of traces
		\begin{flalign*}
			\Tr[T_{f_1},T_{f_2},\ldots,T_{f_{2n}}]
			=\Tr[T^{(1,-1)}_{f_1},T^{(1,-1)}_{f_2},\ldots,T^{(1,-1)}_{f_{2n}}].
		\end{flalign*}
		On the other hand, by Corollary \ref{cor: AS sum trace membership and trace limit}, the operator
		\[
		[T^{(1)}_{f_1},\ldots,T^{(1)}_{f_{2n}}]
		\]
		is a trace class operator on $L_{a,1}^2(\bn)$. Since
		\[
		[T^{(1)}_{f_1},\ldots,T^{(1)}_{f_{2n}}]\big|_{H^2(\sn)}=[T^{(1,-1)}_{f_1},T^{(1,-1)}_{f_2},\ldots,T^{(1,-1)}_{f_{2n}}],
		\]
		by Lemma \ref{lem: trace on two spaces}, we get the following equation of traces
		\[
		\Tr[T^{(1,-1)}_{f_1},T^{(1,-1)}_{f_2},\ldots,T^{(1,-1)}_{f_{2n}}]=\Tr[T^{(1)}_{f_1},\ldots,T^{(1)}_{f_{2n}}].
		\]
		Therefore we arrive at the following equation
		\begin{equation*}
			\Tr[T_{f_1},T_{f_2},\ldots,T_{f_{2n}}]=\Tr[T^{(1)}_{f_1},\ldots,T^{(1)}_{f_{2n}}].
		\end{equation*}
		Thus the case of the Hardy space reduces to that of the weighted Bergman space.
		
		Suppose $t>-1$, by Lemma \ref{lem: AS sum trace t-tp}, the operator on $\bert$
		\[
		[\BTt_{f_1},\BTt_{f_2},\ldots,\BTt_{f_{2n}}]-[\BTtpt_{f_1},\BTtpt_{f_2},\ldots,\BTtpt_{f_{2n}}]
		\]
		is a trace class operator of zero trace. By Corollary \ref{cor: AS sum trace membership and trace limit}, we know  that the antisymmetrization $[\BTt_{f_1},\BTt_{f_2},\ldots,\BTt_{f_{2n}}]$
		is itself in the trace class. Thus so does $[\BTtpt_{f_1},\BTtpt_{f_2},\ldots,\BTtpt_{f_{2n}}]$. Also by Corollary \ref{cor: AS sum trace membership and trace limit},
		\[
		[\BTtp_{f_1},\BTtp_{f_2},\ldots,\BTtp_{f_{2n}}]
		\]
		is a trace class operator on $L_{a,t+1}^2(\bn)$. Clearly
		\[
		[\BTtp_{f_1},\BTtp_{f_2},\ldots,\BTtp_{f_{2n}}]\big|_{\bert}= [\BTtpt_{f_1},\BTtpt_{f_2},\ldots,\BTtpt_{f_{2n}}].
		\]
		Thus by Lemma \ref{lem: trace on two spaces}, we have the following equation
		\[
		\Tr[\BTtp_{f_1},\BTtp_{f_2},\ldots,\BTtp_{f_{2n}}]=\Tr[\BTtpt_{f_1},\BTtpt_{f_2},\ldots,\BTtpt_{f_{2n}}].
		\]
		Therefore we conclude with the following identity
		\begin{flalign*}
			&\Tr[\BTt_{f_1},\BTt_{f_2},\ldots,\BTt_{f_{2n}}]\\
			=&\Tr[\BTtpt_{f_1},\BTtpt_{f_2},\ldots,\BTtpt_{f_{2n}}]\\
			=&\Tr[\BTtp_{f_1},\BTtp_{f_2},\ldots,\BTtp_{f_{2n}}].
		\end{flalign*}
		This holds for any $t>-1$. Thus by Corollary \ref{cor: AS sum trace membership and trace limit}, we have shown
		\[
		\Tr[\BTt_{f_1},\BTt_{f_2},\ldots,\BTt_{f_{2n}}]=\lim_{k\to\infty}\Tr[T^{(t+k)}_{f_1}, T^{(t+k)}_{f_2},\ldots, T^{(t+k)}_{f_{2n}}]=\frac{n!}{(2\pi i)^n}\int_{\bn}\intd f_1\wedge\intd f_2\wedge\ldots\wedge\intd f_{2n}.
		\]
		This completes the proof of Theorem \ref{thm: main}.
	\end{proof}
	
	\subsection{The Connes-Chern Character}\label{subsec: connes chern}
As mentioned in Remark \ref{rem: even odd AS}, in this subsection, we consider the Connes-Chern character at $p>n$.
\begin{prop}\label{prop: many sigma_t}
Suppose $p\geq n+1$ is an integer and $f_1, g_1, \ldots, f_p, g_p\in\mathscr{C}^2(\overline{\bn})$. Then for any $t>-1$, the product $\sigma_t(f_1,g_1)\sigma_t(f_2,g_2)\ldots\sigma_t(f_p,g_p)$ is in the trace class, and
\begin{flalign*}
&\lim_{t\to\infty}t^{p-n}\Tr\bigg(\sigma_t(f_1,g_1)\sigma_t(f_2,g_2)\ldots\sigma_t(f_p,g_p)\bigg)\\
=&\frac{n^p}{\pi^n}\int_{\bn}\prod_{j=1}^pC_1(f_j,g_j)(z)\frac{\intd m(z)}{(1-|z|^2)^{n+1}}\\
=&\frac{(-1)^p}{\pi^n}\int_{\bn}\prod_{j=1}^p\bigg[\sum_{i=1}^n\partial_if_j(z)\bpartial_ig_j(z)-Rf_j(z)\bar{R}g_j(z)\bigg](1-|z|^2)^{p-n-1}\intd m(z).
\end{flalign*}
\end{prop}

\begin{proof}
The proof is similar to Part 2 of Section \ref{subsec: pf of thm t>-1}. Recall that the following facts were used in Part 2 and follow from Theorem \ref{thm: quantization bergman} and its corollaries. For $f, g\in\mathscr{C}^2(\overline{\bn})$, we decompose $\sigma_t(f,g)$ as follows
\begin{equation}\label{eqn: eqn 0}
	\sigma_t(f,g)=R^{(t)}_{f,g,1}=c_{1,t}\BTt_{C_1(f,g)}+R^{(t)}_{f,g,2},
\end{equation}
and the following hold.
\begin{enumerate}
	\item[(1)] $c_{1,t}=nt^{-1}+O(t^{-2})$.
	\item[(2)] $C_1(f, g)=\phi_1D_{f,g,1}+\phi_2D_{f,g,2}$,
where
\[
\phi_i(z)=(1-|z|^2)^i, \quad i=1, 2,\ldots,
\]
and $D_{f,g,1}, D_{f,g,2}\in\mathscr{C}^1(\overline{\bn})$ is defined as in \eqref{eqn: Dfg1} and \eqref{eqn: Dfg2}.
	\item[(3)] $R^{(t)}_{f,g,i}\in\mathcal{S}^p,\quad i=1,2$,
and for large $t$, $\|R^{(t)}_{f,g,i}\|_{\mathcal{S}^p}\lesssim t^{-i+\frac{n}{p}},\quad i=1,2$.
\item[(4)] For $t$ large enough, $i, j=1, 2, \ldots$, and $u, v\in\mathscr{C}^1(\overline{\bn})$,
\begin{enumerate}
	\item $\BTt_u\BTt_v-\BTt_{uv}\in\mathcal{S}^p,~\forall p>n,\quad\text{ and }\|\BTt_u\BTt_v-\BTt_{uv}\|_{\mathcal{S}^p}\lesssim_p t^{-1+\frac{n}{p}},$
	\item $\BTt_{\phi_i}\BTt_{u}-\BTt_{\phi_iu}\in\mathcal{S}^p,~\forall p>\frac{n}{i+\frac{1}{2}}, \quad\text{ and }\|\BTt_{\phi_i}\BTt_{u}-\BTt_{\phi_iu}\|_{\mathcal{S}^p}\lesssim_p t^{-1+\frac{n}{p}},$
	\item $\BTt_{u}\BTt_{\phi_i}-\BTt_{\phi_iu}\in\mathcal{S}^p,~\forall p>\frac{n}{i+\frac{1}{2}}, \quad\text{ and }\|\BTt_{u}\BTt_{\phi_i}-\BTt_{\phi_iu}\|_{\mathcal{S}^p}\lesssim_p t^{-1+\frac{n}{p}},$
	\item $\BTt_{\phi_i}\BTt_{\phi_j}-\BTt_{\phi_i\phi_j}\in\mathcal{S}^p,~\forall p>\frac{n}{i+j},\quad\text{ and }\|\BTt_{\phi_i}\BTt_{\phi_j}-\BTt_{\phi_i\phi_j}\|_{\mathcal{S}^p}\lesssim_p t^{-1+\frac{n}{p}},$
	\item $\BTt_{\phi_iu}\in\mathcal{S}^p,~\forall p>\frac{n}{i},\quad\text{ and }\|\BTt_{\phi_iu}\|_{\mathcal{S}^p}\lesssim_p t^{\frac{n}{p}}.$
\end{enumerate}
\end{enumerate}
Iterating Lemma \ref{lem: holder inequality for operators} as in Remark \ref{rem: X1...Xn trace},  we have that the property
\[
\sigma_t(f_1,g_1)\sigma_t(f_2,g_2)\ldots\sigma_t(f_p,g_p)\in\mathcal{S}^1.
\]

 As in the proof of Theorem \ref{thm: AS partial trace dimension n}, we write $S\sim_aT$ when $S-T$ is in trace class with trace norm converging to $0$. Then by \eqref{eqn: eqn 0}, we compute
\begin{flalign*}
&\sigma_t(f_1,g_1)\sigma_t(f_2,g_2)\ldots\sigma_t(f_p,g_p)-c_{1,t}^p\BTt_{C_1(f_1,g_1)}\BTt_{C_1(f_2,g_2)}\ldots\BTt_{C_1(f_p,g_p)}\\
=&R^{(t)}_{f_1,g_1,1}R^{(t)}_{f_2,g_2,1}\ldots R^{(t)}_{f_p,g_p,1}-\big(R^{(t)}_{f_1,g_1,1}-R^{(t)}_{f_1,g_1,2}\big)\big(R^{(t)}_{f_2,g_2,1}-R^{(t)}_{f_2,g_2,2}\big)\ldots\big(R^{(t)}_{f_p,g_p,1}-R^{(t)}_{f_p,g_p,2}\big)\\
=&\sum\pm R^{(t)}_{f_1,g_1,i_1}R^{(t)}_{f_2,g_2,i_2}\ldots R^{(t)}_{f_p,g_p,i_p},
\end{flalign*}
where $i_1, i_2,\ldots, i_p\in\{1,2\}$ and at least one $i_k=2$. Again, applying Lemma \ref{lem: holder inequality for operators} inductively as in Remark \ref{rem: X1...Xn trace} gives the following bounds
\[
\|R^{(t)}_{f_1,g_1,i_1}R^{(t)}_{f_2,g_2,i_2}\ldots R^{(t)}_{f_p,g_p,i_p}\|\leq\|R^{(t)}_{f_1,g_1,i_1}\|_{\mathcal{S}^p}\ldots\|R^{(t)}_{f_p,g_p,i_p}\|_{\mathcal{S}^p}\lesssim t^{-i_1+\frac{n}{p}}\cdot\ldots\cdot t^{-i_p+\frac{n}{p}}\leq t^{-p-1+n}.
\]
We reach the following equation
\begin{equation}\label{eqn: eqn 1}
t^{p-n}\sigma_t(f_1,g_1)\sigma_t(f_2,g_2)\ldots\sigma_t(f_p,g_p)\sim_at^{p-n}c_{1,t}^p\BTt_{C_1(f_1,g_1)}\BTt_{C_1(f_2,g_2)}\ldots\BTt_{C_1(f_p,g_p)}.
\end{equation}
Also, by (1) and (4)-(e), we have the following equation
\begin{equation}\label{eqn: eqn 2}
t^{p-n}c_{1,t}^p\BTt_{C_1(f_1,g_1)}\BTt_{C_1(f_2,g_2)}\ldots\BTt_{C_1(f_p,g_p)}\sim_a n^pt^{-n}\BTt_{C_1(f_1,g_1)}\BTt_{C_1(f_2,g_2)}\ldots\BTt_{C_1(f_p,g_p)}.
\end{equation}
Write $u_{ij}=D_{f_i,g_i,j}$, $i=1,\ldots,p$, $j=1, 2$. Then we arrive at the following equation,
\[
t^{-n}\BTt_{C_1(f_1,g_1)}\BTt_{C_1(f_2,g_2)}\ldots\BTt_{C_1(f_p,g_p)}=\sum_{j_1,\ldots,j_p=1,2}t^{-n}\BTt_{\phi_{j_1}u_{1j_1}}\BTt_{\phi_{j_2}u_{2j_2}}\ldots\BTt_{\phi_{j_p}u_{pj_p}}.
\]
By (4), we have the following estimate
\begin{flalign*}
&t^{-n}\BTt_{\phi_{j_1}u_{1j_1}}\BTt_{\phi_{j_2}u_{2j_2}}\ldots\BTt_{\phi_{j_p}u_{pj_p}}\\
=&t^{-n}\big(\BTt_{\phi_{j_1}u_{1j_1}}-\BTt_{\phi_i}\BTt_{u_{1j_1}}\big)\BTt_{\phi_{j_2}u_{2j_2}}\ldots\BTt_{\phi_{j_p}u_{pj_p}}+t^{-n}\BTt_{\phi_{j_1}}\BTt_{u_{1j_1}}\BTt_{\phi_{j_2}u_{2j_2}}\ldots\BTt_{\phi_{j_p}u_{pj_p}}\\
\sim_a&t^{-n}\BTt_{\phi_{j_1}}\BTt_{u_{1j_1}}\BTt_{\phi_{j_2}u_{2j_2}}\ldots\BTt_{\phi_{j_p}u_{pj_p}}\\
&\ldots\\
\sim_a&t^{-n}\BTt_{\phi_{j_1}}\BTt_{u_{1j_1}}\BTt_{\phi_{j_2}}\BTt_{u_{2j_2}}\ldots\BTt_{\phi_{j_p}}\BTt_{u_{pj_p}}\\
=&t^{-n}\BTt_{\phi_{j_1}}[\BTt_{u_{1j_1}},\BTt_{\phi_{j_2}}]\BTt_{u_{2j_2}}\ldots\BTt_{\phi_{j_p}}\BTt_{u_{pj_p}}+t^{-n}\BTt_{\phi_{j_1}}\BTt_{\phi_{j_2}}\BTt_{u_{1j_1}}\BTt_{u_{2j_2}}\ldots\BTt_{\phi_{j_p}}\BTt_{u_{pj_p}}\\
\sim_a&t^{-n}\BTt_{\phi_{j_1}}\BTt_{\phi_{j_2}}\BTt_{u_{1j_1}}\BTt_{u_{2j_2}}\ldots\BTt_{\phi_{j_p}}\BTt_{u_{pj_p}}\\
&\ldots\\
\sim_a&t^{-n}\BTt_{\phi_{j_1}}\BTt_{\phi_{j_2}}\ldots\BTt_{\phi_{j_p}}\BTt_{u_{1j_1}}\BTt_{u_{2j_2}}\ldots\BTt_{u_{pj_p}}\\
=&t^{-n}\big(\BTt_{\phi_{j_1}}\BTt_{\phi_{j_2}}-\BTt_{\phi_{j_1+j_2}}\big)\BTt_{\phi_{j_3}}\ldots\BTt_{\phi_{j_p}}\BTt_{u_{1j_1}}\BTt_{u_{2j_2}}\ldots\BTt_{u_{pj_p}}\\
&\qquad \qquad \qquad \qquad +t^{-n}\BTt_{\phi_{j_1+j_2}}\BTt_{\phi_{j_3}}\ldots\BTt_{\phi_{j_p}}\BTt_{u_{1j_1}}\BTt_{u_{2j_2}}\ldots\BTt_{u_{pj_p}}\\
\sim_a&t^{-n}\BTt_{\phi_{j_1+j_2}}\BTt_{\phi_{j_3}}\ldots\BTt_{\phi_{j_p}}\BTt_{u_{1j_1}}\BTt_{u_{2j_2}}\ldots\BTt_{u_{pj_p}}\\
&\ldots\\
\sim_a&t^{-n}\BTt_{\phi_{j_1+j_2+\ldots+j_p}}\BTt_{u_{1j_1}}\BTt_{u_{2j_2}}\ldots\BTt_{u_{pj_p}}\\
=&t^{-n}\BTt_{\phi_{j_1+j_2+\ldots+j_p}}\big(\BTt_{u_{1j_1}}\BTt_{u_{2j_2}}-\BTt_{u_{1j_1}u_{2j_2}}\big)\BTt_{u_{3j_3}}\ldots\BTt_{u_{pj_p}}\\
&\qquad \qquad \qquad \qquad +t^{-n}\BTt_{\phi_{j_1+j_2+\ldots+j_p}}\BTt_{u_{1j_1}u_{2j_2}}\BTt_{u_{3j_3}}\ldots\BTt_{u_{pj_p}}\\
\sim_a&t^{-n}\BTt_{\phi_{j_1+j_2+\ldots+j_p}}\BTt_{u_{1j_1}u_{2j_2}}\BTt_{u_{3j_3}}\ldots\BTt_{u_{pj_p}}\\
&\ldots\\
\sim_a&t^{-n}\BTt_{\phi_{j_1+j_2+\ldots+j_p}}\BTt_{u_{1j_1}u_{2j_2}\ldots u_{pj_p}}\\
\sim_a&t^{-n}\BTt_{\phi_{j_1+j_2+\ldots+j_p}u_{1j_1}u_{2j_2}\ldots u_{pj_p}}.
\end{flalign*}
Adding up over $j_1,j_2,\ldots,j_p=1,2$, we get the following equation
\begin{equation}\label{eqn: eqn 3}
t^{-n}\BTt_{C_1(f_1,g_1)}\BTt_{C_1(f_2,g_2)}\ldots\BTt_{C_1(f_p,g_p)}\sim_a t^{-n}\BTt_{C_1(f_1,g_1)C_1(f_2,g_2)\ldots C_1(f_p,g_p)}.
\end{equation}
Combining \eqref{eqn: eqn 1} \eqref{eqn: eqn 2} and \eqref{eqn: eqn 3}, we arrive at the following equation
\begin{equation}\label{eqn: eqn 4}
t^{p-n}\sigma_t(f_1,g_1)\sigma_t(f_2,g_2)\ldots\sigma_t(f_p,g_p)\sim_a n^pt^{-n}\BTt_{C_1(f_1,g_1)C_1(f_2,g_2)\ldots C_1(f_p,g_p)}.
\end{equation}
Denote
\[
F=C_1(f_1,g_1)C_1(f_2,g_2)\ldots C_1(f_p,g_p).
\]
Then $|F(z)|\lesssim(1-|z|^2)^p$.
By \cite[Lemma 2.5]{TWZ:semicommutator} and \ref{lem: Rudin Forelli generalizations}, we compute $\Tr\big(n^pt^{-n}\BTt_{F}\big)$
\begin{flalign*}
&\Tr\big(n^pt^{-n}\BTt_{F}\big)\\
=&n^pt^{-n}\int_{\bn}\la\BTt_F\BKt_\xi,\BKt_\xi\ra\intd\lambda_t(\xi)\\
=&n^pt^{-n}\int_{\bn}\int_{\bn}F(z)\BKt_\xi(z)\BKt_z(\xi)\intd\lambda_t(z)\intd\lambda_t(\xi)\\
=&n^pt^{-n}\int_{\bn}\int_{\bn}F(z)\BKt_\xi(z)\BKt_z(\xi)\intd\lambda_t(\xi)\intd\lambda_t(z)\\
=&n^pt^{-n}\int_{\bn}F(z)\BKt_z(z)\intd\lambda_t(z)\\
=&n^pt^{-n}\frac{(n-1)!}{\pi^nB(n,t+1)}\int_{\bn}\frac{F(z)}{(1-|z|^2)^{n+1}}\intd m(z)\\
\to&\frac{n^p}{\pi^n}\int_{\bn}\frac{F(z)}{(1-|z|^2)^{n+1}}\intd m(z),\quad t\to\infty.
\end{flalign*}
This gives the first equation. The second equation follows from plugging in the formula of $C_1(f,g)$ in Remark \ref{rem: C1 formula for any basis}.
This completes the proof of Proposition \ref{prop: many sigma_t}.
\end{proof}

Recall that $C_1(f,g)-C_1(g,f)=\frac{-i}{n}\{f,g\}$. Thus Proposition \ref{prop: many sigma_t} implies the following.

\begin{cor}\label{cor: many commutators trace}
Suppose $p\geq n+1$ is an integer and $f_1,g_1,\ldots,f_p,g_p\in\mathscr{C}^2(\overline{\bn})$. Then for any $t\geq-1$, $[\BTt_{f_1},\BTt_{g_1}][\BTt_{f_2},\BTt_{g_2}]\ldots[\BTt_{f_p},\BTt_{g_p}]$ is in the trace class, and
\[
\lim_{t\to\infty}t^{p-n}\Tr\bigg([\BTt_{f_1},\BTt_{g_1}][\BTt_{f_2},\BTt_{g_2}]\ldots[\BTt_{f_p},\BTt_{g_p}]\bigg)=\frac{(-i)^p}{\pi^n}\int_{\bn}\prod_{j=1}^p\{f_j,g_j\}(z)\frac{\intd m(z)}{(1-|z|^2)^{n+1}}.
\]
\end{cor}

We can compare the formula above with the following Dixmier trace formula obtained by Engli\v{s}, Guo and Zhang in \cite{En-Gu-Zh:toeplitz}.
\[
Tr_\omega[\BTt_{f_1},\BTt_{g_1}]\ldots[\BTt_{f_n},\BTt_{g_n}]=\frac{1}{n!}\int_{\sn}\prod_{j=1}^n\{f_j,g_j\}(z)\frac{\intd \sigma(z)}{\sigma_{2n-1}}.
\]

Also recall the identity
\[
\sigma_t(f,g)=-H^{(t)*}_{\bar{f}}H^{(t)}_g.
\]
Thus taking $f_i=\bar{g}, g_i=g, i=1,\ldots,p$ in Proposition \ref{prop: many sigma_t} gives the following asymptotic formula for Schatten-norm of Hankel operators.
\begin{cor}
Suppose $p\geq n+1$ is an integer, and $g\in\mathscr{C}^2(\overline{\bn})$. Then
\[
\lim_{t\to\infty}t^{p-n}\|H^{(t)}_{g}\|_{\mathcal{S}^{2p}}^{2p}=\frac{1}{\pi^n}\int_{\bn}\bigg[|\bpartial g(z)|^2-|\bar{R}g(z)|^2\bigg]^p(1-|z|^2)^{p-n-1}\intd m(z).
\]
\end{cor}
\begin{rem}\label{rem: hankel schatten norms}
There are profound study of Schatten-class membership and Schatten norm formulas for Hankel operators. See \cite{Ar-Fi-Ja-Pe:hankel, Fa-Xi:hankel, Fe-Ro:hankel, Peller:hankel, Zhu:schattenHankel} for Schatten-class membership criteria of Hankel operators. For $q=2, 4, 6$, Janson, Upmeier and Wallst\'{e}n \cite{Ja-Up-Wa:schatten-hankel} gave the following identity on the Hardy space of the unit disk.
\[
\|H_\phi\|_{\mathcal{S}^q}^q=c_q\int_{\tori}\int_{\tori}\frac{|\psi(\zeta)-\psi(\tau)|^q}{|\zeta-\tau|^2}\intd\sigma(\zeta)\intd\sigma(\tau),
\]
where $c_q$ are constants, and $\psi=(I-P)\phi$. In fact, it was shown that such identities hold only for $q=2, 4, 6$. Recently, Xia \cite{Xia:schatten-hankel} extended this formula to the open unit ball $\bn$.
\end{rem}

Finally, Proposition \ref{prop: many sigma_t} gives the following asymptotic formula for the Connes-Chern character at $p>n$.
\begin{thm}\label{thm: connes chern}
Suppose $p\geq n+1$ is an integer and $f_0, f_1,\ldots, f_{2p-1}\in\mathscr{C}^2(\overline{\bn})$. Set $f_{2p}:=f_0$. Then 
\[
\begin{split}
&\lim_{t\to\infty}t^{p-n}\tau_t(f_0,f_1,\ldots,f_{2p-1})\\
=&\frac{n^p}{\pi^n}\int_{\bn}\bigg(\prod_{j=0}^{p-1}C_1(f_{2j},f_{2j+1})(z)-\prod_{j=0}^{p-1}C_1(f_{2j+1},f_{2j+2})(z)\bigg)\frac{\intd m(z)}{(1-|z|^2)^{n+1}}.
\end{split}
\]
\end{thm}

	\bibliographystyle{plain}
	\bibliography{referenceHH}

\end{document}